\documentclass[12pt]{amsart}
\usepackage[utf8]{inputenc}
\usepackage[left=2.5cm,right=2.5cm, top=1.5cm,bottom=2.3cm,bindingoffset=0cm, marginparwidth=58pt]{geometry}
\usepackage{ytableau}
\usepackage{hyperref}
\usepackage{marginnote}
\usepackage{wrapfig}
\usepackage{mathrsfs}
\usepackage{url}
\hypersetup{
    colorlinks=true,
    linkcolor=blue,
    filecolor=magenta,      
    urlcolor=cyan,
    pdftitle={Overleaf Example},
    pdfpagemode=FullScreen,
    }
\usepackage{xcolor}

\usepackage{amscd, amsmath, amsfonts, amssymb, amsthm, fancyhdr, epsfig}
\usepackage[all,cmtip]{xy}
\usepackage{tikz}
\usetikzlibrary{decorations.pathreplacing}

\usepackage{bbm}
\usepackage{xy}
\usepackage{graphicx}
\usepackage{mathtools}
\usepackage{relsize}
\usetikzlibrary{patterns}
\usepackage[T3,T1]{fontenc}
\DeclareSymbolFont{tipa}{T3}{cmr}{m}{n}
\DeclareMathAccent{\caap}{\mathalpha}{tipa}{16}

\newcommand{\g}{{\mathfrak g}}
\newcommand{\n}{{\mathfrak n}}
\newcommand{\h}{{\mathfrak h}}

\newcommand{\gl}{{\mathfrak {gl}}}
\newcommand{\Z}{{\mathbb Z}}

\newcommand{\C}{{\mathbb C}}

\newcommand{\cc}{\mathcal C}
\newcommand{\co}{\mathcal O}
\newcommand{\ck}{\mathbbm{k}}
\newcommand{\on}{\mathbbm{1}}

\newcommand{\Rep}{\operatorname{Rep}}
\newcommand{\Res}{\operatorname{Res}}
\newcommand{\Ker}{\operatorname{Ker}}

\renewcommand{\dim}{\operatorname{dim}}

\newcommand{\coev}{\operatorname{coev}}
\newcommand{\En}{\operatorname{End}}
\newcommand{\Hom}{\operatorname{Hom}}
\newcommand{\Ind}{\operatorname{Ind}}

\newcommand{\tens}[1]{%
  \mathbin{\mathop{\otimes}\displaylimits_{#1}}%
}

\newcommand{\ssl}{{\mathfrak {sl}}}

\newcommand{\cp}{\mathscr P}

\newcommand{\cw}{\mathsmaller >}
\newcommand{\ccw}{\mathsmaller <}
\newcommand{\krug}{\mathlarger \circ}

\renewcommand{\dim}{\operatorname{dim}}

\renewcommand{\l}{\lambda}
\newcommand{\bl}{{\boldsymbol\l}}
\newcommand{\br}{{\boldsymbol\rho}}

\newcommand{\bm}{{\boldsymbol\mu}}
\newcommand{\bn}{{\boldsymbol\nu}}
\newcommand{\ba}{{\boldsymbol\alpha}}
\newcommand{\bb}{{\boldsymbol\beta}}

\newcommand{\onode}{ \path (O) ++(\x:\radius) coordinate (A);
\draw[very thick,black, fill=white] (A) circle[radius=4pt] ++(\x:1em);}

\newcommand{\genode}{
\path (O) ++(\x:\radius) coordinate (A);
   \path (O) ++(\x:\radius) coordinate (B);
   \fill[black] (A) circle[radius=1pt];
    \draw[very thick, rotate=\x+130,color=black] (A) -- ++ (0.4,0) coordinate (A);
    \draw[very thick, rotate=\x+60,color=black] (B) -- ++ (0.4,0) coordinate (B);
    }
\newcommand{\lenode}{
\path (O) ++(\x:\radius) coordinate (A);
   \path (O) ++(\x:\radius) coordinate (B);
    \fill[black] (A) circle[radius=1pt];
    \draw[very thick, rotate=\x+300,color=black] (A) -- ++ (0.4,0) coordinate (A);
    \draw[very thick, rotate=\x+230,color=black] (B) -- ++ (0.4,0) coordinate (B);
    }
\newcommand{\dotnode}{
\path (O) ++(\x:\radius) coordinate (A);
 \fill[black] (A) circle[radius=1pt];
}
\newcommand{\xnode}{
\path (O) ++(\x:\radius) coordinate (A);
   \path (O) ++(\x:\radius) coordinate (B);
     \path (O) ++(\x:\radius) coordinate (C);
       \path (O) ++(\x:\radius) coordinate (D);
    \fill[black] (A) circle[radius=1pt];
    \draw[very thick, rotate=\x+135,color=black] (A) -- ++ (0.3,0) coordinate (A);
    \draw[very thick, rotate=\x+45,color=black] (B) -- ++ (0.3,0) coordinate (B);
    \draw[very thick, rotate=\x+315,color=black] (C) -- ++ (0.3,0) coordinate (C);
    \draw[very thick, rotate=\x+225,color=black] (D) -- ++ (0.3,0) coordinate (D);
    }

 \newcommand{\nodelabel}[1]{ \path (O) ++(\x:\radius) coordinate (A);
    \draw[black ] (A)  ++(\x:\y) node {#1};}

\def\blacksquare{\hbox{\vrule width 5pt height 5pt depth 0pt}}

\theoremstyle{plain}
\newtheorem{theorem}{Theorem}[section]

\newtheorem{lemma}[theorem]{Lemma}
\newtheorem{definition}[theorem]{Definition}
\newtheorem{question}[theorem]{Question}
\newtheorem{corollary}[theorem]{Corollary}

\theoremstyle{definition}
\newtheorem{remark}[theorem]{Remark}

\newtheorem{ex}[theorem]{Example}

\numberwithin{equation}{theorem}

\title[The highest weight theory for Representations of $GL(X)$ in $Ver_p$]{The highest weight theory for Representations of General Linear groups in the Verlinde categories in positive characteristic}

\author{Alexandra Utiralova}
\begin{document}

\begin{abstract}
Following the work of Venkatesh (\cite{V24}), we study further the categories of representations of the general linear groups $GL(X)$ in the Verlinde category $Ver_p$ in characteristic $p$.
The main question we answer is how to translate between highest weight labelings for different choices of the Borel subgroup $B(X)\subset GL(X)$. We do this by reducing the general case to the study of representations of the group $GL(X)$ for $X=L_m\oplus L_{n}$ using the method of odd reflections.

On the category of representations of $GL(L_m\oplus L_{n})$ we introduce the structure of a highest weight category, as well as a categorical action of $\widehat\ssl_p$ through translation functors. It allows us to understand projective and injective objects, BGG reciprocity,  duality and lowest weights for simple modules, and standard filtration multiplicities for projective objects.
\end{abstract}

\maketitle

\tableofcontents

\section{Introduction}
The famous reconstruction theorem by Deligne \cite{deltan} says that any symmetric tensor category $\cc$ admitting a fiber functor (i.e. a symmetric tensor functor) to some category $\mathcal D$ can be reconstructed as a category of representations of some affine group scheme in $\mathcal D$. Thus, it is interesting to study group schemes in categories that appear as targets of fiber functors.

Let $\ck$ be an algebraically closed field. In characteristic zero, Deligne (\cite{deltens}) proved that any $\ck$-linear symmetric tensor category (STC for short) $\cc$ of moderate growth (such that the length of $X^{\otimes n}$ grows at most exponentially with $n$ for any $X\in \cc$) admits a fiber functor to the category $sVec$ of finite dimensional super vector spaces. In characteristic$p$, however, the situation is more complicated and this result no longer holds verbatim. One of the simplest counterexamples is the Verlinde category $Ver_p$ defined as the semisimplification of the category of representations of $\Z/p\Z$ over the field $\ck$ of characteristic $p>0$. It is known (see for example \cite{BEO23}) that $Ver_p$ is incompressible, so it cannot be mapped to a smaller category like $sVec$. The theorem of Deligne was partially generalized for the characteristic$p$ case by Coulembier, Etingof and Ostrik in \cite{CEO22}, who proved that when char $\ck = p$, any $\ck$-linear \textit{Frobenius exact} STC $\cc$ of moderate growth admits a fiber functor to $Ver_p$. The additional condition of being Frobenius exact is necessary and is specific to positive characteristic. As a consequence of this result and the reconstruction theorem, we get that any such $\cc$ is tensor-equivalent to the category of representations of some affine group scheme in $Ver_p$, which motivates our study of representation theory in the Verlinde categories.

In his thesis \cite{Vthesis}, and later in \cite{V23}, Venkatesh laid some ground work in representation theory in the Verlinde categories. In \cite{V24}, he also proved the classification result about representations of the general linear group $G=GL(X)$ for $X\in Ver_p$. His classification runs roughly as follows. Let $X=X_1\oplus\ldots\oplus X_k$, where $X_i$'s are simple objects in $Ver_p$. Define the maximal torus $T=GL(X_1)\times\ldots\times GL(X_k)\subset G$ and the Borel subgroup $B$ with $Lie(B)=\bigoplus_{i\le j} X_i\otimes X_j^*\subset \gl(X)$. The category $\Rep_{Ver_p}(T)$ of representations of $T$ is equivalent to $Ver_p(GL_{n_1})\boxtimes \ldots \boxtimes Ver_p(GL_{n_k})$, where $n_i=\dim X_i$, and $Ver_p(GL_n)$ is the semisimplification of the category of tilting modules for $GL_n$. Let us call simple objects in $\Rep_{Ver_p}(T)$ \textbf{weights}. For any weight $\bl$ one can define a generalized Verma module $M(\bl)$ as $Dist(G)\tens{Dist(B)} \bl$, where $Dist(G)$ is the algebra of distributions at the identity (hyperalgebra) of $G$. Then $M(\bl)$, as a $Dist(G)$-module, has a unique simple quotient $L(\bl)$, with $L(\bl)\not\simeq L(\bm)$ if $\bl\not\simeq\bm$. Under some conditions on $\bl$, the action of $Dist(G)$ on $L(\bl)$ integrates to the action of $G$, and we get that the collection of all $L(\bl)$'s with $\bl$ integrable is in bijection with the set of isomorphism classes of simple objects in $\Rep_{Ver_p}(G)$.

In this paper we study in more detail the highest weight theory for representations of $GL(X)$. The main question we answer is how to translate between highest weight labels for simple modules for different choices of the Borel subgroup $B$. In the above notations, different orderings (i.e. permutations) of simple summands $X_1,\ldots, X_k$ provide different choices of the Borel subgroup. Some of the permutations of these summands, namely those that only permute isomorphic $X_i$'s, can be realized through the action of some subgroup of $G(\ck)$ (the $\ck$-points of $G$) which we call the classical Weyl group of $G$. The resulting Borel subgroups are then conjugate to each other, just like for classical reductive groups. We describe what happens in this case in Section \ref{s_conj_Borels}. However, for permutations that swap nonisomorphic $X_i$'s we get Borel subgroups which are not conjugate to each other. This situation is similar to that for quasi-reductive supergroups. We adapt Serganova's method of odd reflections to our case and reduce the question to studying what happens when we permute just two consequent nonisomorphic simple summands $X_i\simeq L_m$ and $X_{i+1}\simeq L_r$ with $m<r$, where $L_n$ denotes the unique $n$-dimensional simple object in $Ver_p$ (see Section \ref{s_odd_refl}). The whole question is then reduced to the study of lowest weights for representations of the group $GL(L_m\oplus L_r)$.

The remainder of the paper is dedicated to the study of the category $$\cc=\Rep_{Ver_p}(GL(L_m\oplus L_r)),$$ which turns out to be of interest on its own. The main idea for working with it is to make use of the subcategory of super vector spaces inside $Ver_p$ (it is generated by $L_1$ and $L_{p-1}$). One can then consider $L_r$ as an odd version of $L_n$ for $n= p-r$, since $L_r=L_n\otimes L_{p-1}$, and thus the group $GL(L_m\oplus L_r)$ can be thought of as a super group scheme $GL(L_m|L_n)$ (see Section \ref{sec_super}). It turns out that there are multiple similarities between the representation theory of $GL(L_m|L_n)$ and the representation theory of the supergroup $GL(m|n)$. Getting inspiration from works about $GL(m|n)$, especially \cite{B03} and \cite{BS12} (but also \cite{Z96}, \cite{S11}, \cite{GS11}), we use the highest weight structure and the action of translation functors to study $\cc$.

We first reprove in Theorem \ref{reps_of_GL(L_n)}
the result of Venkatesh (\cite{V24}) stating that the category of representations of $GL(L_n)$ is equivalent to the category $Ver_p(GL_n)$ (see Definition \ref{def_Ver_p(G)}). We then study the categories $\cc_n = Ver_p(GL_n)$ through the categorical action of $\widehat\ssl_p$ on them via translation functors. We prove in Theorem \ref{thm_sl_p_act} that $\cc_n$ categorifies the loop module $(\Lambda^n \C^p)^{loop}$ for $\widehat\ssl_p$ (where $\C^p$ is the tautological representation of $\ssl_p$) and that simple objects in $\cc_n$ correspond to the standard basis vectors. We then use this result to prove that $\cc=\Rep_{Ver_p}(GL(L_m|L_n))$ categorifies the tensor product
$$
(\Lambda^m \C^p)^{loop}\otimes ((\Lambda^n \C^p)^*)^{loop},
$$
see Section \ref{sec_catact}.

In Section \ref{s_hw_cat}, we prove that $\cc$ is a highest weight category in the sense of Cline, Parshall and Scott \cite{CPS88}, where the role of standard objects is played by the generalized Verma modules $M(\bl)$. We prove that, just like for $GL(m|n)$, projective objects are also injective, and irreducible standard objects are projective. We then prove an irreducibility criterion for standard objects, that is identical to that for Kac modules for $GL(m|n)$ (see Section \ref{s_irrKac}). In Section \ref{s_wd}, we introduce the language of weight diagrams to describe the action of translation functors on $\cc$. The weight diagrams we obtain are a \textbf{circular} version of the weight diagrams for $GL(m|n)$. Those were introduced by Brundan and Stroppel in \cite{BS08} and \cite{BS12} in a slightly more general setting. It is more common now to use a slightly modified version of these diagrams in the literature on $GL(m|n)$ (see for example \cite{GS11}), which we also do. In Section \ref{s_act_on_ipm} we describe the action of translation functors on simple and indecomposable projective modules (IPMs) using the language of weight diagrams. We deduce from this the computations of multiplicities in the standard filtration of IPMs (see Theorem \ref{thm_compos_series}). Finally, in Section \ref{s_lowest} we use this to compute the lowest weight in $L(\bl)$.

This paper has the following layout:
\begin{itemize}
    \item In Section \ref{s_not}, we provide the necessary definitions and preliminaries.

    \item In Section \ref{ver}, we give the definition and state some results about the Verlinde categories $Ver_p$.

    \item In Section \ref{s_ver(G)} we define the categories $Ver_p(G)$ for arbitrary reductive group $G$. We also give an explicit description of the level-rank duality, i.e. the equivalence between $Ver_p(PGL_n)$ and $Ver_p(PGL_{p-n})$.

    \item In Sections \ref{GL(L_n)} and \ref{section_simple_reps_of_GL(X)} we review the classification results of Venkatesh \cite{V24} of simple representations of $GL(X)$. More specifically, in Section \ref{GL(L_n)} we discuss representations of $GL(L_n)$ for simple object $L_n$. Whereas, in Section \ref{section_simple_reps_of_GL(X)} we look at $GL(X)$ for non-simple object $X$.

    \item In Section \ref{catact} we study the categorical $\widehat\ssl_p$-action on $\Rep_{Ver_p}(GL(L_n))$ through translation functors.

    \item In Section \ref{s_8} we introduce the problem of translating between different Borel subgroups, give the answer for conjugate Borels, and reduce the general situation to the study of representations of $GL(L_m\oplus L_r)$.

    \item Section \ref{s_9} is dedicated to the study of representations of $GL(L_m\oplus L_r)$ using the techniques of highest weight categories, translation functors and weight diagrams.
\end{itemize}

\subsection*{Acknowledgments}
We would like to thank Pavel Etingof for suggesting the project to us, as well as multiple useful discussions. We are especially grateful to Vera Serganova, who shared her knowledge of the techniques used in the study of representations of supergroups and asked all the right questions, which turned out to be the crucial inspiration for the proofs in this paper.

\section{Notations and preliminaries}\label{s_not}
\subsection{Notations and assumptions}\label{notations}
\begin{itemize}
    \item The base field $\ck$ is assumed to be algebraically closed.

    \item Starting from section \ref{ver} we will assume that char $\ck$ is a prime number $p$; throughtout most of the paper we will assume $p>2$.

    \item We will use the conventions and definitions from \cite{EGNO}, in particular of symmetric tensor categories\footnote{In particular, for us all tensor categories are rigid, locally finite, with $\En(\on)=\ck$. }, symmetric tensor functors,  Karoubi subcategories, and the Deligne tensor product of categories.

    \item We denote the monoidal unit in a monoidal category by $\on$.

    \item If $\cc$ is a symmetric tensor category and $X,Y\in \cc$, we denote by $\tau_{X,Y}$ the braiding isomorphism:
    $$
    \tau_{X,Y}:X\otimes Y\to Y\otimes X.
    $$

    \item $Gps$ denotes the category of groups; $Vec$ and $sVec$ denote the categories of \textit{finite dimensional} vector spaces and super vector spaces over $\ck$ respectively, and for a group (or a group scheme) $G$ we denote its category of \textit{finite dimensional} representations by $\Rep G$.

    \item When $G$ is a super group and $z\in G$, we denote by $\Rep(G,z)$ the category of super representations of $G$, on which $z$ acts via the \textit{parity map} (trivially on the even part and via multiplication by $-1$ on the odd part).
    
\end{itemize}

\subsection{Affine group schemes in symmetric tensor categories} 

Let $\cc$ be a symmetric tensor category.

\begin{definition}
    An \textbf{affine group scheme} $G$ in $\cc$ is an object in $$GpSch(\cc):=CommHopfAlg(\cc)^{op},$$ where $CommHopfAlg(\cc)$ is the category of commutative Hopf algebras in the ind-completion $\Ind(\cc)$ of $\cc$. The corresponding Hopf algebra is called the \textbf{algebra of functions} on $G$ and is denoted by $\co(G)$.

    For any $G\in GpSch(\cc)$ we can define the \textbf{functor of points}
    $$
    G: CommAlg(\cc) \to Gps,
    $$
    which sends a commutative algebra $A$ in $\Ind(\cc)$ to
    $$
    G(A)=\Hom_{CommAlg(\cc)}(\co(G),A).
    $$
\end{definition}

\begin{ex}\begin{enumerate}
    \item As $Vec$ is naturally a tensor subcategory (spanned by $\on$) in any symmetric tensor category $\cc$, one can consider any affine group scheme $G$ as an affine group scheme in $\cc$.

    \item An affine super group scheme is an affine group scheme in $sVec$.
\end{enumerate}
\end{ex}
\medskip
Sometimes we will need to consider affine group schemes of \textbf{finite type} in $\cc$. Similarly to the classical case, $G\in GpSch(\cc)$ is said to be of \textbf{finite type} if $\co(G)$ is \textit{finitely generated}, i.e. if there exists $X\in \cc$ (as opposed to $\Ind(\cc)$), for which $\co(G)$ is a quotient of the symmetric algebra $S(X)$. We denote the corresponding category by $GpSch(\cc)^{ft}$.

\medskip

For any $G\in GpSch(\cc)$ one can define the corresponding \textbf{Lie algebra} object in the pro-completion of $\cc$ as follows:
$$
\g = Lie(G) = (I/I^2)^*,
$$
where $I$ is the augmentation ideal in $\co(G)$.

\begin{remark}
 The definition of a Lie algebra object in a STC in characteristic $p$ is subtle and can be found for example in \cite{V23},  Section 5.1. It is easy to see that $Lie(G)$ is an \textit{operadic Lie algebra} object, i.e. an object with a skew-symmetric bracket, satisfying the Jacobi identity. However, operadic Lie algebras are not well-behaved, and one should impose additional conditions to ensure nice properties like Poincare-Birkhoff-Witt (PBW) theorem. By a \textbf{Lie algebra object} we will mean an operadic Lie algebra object for which PBW theorem holds. 

Etingof proved in \cite{E18}, Lemma 4.2, that any subobject of an associative algebra object closed under the commutator is a Lie algebra object. This implies that $Lie(G)$ is an actual Lie algebra object, as it is a Lie subalgebra of the algebra of distributions on $G$ constructed below.
\end{remark}

The Lie algebra object $\g = Lie(G)$ can also be described through the functor of points. Let us consider $\g$ as an affine scheme whose algebra of functions is $S(\g^*)$. For any $A\in CommAlg(\cc)$ put
$$
\g(A) = \Hom_{CommAlg(\cc)}(S(\g^*), A) = \Hom_\cc(\g^*, A) = \Hom_\cc(I/I^2,A),
$$
where $I$ is the augmentation ideal in $\co(G)$.

Thus similarly to the classical case, we deduce that
$$
\g(A) = \Ker(G(A[\epsilon])\to G(A)),
$$
where $\epsilon^2=0$.

 \begin{remark}
     Let $G$ be of finite type so that $\co(G)$ is a quotient of $S(X)$ for some $X\in \cc$. Then  $\g = Lie(G)$ is an object of $\cc$ (rather than of the pro-completion of $\cc$),  since $\g$ is a quotient of $X$.
 \end{remark}

Similarly to the classical case, one defines the \textbf{algebra of distributions} on $G$ as:
$$
Dist(G) = \varinjlim_k (\co(G)/I^k)^*.
$$
This algebra is filtered with
$$
Dist(G)^{k-1} = (\co(G)/I^k)^*
$$
and naturally acts on every $\co(G)$-comodule object.

The Lie algebra $\g$ is a subobject of $Dist(G)^1$ closed under the commutator.

\begin{definition}
    A $Dist(G)$-module $M\in \Ind(\cc)$ is called \textbf{integrable} if it admits a compatible $\co(G)$-comodule structure.
\end{definition}

\subsection{The fundamental group of a symmetric tensor category}

Let us consider the (dual of) internal hom bifunctor on $\cc$:
$$
h: \cc^{op}\times \cc\to \cc,
$$
$$
~~(X,Y)\mapsto Y\otimes X^*.
$$

In \cite{deltan} Deligne defined the Hopf algebra object in $\Ind(\cc)$ 
$$
H_\cc = \int^{X\in \cc} h(X,X) = \int^{X\in \cc} X\otimes X^*,
$$
where for any bifunctor $F:\cc^{op}\times \cc\to \mathcal D$ the \textbf{coend} of $F$, denoted
$$
\int^{X\in C} F(X,X),
$$
is the coequalizer of
$$
\prod_{X,Y\in \cc} \prod_{f:X\to Y} F(Y,X) \rightrightarrows \prod_{X\in \cc} F(X,X)
$$
in $\mathcal D$, with one arrow sending $F(Y,X)$ to $F(X,X)$ via $F(f\times id_X)$, and the second one sending $F(Y,X)$ to $F(Y,Y)$ via $F(id_Y\times f)$. 

\begin{remark}
    If $\cc$ is semisimple then 
    $$
    H_\cc = \bigoplus_{X\in sOb(\cc)} X\otimes X^*,
    $$
    where $sOb(\cc)$ is the set of isomorphism classes of simple objects in $\cc$.
\end{remark}

\begin{description}
    \item[$\bullet$ Multiplication] The multiplication on $H_\cc$ is inherited from the multiplication map on $\prod_{X\in\cc} X\otimes X^*$ which sends $(X\otimes X^*)\otimes (Y\otimes Y^*)$ to $(X\otimes Y)\otimes (X\otimes Y)^*$.

    \item[$\bullet$ Unit] The unit map on $H_\cc$ is inherited from the unit map
    $$
    \eta = \prod_{X\in \cc} coev_X: \on\to \prod_{X\in\cc} X\otimes X^*.
    $$

    \item[$\bullet$ Comultiplication] The comultiplication on $H_\cc$ is inherited from $\prod_{X\in \cc} X\otimes X^*$, on which it is defined via
    $$
    \Delta|_{X\otimes X^*} = id_X\otimes coev_X\otimes X: X\otimes X^*\to (X\otimes X^*)\otimes (X\otimes X^*).
    $$

    \item[$\bullet$ Counit] The counit map is inherited from the evaluation map $\prod_{X\in\cc} ev_{X}$ on \\$\prod_{X\in \cc} X\otimes X^*$.

    \item[$\bullet$ Antipode] The antipode is inherited from the antipode on $\prod_{X\in\cc} X\otimes X^*$
    $$
    S|_{X\otimes X^*} =\tau_{X,X^*}: X\otimes X^*\to X^*\otimes X.
    $$

\end{description}

\begin{definition}
    The \textbf{fundamental group} of a symmetric tensor category $\cc$ is the affine group scheme $\pi_\cc$ in $\cc$ with 
    $$
    \co(\pi_\cc) = H_\cc.
    $$
\end{definition}

\begin{remark}\label{fundgppoints}
    For any $A\in CommAlg(\cc)$ the $A$-points of $\pi_\cc$ are given by
    $$
    \pi_\cc(A)=\mathrm{Aut}^{\otimes}(I_A),
    $$
    where $I_A: \cc\to A$-$mod_\cc$ is the functor sending $X\in \cc$ to a free $A$-module generated by $X$, i.e. $A\otimes X$.

    In particular the $\ck$-points of $\pi_\cc$ are the tensor automorphisms of the identity functor $id_\cc$.
\end{remark}

Any object $X\in \cc$ is naturally an $H_\cc$-comodule (and thus a $\pi_\cc$-module). The coaction map
$$
X\to X\otimes H_\cc
$$
is dual to the natural map
$$
X^*\otimes X\to H_\cc.
$$

\begin{definition}
    For any $G\in GpSch(\cc)$ and any homomorphism
    $$
    \varepsilon: \pi_\cc\to G,
    $$
    define the category $\Rep_\cc(G,\varepsilon)$ to be the category of $G$-modules $M$ in $\cc$ (i.e. $\co(G)$-comodules), whose restriction to $\pi_\cc$ along $\varepsilon$ coincides with the natural action of $\pi_\cc$ on the object $M$.
\end{definition}

Let $F:\cc\to \mathcal D$ be a tensor functor. There is a natural group scheme homomorphism
$$
\varepsilon: \pi_{\mathcal D}\to F(\pi_\cc)
$$
corresponding to the Hopf algebra object map
$$
\varepsilon^*: F(H_\cc)\to H_{\mathcal D},
$$
which in turn is induced from the map
$$
F(X^*\otimes X)=F(X)^*\otimes F(X)\to H_{\mathcal D}.
$$

\begin{theorem}[\cite{deltan}, Theorem 8.17] \label{reconstr_thm}
    The functor $F$ induces an equivalence of categories
    $$
    \cc \to \Rep_{\mathcal D}(F(\pi_\cc),\varepsilon).
    $$
\end{theorem}

\begin{remark}
    For any $A\in CommAlg(\mathcal D)$ the $A$-points of $F(\pi_\cc)$ are given by 
    $$
    F(\pi_\cc)(A)=\mathrm{Aut}^\otimes(I_A\circ F),
    $$
    where $I_A:\mathcal D\to A$-$mod_{\mathcal D}$ is the functor from Remark \ref{fundgppoints}.

    Thus, the $\ck$-points are
    $$
    F(\pi_\cc)(\ck)=\mathrm{Aut}^\otimes(F).
    $$
\end{remark}

\begin{ex}
    \begin{enumerate}
        \item For $\cc=Vec$ the fundamental group $\pi_\cc$ is trivial, so we recover the standard Tannakian reconstruction theorem: If $\cc$ is a Tannakian category, i.e. a symmetric tensor category admitting a fiber functor
        $$
        F:\cc\to Vec,
        $$
        then $\cc$ is equivalent to the category $\Rep(G)$, where $G$ is an affine group scheme with
        $$
        G(A)=\mathrm{Aut}^\otimes(I_A\circ F)
        $$
        for any commutative $\ck$-algebra $A$.

        \item For $\cc = sVec$ the fundamental group is $\Z/2\Z$, which acts trivially on even vector spaces and via multiplication by $-1$ on odd vector spaces.

        If $\cc$ is a super Tannakian category, i.e. a symmetric tensor category admitting a super fiber functor:
        $$
        F:\cc\to sVec,
        $$
        then $G=F(\pi_\cc)$ is an affine super group scheme receiving a group scheme homomorphism
        $$
        \varepsilon: \Z/2\Z\to G.
        $$
        Let $z=\varepsilon(\overline 1)$, then $\cc$ is equivalent to $\Rep(G,z)$ (see the last notation in Section \ref{notations}).

        \item If $\cc=\Rep G$ for some affine (super) group scheme $G$ then $\pi_\cc = G$ and $H_\cc=\co(G)$ is an ind Hopf algebra object in $\Rep(G)$ via the conjugation action of $G$ on it.

        For any other affine (super) group scheme $K$ we get a bijection
        $$
        \{\text{Tensor functors } F:\Rep G\to \Rep K\} \leftrightarrow \{\text{Group homomorphisms } \varepsilon: K\to G\}.
        $$
        
    \end{enumerate}
\end{ex}

Recall that the Lie algebra object $\g = Lie(\pi_\cc)$ can also be described through the functor of points:
$$
\g(A) = \Ker(\pi_\cc(A[\epsilon])\to \pi_\cc(A)),
$$
where $\epsilon^2=0$.

\begin{theorem}\label{fundamental_Lie_alg}
    As an affine scheme in $\cc$, the Lie algebra object $\g=Lie(\pi_\cc)$ represents the functor of points
    $$
    A\mapsto Der^\otimes(I_A),
    $$
    where $Der^\otimes(I_A)$ are ``tensor derivations'' of $I_A$, that is, natural transformations $\beta: I_A\to I_A$ satisfying
    $$
    \beta_{X\otimes Y} = \beta_X\otimes I_A(id_{Y}) + I_A(id_X)\otimes \beta_Y.
    $$
\end{theorem}
\begin{proof}
    Any automorphism $\phi$ in $\pi_\cc(A[\epsilon])=\mathrm{Aut}^\otimes(I_{A[\epsilon]})$ is of the form 
    $$
    \phi_X = \alpha_X+\epsilon\beta_X, \text{ for } X\in \cc,
    $$
    where $\alpha, \beta\in \En(I_A)$ with $\alpha_X, \beta_X: A\otimes X\to A\otimes X$. The natural map $\pi_\cc(A[\epsilon])\to \pi_\cc(A)$ sends $\alpha+\epsilon\beta$ to $\alpha$.  Therefore, if $\phi\in \Ker(\pi_\cc(A[\epsilon])\to \pi_\cc(A))$ then $\phi_X=I_A(Id_X)+\epsilon\beta_X$.

     The condition that $\phi$ is a tensor automorphism implies that
    $$
    \phi_{X\otimes Y} = I_A(id_{X\otimes Y})+\epsilon \beta_{X\otimes Y} = I_A(id_X\otimes id_Y) + \epsilon(I_A(id_X)\otimes \beta_Y + \beta_X\otimes I_A(id_Y)) = \phi_X\otimes \phi_Y.
    $$

\end{proof}

\subsection{General linear groups in symmetric tensor categories}

Let $\cc$ be a symmetric tensor category linear over $\ck$. 

In view of the reconstruction theorem (Theorem \ref{reconstr_thm}), we are interested in studying group schemes $G$ in $\cc$ together with a homomorphism 
$$
\varepsilon: \pi_\cc\to G.
$$

The first natural example of such a group scheme is the \textbf{general linear 
group} $GL(X)$ defined for any object $X\in \cc$.

First, one can define the matrix algebra $M(X)$ on $X$ as $X\otimes X^*$ with multiplication map
$$
m = id_X\otimes ev_{X}\otimes id_{X^*},
$$
$$
m: (X\otimes X^*)\otimes (X\otimes X^*)\to X\otimes X^*,
$$
and the unit map (the ``identity matrix")
$$
\eta = \tau_{X^*, X}\circ \coev_X: \on\to X\otimes X^*.
$$

We can then define the Lie algebra object $\gl(X)$ as $M(X)=X\otimes X^*$ with the bracket given by the commutator map:
$$
c =  m \circ (1- \tau_{X\otimes X^*, X\otimes X^*} ),
$$
$$
c: \gl(X)\otimes \gl(X)\to \gl(X).
$$

\begin{definition}
Let $X\in \cc$, define $\co(GL(X))$ to be the commutative Hopf algebra object in $\Ind \cc$ obtained as a quotient of the symmetric algebra
$$
S (M(X)_1^*\oplus M(X)_2^*) = S(M(X)_1^*)\otimes S(M(X)_2^*)
$$
on the direct sum of two copies of $M(X)^*$,  by the  ideal generated by the images of the maps
$$
m^*_{12}-\eta^*: M(X)^*\to \on\oplus (M(X)_1^*\otimes M(X)_2^*)\to
$$
$$
\to \on \oplus S^2(M(X)_1^*\oplus M(X)_2^*)\hookrightarrow S(M(X)_1^*\oplus M(X)_2^*)
$$
and
$$
 m^*_{21}-\eta^*: M(X)^* \to \on\oplus (M(X)_2^*\otimes M(X)_1^*)\to
$$
$$
\to \on \oplus S^2(M(X)_1^*\oplus M(X)_2^*)\hookrightarrow S(M(X)_1^*\oplus M(X)_2^*).
$$
  
\end{definition}
\begin{remark}\label{intuitive_GL(X)}
   Intuitively speaking, $GL(X)$ is an affine scheme cut out in the affine space $M(X)_1\oplus M(X)_2$ by the equations 
$$
AB = BA= 1.\footnote{Venkatesh proved in \cite{V24} that we only need the relation $AB=1$ when $X$ is of finite length (see Remark on p.5).}
$$ 
\end{remark}

\begin{remark}
    The coalgebra structure on $\co(GL(X))$ is inherited from the coalgebra algebra structure on $S(M(X)_1^*\oplus M(X)_2^*)=S(M(X)^*_1)\otimes S(M(X)_2^*)$. The coalgebra structure on $S(M(X)^*)$ is, in turn, induced from the comultiplication $m^*$ on $M(X)^*$.

    The antipode is induced from the map
    $$
    \begin{pmatrix}
        0&1\\1&0
    \end{pmatrix}: M(X)_1\oplus M(X)_2\to M(X)_1\oplus M(X)_2,
    $$
    ``swapping $A$ and $B$'' from Remark \ref{intuitive_GL(X)}.
\end{remark}

\begin{remark}
    By definition, for any $X\in \cc$ the group scheme $GL(X)$ is of finite type.
\end{remark}

Just as in the classical case, one can prove the following:

\begin{theorem}
\begin{itemize}
    \item \label{points_of_GL(X)}  The group scheme $GL(X)$ represents the functor of points
    $$
    A\mapsto \mathrm{Aut}_{A-mod_\cc}(A\otimes X) = \mathrm{Aut}_{A-mod_\cc}(I_A(X)),
    $$
    where $A\in CommAlg(\cc)$.
    
    That is, $GL(X)(A)$ is the $A$-linear automorphisms of the ``free" $A$-module $A\otimes X$.

    \item The corresponding Lie algebra object $Lie(GL(X))$ is $\gl(X)$.
\end{itemize}

\end{theorem}
\begin{proof}
    For proof see for example \cite{V24}, Proposition 3.2 and Corollary 3.4.
\end{proof}

The group scheme $GL(X)$ is equipped with a natural homomorphism 
$$
\varepsilon: \pi_\cc\to GL(X).
$$
On the level of Hopf algebras we have:
$$
\varepsilon^*: \co(GL(X))\to H_\cc,
$$
$$
\varepsilon^*|_{M(X)_1}:M(X)_1 \xrightarrow{\sim} X\otimes X^*\to  H_\cc,
$$
$$
\varepsilon^*|_{M(X)_2}:M(X)_2 \xrightarrow{\tau_{X,X^*}} X^*\otimes X\to H_\cc.
$$

On the level of functor of points the homomorphism $\varepsilon(A)$ is the natural map
$$
\varepsilon(A): \pi_\cc(A)\to GL(X)(A),
$$
$$
\mathrm{Aut}^\otimes(I_A)\rightarrow \mathrm{Aut}_{A-mod_\cc}(I_A(X)),
$$
for any $A\in CommAlg(\cc)$.

\begin{theorem}\label{fund_gp_is_subgp_of_GL(X)}
    Suppose $\cc$ is tensor generated by an object $X\in \cc$ (and its dual), i.e. every object of $\cc$ is isomorphic to a subquotient in $X^{\otimes k} \otimes (X^*)^{\otimes l}$ for some $k,l\ge 0$. Then the homomorphism $\varepsilon$ is injective, i.e. $\pi_\cc$ is a subgroup scheme in $GL(X)$.
\end{theorem}
\begin{proof}
    If $X$ tensor generates $\cc$ then any tensor automorphism of the functor $I_A$ is uniquely determined by its value on $X$.
\end{proof}

\subsection{Semisimplification of tensor categories}
In this section we will give a brief description of the semisimplification construction. 
We refer the reader to \cite{EO18} for more details and proofs.  

Throughout this section let $\cc$ be a Karoubian rigid symmetric monoidal category, which admits an embedding into a symmetric tensor category.

\begin{definition}
    For $X,Y\in \cc$ a morphism $f:X\to Y$ is called \textbf{negligible} if for any morphism $g:Y\to X$ the trace $tr(f\circ g)=0$.

    Let $\mathcal N_\cc(X,Y)\subset \Hom_\cc(X,Y)$ denote the set of all negligible morphisms from $X$ to $Y$. 
\end{definition}

Consider the category $\overline \cc$ whose objects are the same as in $\cc$ and morphisms are defined as the quotient
$$
\Hom_{\overline \cc}(X,Y)=\Hom_\cc(X,Y)/\mathcal N(X,Y).
$$

\begin{definition}
    The category $\overline \cc$ is called the \textbf{semisimplification} of $\cc$.
\end{definition}

One can prove that $\overline\cc$ inherits the monoidal structure from $\cc$. Moreover:

\begin{theorem}[\cite{EO18}, Theorem 2.6]
    The category $\overline \cc$ is a semisimple symmetric tensor category, whose simple objects are the indecomposable objects of $\cc$ of nonzero dimension.
\end{theorem}

The natural functor $\mathbf{S}:\cc\to\overline \cc$ is monoidal (by construction) and is called the \textit{semisimplification functor}.

\section{Verlinde category \texorpdfstring{$Ver_p$}{Ver p}}\label{ver}

\subsection{Definition and properties}
Starting from this section let us assume that $\mathrm{char}~\ck = p>0$.

\begin{definition}
    The \textbf{Verlinde category} $Ver_p$ is the semisimplification of the category of representations of $\Z/p\Z$ over $\ck$.
\end{definition}

It turns out that this category plays an important role in the theory of symmetric tensor categories over fields of positive characteristics. Let us describe it in more detail.

It is not difficult to prove that indecomposable representations of $\Z/p\Z$ over $\ck$ whose dimensions are nonzero modulo $p$ are the representations on which the generator $\overline 1$ acts via a Jordan block of size at most $p-1$. Thus we get $p-1$ simple objects in $Ver_p$:
$$
L_1,L_2,\ldots,L_{p-1}
$$
of dimensions $1,2,\ldots , p-1$ respectively.

The fusion rules are given by the modified Clebsch-Gordan rule:
$$
L_i\otimes L_j = \bigoplus_{k=1}^{
\min(i,j , p-i, p-j)} L_{|i-j|+2k-1}.
$$

By looking closely at the formula above as well as at the explicitly defined tensor product on $\Rep(\Z/p\Z)$, one can deduce the following:
\begin{itemize}
    \item The $1$-dimensional object $L_1$ is the monoidal unit $\on$.

    \item For $p>2$, the $(-1)$-dimensional object $L_{p-1}$ is invertible and the subcategory spanned by $L_1$ and $L_{p-1}$ is a tensor subcategory of $Ver_p$ equivalent to $sVec$.

    \item $L_i\otimes L_{p-1}=L_{p-i}$.

    \item Each $L_i$ is self-dual:
    $$
    L_i=L_i^*.
    $$

    \item $Ver_p$ is tensor generated by $L_2$ with $L_n=S^{n-1}L_2$.

    \item The subcategory spanned by objects $L_{2i+1}$ of odd dimension is closed under tensor products. We denote this tensor subcategory $Ver_p^+$.

    \item $L_3$ tensor generates $Ver_p^+$.

\end{itemize}

Ostrik proved in \cite{O15} that 
\begin{theorem}[\cite{O15}, Proposition 3.3]
  If char $\ck>2$, we have  $Ver_p = sVec\boxtimes Ver_p^+.$
\end{theorem}

\begin{corollary}
   The fundamental group $\pi_{Ver_p}$ of $Ver_p$ has a subgroup $\Z/2\Z$, such that the generator acts by $1$ on all simple objects in $Ver_p^+$ and by $-1$ on all simple objects not in $Ver_p^+$. 

   Moreover, we get a decomposition
   $$
   \pi_{Ver_p} = \Z/2\Z\times \pi_{Ver_p^+}.
   $$
\end{corollary}

\begin{ex}
    \begin{enumerate}
        \item We have equivalences of categories $Ver_2\simeq Vec$, $Ver_3\simeq sVec$.

        \item $Ver_5^+$ has two simple objects: $L_1$ and $L_3$ with
        $$
        L_3\otimes L_3  = L_1\oplus L_3.
        $$
    \end{enumerate}
\end{ex}

It is known that $Ver_p$ is \textit{incompressible}, i.e any symmetric tensor functor $Ver_p\to \cc$ to a symmetric tensor category $\cc$ is fully faithful (see for instance \cite{BEO23}).

In \cite{CEO22} the authors proved that $Ver_p$ is the target of the fiber functor from a large class of moderate growth symmetric tensor categories in characteristic $p$. More specifically:
\begin{theorem}[\cite{CEO22}, Theorem 1.1]\label{frobexact}
    A symmetric tensor category $\cc$ linear over a field $\ck$ of characteristic $p$ admits a fiber functor
    $$
    F: \cc\to Ver_p
    $$
    if and only if it has moderate growth and is \textit{Frobenius exact}\footnote{See Definition 3.5 in \cite{CEO22}.}. Moreover such a functor $F$ is unique up to isomorphism when it exists.
\end{theorem}

We will not use the definition of Frobenius exactness in this paper but it should be noted that semisimple categories are Frobenius exact and thus admit a fiber functor to $Ver_p$.

\begin{remark}
    Theorem \ref{frobexact} is the analog of the famous theorem by Deligne, which asserts that any symmetric tensor category $\cc$ of moderate growth over a field of characteristic zero admits a super fiber functor $F:\cc\to sVec$.

    Note that in characteristic $p$ this breaks as $Ver_p$ admits no tensor functors to $sVec$.
 \end{remark}

 An immediate corollary of Theorem \ref{frobexact} and the reconstruction theorem is that any Frobenius exact symmetric tensor category of moderate growth is equivalent to the category $\Rep_{Ver_p}(G,\varepsilon)$ for some affine groups scheme $G$ in $Ver_p$ and a homomorphism
 $$
 \varepsilon: \pi_{Ver_p}\to G.
 $$

 The rest of this paper will be dedicated to the study of representations of groups schemes in $Ver_p$.

 \subsection{Harish-Chandra pairs for group schemes in \texorpdfstring{$Ver_p$}{Ver p}} \label{section_HC}
 In \cite{V23} Venkatesh proved a remarkable result about affine group schemes in $Ver_p$, which showed that they behave in some ways similarly to super group schemes,\footnote{Recall that affine group schemes in $Ver_p$ are a generalization of affine super group schemes as $sVec$ is a tensor subcategory in $Ver_p$.} and one could apply similar methods to study them.

 Let $G$ be an affine group scheme in $Ver_p$. Define $J\subset \co(G)$ to be an ideal generated by all simple subobjects in $\co(G)$ not isomorphic to $\on=L_1$. It is easy to see that $J$ is a Hopf ideal of $\co(G)$. 

 \begin{definition}
     The \textbf{classical subgroup} scheme $G_{0}$ is cut out in $G$ by the ideal $J$. That is
     $$
     \co(G_{0})=\co(G)/J.
     $$
 \end{definition}
 
 \begin{remark}
     The classical subgroup scheme $G_{0}$ in $G$ is an ordinary affine group scheme in $Vec$.

     When $G$ is of finite type then $G_{0}$ is of finite type.
 \end{remark}

 \begin{ex}
     \begin{enumerate}
         \item Let $G$ be an affine group scheme in $Ver_p$ and let $A$ be a commutative $\ck$-algebra in $Vec$ (that is, $A$ considered as an algebra object in $Ver_p$ has no simple subobjects, not isomorphic to $\on$). Then
         $$
         G_{0}(A) = G(A),
         $$
         in particular
         $$
         G_{0}(\ck)=G(\ck).
         $$
         \item If $X\in Ver_p$ with $[X:L_n]=k_n$  for each $n=1,\ldots, p-1$ then
         $$
         (GL(X))_{0} = \prod_{n=1}^{p-1} GL_{k_n} = \mathrm{Aut}_{Ver_p}(X).
         $$
     \end{enumerate}
     \end{ex}

For any object $X\in Ver_p$ define its classical part $X_{0}$ as the subobject of $X$ spanned by all copies of $\on$ in it:
$$
X_{0}=\Hom_{Ver_p}(\on, X)\otimes \on.
$$

 If $\g$ is a Lie algebra object in $Ver_p$, its classical part $\g_{0}$ is a Lie subalgebra of $\g$.

 \begin{definition}
     A \textbf{Harish-Chandra pair} in $Ver_p$ is a pair $(G_0, \g)$, where $G_0$ is a (classical) affine group scheme of finite type in $Vec$, $\g$ is a Lie algebra object in $Ver_p$ with an $G_0$-module structure, such that the bracket on $\g$ is a $G_0$-module map and $\g_0$ is a submodule of $\g$. Together with an isomorphism
     $$
     \g_{0}\simeq Lie(G_0)
     $$
     of $G_0$-modules and a compatibility condition: the action of $\g_0 \simeq Lie(G_0)$ on $\g$ induced from the $G_0$-module structure on it must coincide with the adjoint action  of $\g_0$.

     Let us denote the category of Harish-Chandra pairs in $Ver_p$ by $\mathcal{HC}(Ver_p)$.     
 \end{definition}

 \begin{theorem}[ \cite{V23}, Theorem 1.2]\label{thm_HC}
 The natural functor 
 $$
 \mathbf{HC}: GpSch(Ver_p)^{ft} \to \mathcal{HC}(Ver_p),
 $$
 sending an affine group scheme $G$ of finite type in $Ver_p$ to the Harish-Chandra pair $(G_0,\g)=(G_{0}, Lie(G))$ is an equivalence of categories. 
 
 Functor $\mathbf{HC}$ is called the Harish-Chandra functor.  
 \end{theorem}

 \begin{ex}
     For $X\in Ver_p$ with $[X:L_n]=k_n$ for each $n=1,\ldots, p-1$, we have
     $$
    \mathbf{HC}(GL(X)) = (\prod_{n=1}^{p-1} GL_{k_n}, \gl(X)) = \left(\prod_{n=1}^{p-1} GL(W_n), \bigoplus_{1\le a,b \le p-1} (W_a\otimes W_b^*)\otimes (L_a\otimes L_b^*)\right),
     $$
     where $W_n=\Hom_{Ver_p}(L_n, X) = \ck^{\oplus k_n}$.
 \end{ex}

 \begin{corollary}[\cite{V23}, Corollary 1.3]
     Let $G$ be an affine group scheme of finite type in $Ver_p$ and let $(G_0, \g)$ be the corresponding
Harish-Chandra pair in $Ver_p$. The Harish-Chandra functor establishes a bijection between the set of
closed subgroup schemes of $G$ and the set
$$
\{(H_0, \mathfrak h)~|~ H_0 \text{ a closed subgroup scheme of } G_0, \mathfrak h \text{ a Lie subalgebra object of } \g; \text{ with } Lie(H_0) = \mathfrak h_{0}\}.
$$
 \end{corollary}

 \begin{lemma}[\cite{V24}, Section 7.1]\label{lemma_gps_with_no_even_part}
     Let $G$ be an affine group scheme of finite type in $Ver_p$ with $G_{0}=1$ and let $\g = Lie(G)$. Then $\co(G)$ has finite length in $Ver_p$ and
     $$
     \co(G) = S(\g^*) = U(\g)^* = Dist(G)^*.
     $$
 \end{lemma}

 \begin{ex}\label{ex_SL(X)}
 Suppose $X\in Ver_p$ and $[X:L_n]=k_n$ for each $n=1,\ldots, p-1$.

     Let $H_0$ be the subgroup of $GL(X)_{0}=\prod_{k=1}^{p-1} GL_{k_n}$ of tuples $(A_1,\ldots, A_{p-1})$ of matrices with the property 
     $$
     \det A_1\cdot\ldots\cdot \det A_{p-1} = 1.
     $$

     Let $\h=\ssl(X)$ be the Lie subalgebra in $\gl(X)$ defined as the kernel of the evaluation map $ev_X: \gl(X)\to \on$.

     We can define the subgroup scheme $SL(X)$ in $GL(X)$ by putting
     $$
     \mathbf{HC}(SL(X)) = (H_0,\h).
     $$
 \end{ex}

 \begin{remark}\label{rem_SL}
     When $X=L_n$ is simple, we get that $SL(L_n)_{0}$ is trivial. We can then use Lemma \ref{lemma_gps_with_no_even_part} to see that $\co(SL(L_n))=S(\ssl(L_n)^*)=U(\ssl(L_n))^*$.
     
     \end{remark}

 \begin{definition}
     A representation of a Harish-Chandra pair $(G_0, \g)$ in $Ver_p$ is a $G_0$-module $X$ in $Ver_p$ with a Lie algebra action map
     $$
     a:\g\otimes X\to X,
     $$
     which is $G_0$-equivariant. Such that the restriction of $a$ to $\g_{0}$ coincides with the $Lie(G_0)$-action on $X$ after the identification $\g_{0}\simeq Lie(G_0)$.
 \end{definition}

 \begin{corollary}[\cite{V23}, Corollary 1.4.]
     Let $G$ be an affine group scheme of finite type in $Ver_p$ and let $(G_0, \g)$ be the corresponding
Harish-Chandra pair in $Ver_p$. The category of $G$-modules (i.e. $\co(G)$-comodules) in $Ver_p$ is equivalent to the category of representations of $(G_0,\g)$ in $Ver_p$.
 \end{corollary}

\subsection{Triangular decomposition for \texorpdfstring{$GL(X)$ in $Ver_p$}{GL(X) in Ver p}}

Let $X\in Ver_p$ with $[X:L_n] = k_n$ for $n=1, \ldots, p-1$. Let $k=\sum k_n$ and let us choose a decomposition of $X$ into the direct sum of irreducible components
$$
X=X_1\oplus X_2\oplus\ldots\oplus X_k,
$$
where each $X_i$ is simple and if for some $1\le i\le k-1$ we have $X_i\simeq L_a$, $X_{i+1}\simeq L_b$ then $a\le b$.

Denote $G=GL(X)$. Recall that $G_{0}= \prod_{n=1}^{p-1} GL_{k_n}$ and 
$$
\g:=Lie(G)=\gl(X)=X\otimes X^* = \bigoplus_{1\le i,j,\le k} X_i\otimes X_j^*.
$$

For each $k\in\mathbb N$ denote by $B_{k}$ the subgroup of upper-triangular matrices in $GL_{k}$ (i.e. the standard Borel subgroup); and let $T_{k}\simeq \mathbb G_m^k$ denote the subgroup of diagonal matrices in $GL_{k}$ (i.e. the standard maximal torus).

\begin{definition}
    Define the Lie subalgebra $\mathfrak t\subset \g$ as
    $$
    \mathfrak t = \bigoplus_{i=1}^k \gl(X_i) = \bigoplus_{i=1}^k X_i\otimes X_i^*.
    $$
    
    The \textbf{maximal torus} $T=T(X)$ is defined as the subgroup scheme of $G=GL(X)$ corresponding to the Harish-Chandra pair $(\prod_{k=1}^{p-1} T_{k_n}, \mathfrak t)$.

    In other words, $T = \prod_{i=1}^k GL(X_i)$.
\end{definition}
\begin{remark}
    We will refer to $T$ as a torus in $G$ although it is not an actual torus. In fact, it rather resembles a more general Levi subgroup, as each $X_i$ is not necessarily $1$-dimensional.

    However, $T$ will play the role of a maximal torus in our theory. Moreover, we note that $T_{0}=\mathbb G_m^k$ is an actual torus.
\end{remark}

\begin{remark}
    By construction, the natural homomorphism
    $$
    \varepsilon: \pi_{Ver_p} \to GL(X)
    $$
    factors through $T(X)$. We will therefore use the same letter
    $$
    \varepsilon: \pi_{Ver_p}\to T(X)
    $$
    for the natural homomorphism to $T(X)$.
\end{remark}

\begin{definition}
    Define the Lie subalgebra $\mathfrak b\subset \g$ as
    $$
    \mathfrak b = \bigoplus_{1\le i\le j\le k} X_i\otimes X_j^*,
    $$
    the ``upper triangular" subalgebra of $\gl(X)$.

    The (standard) \textbf{Borel subgroup} $B=B(X)$ of $G=GL(X)$ is defined as the subgroup scheme corresponding to the Harish-Chandra pair $(\prod_{n=1}^{p-1} B_{k_n}, \mathfrak b).$
\end{definition}

\begin{remark}
    Note that our choice of Borel subgroup $B$ in $G$ really depends on the ordering of the simple summands in $X$ that we previously chose.
\end{remark}

We can similarly define the subgroups $N^\pm = N^\pm(X)\subset G$ with
$$
\mathbf{HC}(N^+)=(\prod_{n=1}^{p-1} N^+_{k_n}, \bigoplus_{1\le i<j\le k} X_i\otimes X_j^*),
$$
$$
\mathbf{HC}(N^-)=(\prod_{n=1}^{p-1} N^-_{k_n}, \bigoplus_{1\le j<i\le k} X_i\otimes X_j^*),
$$
where $N^+_{n}$ (resp. $N^-_{n}$) denotes the subgroup of strictly upper-triangular (resp. strictly lower-triangular) matrices in $GL_{n}$.

Put $\n^\pm = Lie(N^\pm)$. We get a decomposition
$$
\g = \n^-\oplus(\mathfrak t \oplus \n^+) = \n^-\oplus \mathfrak b.
$$

Moreover, in \cite{V24}, Venkatesh proves the following:
\begin{theorem}[\cite{V24}, Proposition 5.1] \label{thm_PBW}
    In $\Ind(Ver_p)$ we have the isomorphism
    $$
    Dist(G) \simeq Dist(N^-)\otimes Dist(B)
    $$
    of right $Dist(B)$-module algebras.
\end{theorem}

The isomorphism in Theorem \ref{thm_PBW} is an isomorphism of left $Dist(N^-)$-module algebras for similar reasons and we have
$$
Dist(G)\simeq Dist(N^-)\otimes Dist(T)\otimes Dist(N^+).
$$

\section{Categories \texorpdfstring{$Ver_p(G)$}{Ver p(G)}} \label{s_ver(G)}
In this section we assume char $\ck >2$.

\subsection{Definition and properties}

Let $G$ be a connected reductive algebraic group over $\ck$ with the maximal torus $T$ and the Borel subgroup $B\supset T$ corresponding to some choice of positive roots for $G$. Let $R^+$ denote the set of positive roots and $P^+$ denote the set of dominant integral weights, let $\rho$ denote the half-sum of all positive roots, let $h=\max \{\langle \rho, \beta^{\vee}\rangle~|~ \beta\in R^+\}$ be the Coxeter number and let 
$$
C_p = \{\lambda\in P^+~|~\langle \lambda+\rho, \beta^{\vee} \rangle < p~\text{for all } \beta\in R^+\}
$$
be the (integral part of) the fundamental alcove for the affine Weyl group. Note that $C_p\neq\emptyset$ if and only if $h\le p$.

\begin{definition} \label{def_Ver_p(G)} The Verlinde category for $G$ is defined as the semisimplification of the category of tilting modules for $G$:
$$
Ver_p(G)=\overline{\mathcal{T}(G)}.
$$
\end{definition}

Fusion rules and classification of simple objects for $Ver_p(G)$ were determined in \cite{GM}. 

Recall that for each $\lambda\in P^+$ there exists a unique indecomposable tilting module $T_\lambda$ with the highest weight $\lambda$. It was proved in \cite{GM} that $\dim T_\lambda$ is divisible by $p$ if and only if $\lambda\not\in C_p$. Simultaneously, for each $\lambda\in C_p$ we have the isomorphism between $T_\lambda$, the Weyl module $V_\lambda$, and the simple module $L_\lambda$. For this reason we will use the notation $V_\lambda$ (slightly abusing it) for the corresponding simple object in $Ver_p(G)$. Combining this we get the following description for the set of simple objects in $Ver_p(G)$:
$$
sOb(Ver_p(G))=\{V_\lambda~|~\lambda\in C_p\}.
$$

\begin{ex}\label{ex_fund_alc}
    \begin{enumerate}
        \item When $G=SL_2$, we have $P^+=\Z_{\ge 0}$, 
        $$
        C_p=\{k\in \Z_{\ge 0}~|~k+1<p\}.
        $$
        We have the equivalence of tensor categories $Ver_p(SL_2)\simeq Ver_p$ (see \cite{O15}) that sends $V_k$ to $L_{k+1}$ (as $\dim V_k=k+1$).

        \item When $G=GL_n$ with $n<p$ we get:
        $$
        P^+=\{\lambda=(\lambda_1,\ldots, \lambda_n)\in \Z^n~|~\lambda_1\ge \lambda_2\ge\ldots\ge \lambda_n\},
        $$
        $$
        C_p=\{\lambda\in P^+~|~\lambda_1-\lambda_n\le p-n\}.
        $$

        When $\lambda=(1,0,\ldots, 0)$ we denote $V_{\lambda}$ simply by $V$ when there is no ambiguity. The object $V$ has dimension $n$ and tensor generates $Ver_p(GL_n)$.

        Note that in $\Rep(GL_n)$ for each $k< p$ the symmetric power $S^kV$ is a direct summand in $V^{\otimes k}$, and therefore is a tilting module. When $k=p-n+1$ we have
        $$
        \dim S^kV = \binom{p}{n-1} = 0\in \ck.
        $$
        Therefore, after taking semisimplification we get
        \begin{equation}\label{sym_powers_are_zero}
            S^{p-n+1}V=0\in Ver_p(GL_n).
        \end{equation}

        The decomposition of the tensor product of $V_\lambda$ with $V$  stays the same as in the classical case, except for the case when $\lambda_1-\lambda_n=p-n$. We have
        \begin{equation}
              V_\lambda\otimes V = \bigoplus_{\lambda+e_i\in C_p} V_{\lambda+e_i},
        \end{equation}
          where $e_1,\ldots,e_n$ is the standard basis in $\Z^n$. Note that when $\lambda\in C_p$, we have $\lambda+e_i\in C_p$ if and only if $i>1$ and
        $$
        \lambda_{i-1}>\lambda_{i},
        $$
        or $i=1$ and
        $$
        \lambda_1-\lambda_n<p-n.
        $$
        We will explore these conditions further in Section \ref{catact}.
    \end{enumerate}
\end{ex}

The (unique) fiber functor 
$$
F: Ver_p(G)\to Ver_p
$$
can be easily understood through the equivalence $Ver_p\simeq Ver_p(SL_2)$. The restriction functor
$$
\Res: \mathcal T(G)\to \mathcal T(SL_2)
$$
to a \textit{principal} $SL_2$-subgroup in $G$ is well-defined as it sends tilting modules to tilting modules. Moreover, it sends negligible morphisms to negligible morphisms (see \cite{EO18}), and thus descends to a tensor functor between the semisimplified categories.

\begin{ex}
    When $G=GL_n$ we have $F(V)=L_n$.
\end{ex}

\subsection{Tensor subcategories}\label{sect_tensor_subcat}
Let $Z$ denote the center of $G$ and let $G_{ad}$ be the corresponding adjoint group $G/Z$. Then $Ver_p(G_{ad})$ can naturally be identified with a tensor subcategory in $Ver_p(G)$ (coming from tilting modules on which the center acts trivially), which we denote by $Ver_p^+(G)$.

So, for $n< p$ the category $Ver_p(PGL_n)$ is  a subcategory in both $Ver_p(SL_n)$ and $Ver_p(GL_n)$. In particular for $SL_2$ we have
$$
Ver_p(PGL_2)=Ver_p^+(SL_2)=Ver_p^+.
$$
The category $Ver_p(GL_n)$ is naturally $\Z$-graded with degree of $V_{\lambda}$ given by $$|\lambda|=\lambda_1+\ldots+\lambda_n.$$ The subcategory $Ver_p(PGL_n)$ is identified with the degree zero part. 

It can be proved that there is a decomposition 
$$
Ver_p(G)\simeq \Rep(Z,z)\boxtimes Ver_p^+(G),
$$
where $z\in Z$ is the image of $-1$ under the principal map $SL_2\to G$ (see \cite{CEO22}). 

Let us concentrate on type $A$ from now on. We get a decomposition
\begin{equation}\label{decomp}
 Ver_p(GL_n)=\mathcal{D}\boxtimes Ver_p(PGL_n),   
\end{equation}
where $\mathcal{D}=\Rep(\mathbb G_m,z)$ is a pointed tensor category freely generated by an invertible object $\psi$ of infinite order. The dimension of $\psi$ is $(-1)^{n-1}$, so when $n$ is odd, the category $\mathcal D$ is equivalent to the category $Vec_\Z$ of $\Z$-graded vector spaces; whereas when $n$ is even, $\mathcal D$ is equivalent to the subcategory of $sVec_\Z$ generated by the $(0|1)$-dimensional super vector space, sitting in degree $1$. 
Let us describe the decomposition \ref{decomp} in more detail below.

There is an invertible object in $Ver_p(GL_n)$ coming from the determinant representation of $GL_n$. Put $det = \Lambda^nV = V_{\nu}$ with $\nu = (1,\ldots, 1)$.

Now put $\lambda = (p-n, 0,\ldots, 0)$ and note that $\mu =(p-n,p-n,0,\ldots,0)$ is the only highest weight in $C_p$ for which $T_\mu$ appears as a summand in $T_\l\otimes T_\l$ in $\mathcal{T}(GL_n)$. Similarly, the only weight in $C_p$, that appears in the decomposition of $(T_\l)^{\otimes k}$ for $k\le n$ is $$(p-n,p-n,\ldots, p-n,0,\ldots, 0)$$ ($(p-n)$ repeated $k$ times). Thus, the object $\chi=V_{\lambda}=S^{p-n}V$ is invertible with 
\begin{equation}\label{inveq}
   \chi^n = det^{p-n}. 
\end{equation}

We construct the invertible object $\psi$ of degree one in the form $det^a\otimes \chi^b$, where $an+b(p-n)=1$ (it is possible since $p-n$ and $n$ are coprime). Note also that because of equation $\ref{inveq}$, we can always assume that $0\le b<n$, so $\psi$ is uniquely defined.

\begin{ex}
        Let $p=7, n=3$ then $\psi=det^{-1}\otimes \chi = V_{\pi}$, where $\pi=(3,-1,-1)$. We have $det=\psi^3, \phi=\psi^4$.
\end{ex}

Let $\mathcal D_0 =\Rep(\mu_n,z)$ be the quotient of $\mathcal D$ corresponding to the subgroup  $\mu_n\subset \mathbb G_m$ and  sending $det=\psi^n$ to $\on$ (so $\mathcal D_0$ has $n$ simple objects). We get the decomposition
$$
Ver_p(SL_n) = \mathcal D_0\boxtimes Ver_p(PGL_n)= \mathcal D_0\boxtimes Ver_p^+(SL_n).
$$

\begin{ex}
    For $n=2$ this decomposition yields 
    $$
    Ver_p=sVec\boxtimes Ver_p^+.
    $$
\end{ex}

\begin{corollary}
    We get a decomposition
    $$
    \pi_{Ver_p(SL_n)}=\mu_n\times \pi_{Ver_p^+(SL_n)}.
    $$
\end{corollary}

We conclude this section by describing
the \textbf{level-rank duality} for the Verlinde categories in positive characteristic.

There is a well-known equivalence of categories (see for example \cite{CEO22}):
\begin{equation}\label{level_rank}
\mathbf{D}: Ver_p(PGL_n) \xrightarrow{\sim} Ver_p(PGL_{p-n}).   
\end{equation}

This equivalence will also become apparent in Section \ref{GL(L_n)}.

Let us give some explicit description of this equivalence on the level of simple objects. For any $GL_n$-weight $\l\in C_p$ of degree $0$ (i.e. a $PGL_n$-weight) define $D(\l)$ to be the $PGL_{p-n}$-weight, for which $V_{D(\l)}=\mathbf{D}(V_\l)$.

First note that if $V_\l$ corresponds to a polynomial representation of $GL_n$ (i.e. all $\l_i\ge 0$) with $\l_n=0$ then we can assign a Young diagram $Y(\l)$ to $\l$, and the condition $\l\in C_p$ translates to saying that $\l_1\le p-n$, or equivalently, that $Y(\l)$ fits into the $(n-1)$-by-$(p-n)$ rectangle. Simultaneously any polynomial weight $\l$, such that $Y(\l)$ fits into the $n$-by-$(p-n)$ rectangle, is in $C_p$.

For a polynomial $GL_n$-weight $\l$ (or a partition $\l$ of some positive integer $N$) let $\l^t$ denote the \textit{transpose} of $\l$ (that is $Y(\l^t)$ is obtained by reflecting $Y(\l)$ around the $x=y$ line on he plane). Note that if $\l_1\le p-n$,  partition  $\l^t$ has at most $p-n$ parts, and thus can be thought of as a weight for $GL_{p-n}$. Moreover, in this case $\l^t\in C_p$, since $Y(\l^t)$ fits into the $(p-n)$-by-$n$ rectangle.

Roughly speaking, the correspondence $\l\mapsto D(\l)$ is transposition, but one needs to be careful defining it. 

Let $\l$ be a $GL_n$-weight. Let $\l_1,\ldots,\l_k$ be the \textit{positive} parts, and $\l_{n-l+1},\ldots, \l_n$ be the \textit{negative} parts of $\l$. Put
$$
\alpha=(\alpha_1,\ldots, \alpha_k)=(\l_1,\ldots,\l_k),
$$
$$
\beta=(\beta_1,\ldots,\beta_l)=(-\l_n, \ldots, -\l_{n-l+1}).
$$
This defines a bijection between $GL_n$-weights of degree zero and pairs of partitions $(\alpha, \beta)$ with length($\alpha) + $ length$(\beta)\le n$, and $\sum_i \alpha_i = \sum_j \beta_j$. We have $\l\in C_p$ if and only if $\alpha_1+\beta_1\le p-n$.

Now for any $GL_n$-weight $\l\in C_p$ of degree zero define $D(\l)$ to be the $GL_{p-n}$-weight corresponding to the pair $(\alpha^t,\beta^t)$. Note that since length$(\alpha^t) + $ length$(\beta^t)$  $= \alpha_1+\beta_1\le p-n$,  this weight is well-defined. Similarly, since $\alpha^t_1+\beta^t_1 = $ length$(\alpha) + $ length$(\beta)$  $\le n=p-(p-n)$, we have $D(\l)\in C_p$. 

Note that the correspondence $\l\to D(\l)$ is a bijection on degree zero weights in $C_p$ for $GL_n$ and for $GL_{p-n}$.

\begin{remark}
    This construction can also be obtained as follows:
    \begin{enumerate}
        \item Take a weight $\l\in C_p$ for $GL_n$ of degree zero.
        \item Turn $V_\l$ into a polynomial representation of degree $-n\cdot\l_n$ by tensoring it with $det^{-\l_n}$. Call the corresponding polynomial weight $\mu$.
        \item The condition $\l_1-\l_n\le p-n$ translates to the condition $\mu_1\le p-n$ for $\mu$.
        \item The transpose $\mu^t\in C_p$ is a weight for $GL_{p-n}$. Tensor $V_{\mu^t}$ with $\chi^{\l_n}$ to turn it into an object of degree zero (here we work in the category $Ver_p(PGL_{p-n})$).
        \item This object is $\mathbf{D}(V_\l)$.
    \end{enumerate}
    \end{remark}
    This approach makes the combinatorial description of the bijection on simple objects less evident. However, it allows us to work only with polynomial representations of $GL_n$, which can be understood through Schur functors.
\begin{remark}\label{rem_level_rank_over_sVec}
    The equivalence $\mathbf{D}$ can be extended to an equivalence of $sVec$-module categories
    $$
    \mathbf{D}:Ver_p(GL_n)\boxtimes sVec\xrightarrow{\sim}Ver_p(GL_{p-n})\boxtimes sVec.
    $$
    In particular, we can extend the bijection $\l\leftrightarrow D(\l)$ to weights of arbitrary degree. 

    For now, let us avoid ambiguity by denoting $V\in Ver_p(GL_k)$ by $V^{(k)}$ for each $k<p$, similarly, object $V_\l$ will be denoted by $V^{(k)}_\l$.

    We start with the object $V^{(n)}_\l = V^{(n)}_\mu\otimes det^{\l_n} =  \mathbb S_\mu V^{(n)}\otimes det^{\l_n}$, where $\mathbb S_\mu$ is the Schur functor corresponding to $\mu$, and $\det = \Lambda^n V^{(n)} = \mathbb S_{\nu}V^{(n)}$ for $\nu=(1,\ldots, 1)$. Note here that $\nu^t=(n, 0,\ldots, 0)$, so $\mathbb S_{\nu^t}V^{(p-n)}=\chi$ in $Ver_p(GL_{p-n})$. 

    Let $I$ be the purely odd super vector space of super dimension $(0|1)$. We define $\mathbf{D}(V^{(n)})$ to be $V^{(p-n)}\otimes I$. Note for instance that this map behaves well on dimensions: $$\dim V^{(p-n)}\otimes I = -(p-n) = n \in \ck.$$
    
    Then we have
    $$\mathbf{D}(V^{(n)}_\l)=
    \mathbf{D}(V^{(n)}_\mu\otimes det^{\l_n})=
    $$
    
    $$=\mathbb S_\mu(V^{(p-n)}\otimes I)\otimes (\mathbb S_\nu(V^{(p-n)}\otimes I))^{\otimes \l_n}=
    $$
    $$
    =\mathbb S_{\mu^t}V^{(p-n)}\otimes I^{\otimes \sum \mu_i}\otimes (\mathbb S_{\nu^t}V^{(p-n)})^{\otimes \l_n}\otimes I^{\otimes \l_n}=
    $$
    $$
    = V^{(p-n)}_{\mu^t}\otimes \chi^{\l_n}\otimes I^{\sum \l_i}.
    $$

    When $\sum \l_i= 0$,
    $$
    \mathbf{D}(V_\l^{(n)}) = V_{\mu^t}^{(p-n)}\otimes \chi^{\l_n}
    $$
    has degree zero and lies in the subcategory $Ver_p(PGL_{p-n})$.
    \end{remark}
\begin{ex}
    Suppose $p=7, n = 3$ and $\l=(6,5,2)\in C_p$ for $GL_3$. Then $\mu = (4,3,0)$ and
    $$
    V^{(3)}_{\l} = det^{2}\otimes V^{(3)}_{\mu}.
    $$
    Therefore,
    $$
    \mathbf{D}(V_\l^{(3)}) = V^{(4)}_{\mu^t}\otimes \chi^{2}\otimes I^{|\l|}= V^{(4)}_{(2,2,2,1)}\otimes V^{(4)}_{(3,0,0,0)} \otimes V^{(4)}_{(3,0,0,0)}\otimes I= 
    $$
    $$
    = V^{(4)}_{(1,1,1,1)}\otimes V^{(4)}_{(1,1,1,0)}\otimes V^{(4)}_{(3,0,0,0)}\otimes \otimes V^{(4)}_{(3,0,0,0)} I = V^{(4)}_{(1,1,1,1)} \otimes V^{(4)}_{(3,1,1,1)} \otimes V^{(4)}_{(3,0,0,0)} = 
    $$
    $$
    = V^{(4)}_{(2,2,2,2)}  \otimes V^{(4)}_{(2,0,0,0)}\otimes V^{(4)}_{(3,0,0,0)}\otimes I = V^{(4)}_{(2,2,2,2)}  \otimes V^{(4)}_{(3,2,0,0)} = V^{(4)}_{(5,4,2,2)}\otimes I.
    $$
    So, $D(\l)= (5,4,2,2)$.
\end{ex}

\section{Representations of \texorpdfstring{$GL(L_n)$ for simple objects $L_n$}{GL(L n) for simple objects L n}}\label{GL(L_n)}

In this section let us assume char $\ck > 2$.

\subsection{Classification of \texorpdfstring{$GL(L_n)$}{GL(L n)}-modules via fiber functors}
Let us start our discussion of representations of $GL(X)$ in $Ver_p$ by describing the category of representations of $GL(L_n)$ for some $n=1,\ldots, p-1$. Moreover, as $L_1$ and $L_{p-1}$ span the subcategory equivalent to super vector spaces, we already know the representation theory of $GL(L_1)$ and $GL(L_{p-1})$. Let us assume therefore that $2\le n\le p-2$.

In \cite{V24}, Venkatesh proved the following
\begin{lemma}[\cite{V24}, Corollary 4.3]
There is an isomorphism of affine group schemes of finite type in $Ver_p$:
$$
GL(L_n)\simeq GL(L_n)_{0}\times SL(L_n)=\mathbb G_m\times SL(L_n),
$$
where $SL(L_n)$ is a group scheme defined in Example \ref{ex_SL(X)}.
\end{lemma}
\begin{proof}
    It follows immediately from the observation that, since $\dim L_n\neq 0$, we have a Lie algebra decomposition
    $$
    \gl(L_n)=\on\oplus \ssl(L_n)
    $$
    (the splitting map $\on\to\gl(L_n)$ is given by $\tau_{L_n^*,L_n}\circ coev_{L_n}$). 

    Now $\mathbf{HC}(SL(L_n))=(1, \ssl(L_n)), ~\mathbf{HC}(\mathbb G_m)=(\mathbb G_m, \on)$, and
    $$
    \mathbf{HC}(GL(L_n))=(\mathbb G_m, \gl(L_n))=(\mathbb G_m\times 1, \on\oplus\ssl(L_n)).
    $$
\end{proof}

\begin{remark}
    As $SL(L_n)$ is simultaneously a subgroup and a quotient of $GL(L_n)$, we might as well denote it by $PGL(L_n)$. This notation is more suitable in light of the results of this section. 
\end{remark}

\begin{remark}
    Note that we ought to keep track of what happens to the homomorphism $$\varepsilon: \pi_{Ver_p}\to GL(L_n).$$

    Let $p: GL(L_n)\to \mathbb G_m$ be the projection. Then
    $$
    p\circ\varepsilon: \pi_{Ver_p} \to \mathbb G_m
    $$
    is not always trivial.
    
    Recall that $\pi_{Ver_p}=\mu_2\times \pi_{Ver_p^+}$. Then $p\circ\varepsilon$ factors through $\mu_2$ and sends the generator $z\in \mu_2$ to $(-1)^{n-1}$.

    \end{remark}

    \begin{corollary}
    We have an equivalence of categories
        $$
        \Rep_{Ver_p}(GL(L_n),\varepsilon)=\Rep(\mathbb G_m, z)\boxtimes \Rep_{Ver_p}(SL(L_n),\varepsilon'),
        $$
            where $\varepsilon'$ is the composition of the natural map $\varepsilon:\pi_{Ver_p}\to GL(L_n)$ with the projection $GL(L_n)\to SL(L_n)$. 
    \end{corollary}

    Thus, we will focus on studying the representations of $SL(L_n)$. Let us reprove here the following result of Venkatesh:

\begin{theorem}[\cite{V24}, Corollary 4.10]\label{reps_of_GL(L_n)}
    The fiber functor $$F^+: Ver_p^+(SL_n)=Ver_p(PGL_n)\to Ver_p$$ induces an equivalence of categories
    $$
     Ver_p(PGL_n)=Ver_p^+(SL_n)\xrightarrow{\sim} \Rep_{Ver_p}(SL(L_n),\varepsilon').
    $$ 
\end{theorem}

Let us denote the fundamental group of the category $Ver_p(GL_n)$ by $\pi$ and let us denote the fundamental group of the subcategory $Ver_p^+(GL_n)=Ver_p(PGL_n)$ by $\pi^+$. Note that $\pi^+$ is naturally a quotient of $\pi$.

Consider the fiber functor 
$$
F: Ver_p(GL_n)\to Ver_p,
$$
and denote by $F^+$ its restriction to the subcategory $Ver^+_p(GL_n)$.

\subsection{Proof of Theorem \ref{reps_of_GL(L_n)}}
\begin{itemize}
    \item First, let us note that the category $Ver_p(GL_n)$ is tensor generated by the object $V$. The object $\ssl(V) = V_{(1,0,\ldots,0,-1)}$ is simple and tensor generates the subcategory $Ver_p(PGL_n)$.

    \item By Theorem \ref{fund_gp_is_subgp_of_GL(X)}, we get that $\pi$ is a subgroup of $GL(V)$. The Lie algebra object $Lie(\pi)$ is then a subalgebra of $\gl(V)=\on\oplus\ssl(V)$.

\item Note that $Lie(\pi)$ cannot be trivial or isomorphic to $\on$, because this would imply that $F(\pi)$ is an classical group scheme in $Vec\subset Ver_p$ (due to Harish-Chandra correspondence, see Theorem \ref{thm_HC}). The latter, however, is impossible since $F(\gl(V)) = L_n\otimes L_n^*$ is a subobject in $F(\co(\pi))$ (here we need the assumption that $2\le n\le p-2$).

\item Therefore, $Lie(\pi)$ is either isomorphic to $\gl(V)$ or to $\ssl(V)$.

Now, $Lie(\pi^+)$ is a quotient of $Lie(\pi)$, and for similar reasons it is neither trivial, nor isomorphic to $\on$. Therefore, again, there are only two possible options for it: $\gl(V)$ or $\ssl(V)$.

\item Consider now $F(\pi^+)$ as a group scheme in $Ver_p$. We have
$$
F(\pi^+)_{0}(\ck) = F(\pi^+)(\ck) = \mathrm{Aut}^\otimes(F^+).
$$

\item Let us prove that there are no nontrivial tensor automorhisms of the functor $F^+$. This would force the classical part $F(\pi^+)_{0}$ of $F(\pi^+)$ to be trivial.

Let $\alpha\in \mathrm{Aut}^\otimes(F^+)$. Consider the object $\widetilde V = V\otimes \psi^{-1}$ in $Ver_p(GL_n)$, where $\psi$ is the degree one invertible object of dimension $(-1)^{n-1}$ in $Ver_p(GL_n)$ constructed in Section \ref{sect_tensor_subcat}. Since $V$ has degree one, we get $\widetilde V\in Ver_p(PGL_n)$ (as $Ver_p(PGL_n)$ is precisely the degree zero part of $Ver_p(GL_n)$). Then we get
$$
\ssl(V)=\ssl(\widetilde V)\subset \widetilde V\otimes \widetilde V^*.
$$
Moreover, $F^+(\widetilde V)=F(V)\otimes F(\psi) = L_n\otimes (L_{p-1})^{\otimes(n-1)}$ is a simple object of $Ver_p$. Therefore, $\alpha_{\widetilde V}$ is a constant. We have
$$
\alpha_{\widetilde V\otimes \widetilde V^*} = \alpha_{\widetilde V}\cdot \alpha_{\widetilde V}^{-1}=1,
$$
since $\alpha$ is a tensor automorphism. We get that
$$
\alpha_{\ssl(V)}  = 1.
$$
As $\ssl(V)$ tensor generates the category $Ver_p(PGL_n)$, we get that $\alpha_X=1$ for each $X\in Ver_p(PGL_n)$.

\item We deduce that $\mathbf{HC}(F(\pi^+)) = (1, F(Lie(\pi^+)))$ with $F(Lie(\pi^+))$ isomorphic to either $\gl(F(V)) = \gl(L_n)$ or $\ssl(F(V))=\ssl(L_n)$. Since we must have $F(Lie(\pi^+))_{0}=Lie(1)=0$, we obtain the desired result:
$$
\mathbf{HC}(F(\pi^+)) = (1, \ssl(L_n)),
$$
and thus
$$
F(\pi^+) = SL(L_n) = PGL(L_n).
$$

\item The statement of the theorem follows from the reconstruction theorem.
\end{itemize}
    \blacksquare

 \begin{corollary}
     Since $L_n = L_{p-1}\otimes L_{p-n}$, and the Lie algebra objects $\ssl(L_n)$ and $\ssl(L_{p-n})$ are isomorphic in $Ver_p$ (as $L_{p-1}$ is invertible), we also get the isomorphism of group schemes $SL(L_n)\simeq SL(L_{p-n})$ (together with the homomorphism out of the fundamental group). Therefore, we obtain the level-rank duality \ref{level_rank}:
     $$
     Ver_p(PGL_n)\simeq \Rep_{Ver_p}(SL(L_n))\simeq \Rep_{Ver_p}(SL(L_{p-n}))\simeq Ver_p(PGL_{p-n}).
     $$
 \end{corollary}   

\begin{corollary}\label{cor_reps_of_GL(L_n)}
    The fiber functor
    $$
    F: Ver_p(GL_n)\to Ver_p
    $$
    induces an equivalence of categories 
    $$
    Ver_p(GL_n)\xrightarrow{\sim} \Rep_{Ver_p}(GL(L_n),\varepsilon).
    $$
\end{corollary}

\begin{corollary}
    The fundamental group scheme $\pi$ of the category $Ver_p(GL_n)$ is isomorphic to $GL(V)$. Thus, every object in $Ver_p(GL_n)$ enjoys the natural action of the group scheme $GL(V)$ and the Lie algebra object $\gl(V)$.
\end{corollary}

\begin{corollary} The category $\Rep_{Ver_p}(GL(L_n),\varepsilon)$ of representations of $GL(L_n)$ in $Ver_p$ is semisimple, and simple objects are in bijection with weights $\l=(\l_1,\ldots, \l_n)$ with
    $$
    \l_1\ge\l_2\ge\ldots\ge\l_n,~\text{ and } \l_1-\l_n\le p-n.
    $$
    The bijection sends $\l$ to $F(V_\l)$, where $V_\l\in Ver_p(GL_n)$.
\end{corollary}

\section{Categorical \texorpdfstring{$\widehat{\mathfrak{sl}}_p$-action on $Ver_p(GL_n)$}{action on Ver p(GL n)}}\label{catact}
In this section we describe in more detail the combinatorics of the ``highest weights" for representations of $GL(L_n)$ for simple object $L_n\in Ver_p$. We will do so by studying the natural categorical type $A$ action on $\Rep_{Ver_p}(GL(L_n)) = Ver_p(GL_n)$ induced from tensoring with the generating object $V$.

\subsection{Translation functors}\label{section_transl_fun}

Let $2\le n\le p-2$ and let $V$ be the image of the tautological $n$-dimensional representation of $GL_n$ under the semisimplification functor $\mathbf{S}: \mathcal T(GL_n)\to Ver_p(GL_n)$. Let us fix $n$ for now and denote the category $Ver_p(GL_n)$ by $\cc$.

Define functors $F, E:\cc\to \cc$ via
$$
F(M)=V\otimes  M;
$$
$$
E(M)=V^*\otimes M.
$$
These functors are exact and biadjoint.

Recall from the previous section that every object $M\in \cc$ is equipped with a natural action of the Lie algebra object $\gl(V)$, let us denote the corresponding action map by $a_M: V\otimes V^*\otimes M\to M$. Define the \textbf{tensor Casimir} operator $x\in \En(F)$ with $$x_M:V\otimes M\to V\otimes M$$ corresponding to $a_M$ under the identification $$\Hom_\cc(V\otimes M,V\otimes M)=\Hom_\cc(V\otimes V^*\otimes M, M).$$ 

Note that if $M=\mathbf{S}(\tilde M)$ with $\tilde M\in \mathcal T(GL_n)$ then 
$x_M$ is obtained from the usual tensor Casimir acting on $\tilde V\otimes \tilde M$ via $\sum_{i,j\le n} E_{ij}\otimes E_{ji}$ (where $E_{ij}$ with $i,j\in\{1,\ldots ,n \}$ form the standard basis in $\gl_n$). Thus we obtain the direct sum decomposition:
$$
F=\bigoplus_{c\in\mathbb F_p} F_c,
$$
where $F_{c}(M)$ is the eigenspace of $x_M$ with eigenvalue $c\in \mathbb F_p$. One can define $F_{c}$ directly on simple objects. Recall that simple objects of $\cc$ are labeled by integral weights $\lambda=(\lambda_1,\ldots, \lambda_n)\in \Z^n$ such that
$$
\lambda_i\ge \lambda_{i+1} \text{ for } i=1,\ldots, n-1, \text{ and }
$$
$$
\lambda_1-\lambda_n\le p-n
$$
 (see Example \ref{ex_fund_alc}).
 
We will call $n$-tuples $\l\in \Z^n$ satisfying the relations above \textbf{admissible weights}.

Then $F_c(V_\l)$ is a subobject of $V\otimes V_\l$. We have
$$
V\otimes V_\l = \bigoplus_{i=1}^n V_{\l+e_i},
$$
where $V_{\l+e_i}$ is $V_{\mu}$ for 
$$
\mu=(\l_1,\ldots, \l_{i-1},\l_i+1,\l_{i+1},\ldots, \l_n)
$$
either if $i>1$ and $\l_{i-1}>\l_i$, or if $i=1$ and $(\l_1+1)-\l_n\le p-n$ (that is, if $\mu$ is an admissible weight). Otherwise, we put $V_{\l+e_i}=0$.

Let us denote the set of all admissible weights of the form $\l+e_i$ for some $i\le n$ by $\l+\square$. Let $\mu\in \l+\square$. We write $\mu\in \l+\square_c$ if $\mu=\l+e_i$, where 
$$
c = \l_i+1-i \in \mathbb F_p.
$$
\begin{remark}
    Here $c$ generalizes the concept of the \textbf{content} of the box added to the Young diagram of $\l$.
\end{remark}
\begin{lemma}\label{lemma_content}
    For each $c\in \mathbb F_p$ and each admissible weight $\l$, the set $\l+\square_c$ contains at most one element.
\end{lemma}
\begin{proof}
    Suppose $\l+e_i$ and $\l+e_j$ are both admissible weights and
    $$
    \l_i+1-i = \l_j+1-j \mod p.
    $$
    Without loss of generality assume $i< j$. Then $\l_i\ge \l_j$ and $-i> -j$ as integers. If the equality holds modulo $p$ then 
    $$
    \l_i-\l_j -i+j = p\cdot k
    $$
    for some $k>0$. Thus
    $$
    \l_i-\l_j \ge p + i-j \ge p+1-n.
    $$
    This implies that 
    $$
    \l_1-\l_n\ge 
    \l_i-\l_j\ge p+1-n,
    $$
    which contradicts the assumption that $\l_1-\l_n\le p-n$.
 \end{proof}

Thus $F(V_\l) = \bigoplus_{\mu\in \l+\square} V_\mu$, and it is easy to check (using, for example, the explicit formula for the action of the tensor Casimir before the semisimplification) that
$$
F_c(V_\l) = \begin{cases}
    V_\mu, \text{ if } \mu\in \l+\square_c\neq \emptyset,\\
    0, \text{ if } \l+\square_c = \emptyset.
\end{cases}
$$

We get the decomposition
$$
E=\bigoplus_{c\in \mathbb F_p} E_c,
$$
where $E_c$ is the functor left adjoint to $F_c$.

Further, let us define the operator $t\in \En(F^2)$ via
$$
t_M = \tau_{V,V}\otimes id_M: V\otimes V\otimes M\to V\otimes V\otimes M.
$$

\subsection{Definition of categorical action}

\begin{definition}\label{def_daHa}
    The \textbf{degenerate affine Hecke algebra} $\mathcal H_N$ is defined by generators and relations. $\mathcal H_N$ is generated by elements $x_1,\ldots, x_N, t_1,\ldots, t_{N-1}$ subject to the following relations:
    \begin{enumerate}
        \item\label{daHa_relation_1} $
        x_ix_j=x_jx_i,$
        \item \label{daHa_relation_2} $ t_i^2=1,$
        \item \label{daHa_relation_3} $t_it_j=t_jt_i \text{ if } |i-j|>1,
        $
        \item \label{daHa_relation_4} $ t_i t_{i+1} t_i = t_{i+1} t_i t_{i+1},
        $
        \item \label{daHa_relation_5} $ t_i x_{i+1} - x_it_i=1,
        $
        \item \label{daHa_relation_6} $ t_ix_j =x_j t_i \text{ if } j-i\neq 0,1. 
        $
    \end{enumerate}
\end{definition}

As a vector space $\mathcal H_N = \ck[x_1, \ldots, x_n]\otimes \ck S_N$, where $S_N$ is the symmetric group on $N$ elements.

\begin{definition}
    The \textbf{affine Lie algebra} $\widehat \ssl_p$ over $\C$ is the central extension of the \textbf{loop algebra} $\ssl_p[t,t^{-1}]=\ssl_p\tens{\C} \C[t,t^{-1}]$ by central element $C$ defined via the relations
    $$
    [A\cdot t^k+\alpha\cdot C, B\cdot t^l+\beta\cdot C] = [A,B]\cdot t^{k+l} + \delta_{k+l,0}\cdot k\cdot \langle A,B\rangle \cdot C,
    $$
    where $A,B\in\ssl_p, \alpha,\beta\in \C$ and $\langle -, - \rangle$ is the Killing form on $\ssl_p$.
    \end{definition}

    Recall that $\widehat\ssl_p$ is generated by elements $e_a, f_a$, and $h_a=[e_a, f_a]$ for $a\in \mathbb F_p$ satisfying the Kac-Moody relations:
\begin{enumerate}
    \item $h_a=[e_a, f_a], 2e_a=[h_a, e_a], -2f_a=[h_a, f_a]$,
    \item $[h_a, e_b] = -e_b, [h_a, f_b]=f_b$ if $a-b=\pm 1\in \mathbb F_p$,
    \item $[h_a, e_b]=[h_a, f_b]=0$ if $a-b\not\in\{1,0,-1\}\subset \mathbb F_p$,
    \item $[e_a, f_b] = 0$ if $a\neq b$,
    \item $[e_a,[e_a, e_b]] = 0, [f_a,[f_a, f_b]]=0$ if $a-b=\pm 1\in \mathbb F_p$,
    \item $[e_a, e_b]=[f_a,f_b]=0$ if $a-b\not\in \{-1,0,1\}\subset\mathbb F_p$.
\end{enumerate}

For any representation $W$ of the finite-dimensional Lie algebra $\ssl_p$ one can define the loop module $W^{loop}$ as $ W[t,t^{-1}]$ with the action of $\widehat\ssl_p$ so that $C$ is acting by zero.

The following definition is a slightly modified version of Definition 2.6 in \cite{BLW17} and Definition 5.2 in \cite{L12}.

\begin{definition}
    Let $\cc$ be a $\ck$-linear artinian abelian category. A categorical $\widehat\ssl_p$-action on $\cc$ is the data $(E,F, x, t)$, where $E$ and $F$ are endofunctors of $\cc$, $x\in\En(F)$, and $t\in \En(F^2)$, subject to the following axioms:
   \begin{enumerate}
       \item $E$ and $F$ are biadjoint exact functors.
       \item We have the decomposition $F = \bigoplus_{c\in \mathbb F_p} F_c$,  where $F_c$ is the generalized eigensubfunctor of $F$ with eigenvalue $c$ with respect to $x$. This automatically yields the decomposition $E = \bigoplus_{c\in\mathbb F_p}  E_c$, where $E_c$ is left adjoint to $F_c$.
       \item The endomorphisms
       $$
       x_i = 1^{N-i} x 1^{i-1}, t_k = 1^{N-k-1}t1^{k-1}
       $$
       of $F^N$ satisfy the relations of the degenerate affine Hecke algebra $\mathcal H_N$.
        \end{enumerate}
\end{definition}

Let $\cc$ be a category with a categorical $\widehat\ssl_p$-action.
\begin{itemize}
    \item There is a natural algebra homomorphism $\phi_N:\mathcal H_N\to \En(F^N)$.

    \item The operators
    $$
    f_c = [F_c], e_c=[F_c], \text{ for } c\in\mathbb F_p 
    $$
    acting on the complexified Grothendieck group $[\cc]$ of $\cc$ satisfy Kac-Moody relations for $\widehat\ssl_p$. Thus, they define the action of the affine Lie algebra $\hat \ssl_p$ on $[\cc]$.
\end{itemize}

\begin{ex}\label{cat_act_for_GL_n}
    Let $\cc = \Rep_\ck GL_n$ and let $V$ be the defining $n$-dimensional representation of $GL_n$. The functors $F=V\otimes -$ and $E=V^*\otimes -$, the {tensor Casimir} operator $x\in \En(F)$, and the swap $t = \tau_{V,V}\in \En(F^2)$ define the categorical $\widehat\ssl_p$ action on $\cc$. The complexified Grothendieck group $[\cc]$ admits a basis consisting of classes of Weyl modules $[V_\l], \l\in P^+$. Let $U=\C^p$ be the tautological $p$-dimensional $\ssl_p$-representation. We have an isomorphism of $\widehat\ssl_p$-modules
    $$
    \phi:[\cc]\xrightarrow{\sim} \Lambda^n(U^{loop})
    $$
    defined as follows:

   Let $\l_i-i+1 = a_i + p\cdot s_i$ for some $a_i\in \{0,\ldots, p-1\}$, $s_i\in \Z$, and each $i=1, \ldots, n$. Then 
   \begin{equation}\label{cat_act_for_GL_n_formula}
       \phi([V_\l]) = (v_{a_1}\cdot t^{-s_1})\wedge (v_{a_2}\cdot t^{-s_2})\wedge\ldots\wedge (v_{a_n}\cdot t^{-s_n}), 
   \end{equation}
    where $\{v_0, v_1,\ldots, v_{p-1}\}$ is the standard basis in $U$. 

   Note that injectivity of $\phi$ follows from the condition that $\l_1>\l_2-1>\ldots> \l_n-n+1$, which means that all vectors in the wedge product (\ref{cat_act_for_GL_n_formula}) are distinct.

   For proof of this computation see \cite{L12}.
    \end{ex}

\subsection{Computation for \texorpdfstring{$Ver_p(GL_n)$}{Ver p (GL n)}}

Let us go back to the situation $\cc=Ver_p(GL_n)$. Let $F=V\otimes -, E=V^*\otimes -$ and let $x\in \En(F), t\in \En(F^2)$ be the operators defined in Section \ref{section_transl_fun}. 

\begin{theorem}\label{thm_sl_p_act}
\begin{enumerate}
    \item The data $(E,F, x, t)$ defines the categorical $\widehat\ssl_p$-action on the category  $\cc= Ver_p(GL_n)$. 
    \item With respect to this action, we get an isomorphism of $\widehat\ssl_p$-modules
    $$
   \phi:  [\cc] \xrightarrow{\sim} (\Lambda^n U)^{loop} = (\Lambda^n U)[t,t^{-1}],
    $$
    where $[\cc]$ is the complexified Grothendieck group of $\cc$ and $U=\C^p$ is the tautological $p$-dimensional representation of $\ssl_p$.
    \end{enumerate}
\end{theorem}
\begin{proof}
    \begin{enumerate}
        \item It is clear that the operators $t_i$ satisfy the same relations as elementary transpositions $(i~i+1)$ in the symmetric group $S_N$  (i.e. relations (\ref{daHa_relation_2}), (\ref{daHa_relation_3}), and (\ref{daHa_relation_4}) in Definition \ref{def_daHa}). Moreover, operators $x_1,\ldots, x_N$ obviously commute as they act by endomorphisms on different copies of $F$ in $F^N$, thus the relation \ref{def_daHa}(\ref{daHa_relation_1}) is satisfied as well. Similarly, relation \ref{def_daHa}(\ref{daHa_relation_6}) holds because $t_i$ and $x_j$ act on different copies of $F$ in $F^N$ when $j-i\neq 0,1$. The only interesting relation is relation \ref{def_daHa}(\ref{daHa_relation_5}), and it is enough to check it in $\En(F^3)$. We need to prove that
        $$
        t\circ (x1)-(1x)\circ t = 1, 
        $$
        i.e. for any $M\in \cc$ we have
        $$
        (\tau_{V,V}\otimes id_M)\circ x_{V\otimes M} - (id_V\otimes x_M)\circ (\tau_{V,V}\otimes id_M) = id_{V\otimes V\otimes M}.
        $$
        To see this, note that the action map $a_M:V\otimes V^*\otimes M\to M$ satisfies
        $$
        a_{V\otimes M} = id_V\otimes ev_V\otimes id_M+ (id_V\otimes a_M)\circ \tau_{V\otimes V^*, V}: (V\otimes V^*)\otimes (V\otimes M),
        $$
        which implies
        $$
        x_{V\otimes M} = \tau_{V,V}\otimes id_M + (\tau_{V,V}\otimes id_M)\circ (id_V\otimes x_M)\circ (\tau_{V,V}\otimes id_M).
        $$
        Thus
        $$
        x1 = t + t\circ(1x)\circ t,
        $$
        and the statement follows.

        \item Recall that $[\cc]$ has a basis $[V_\l]$, where $\l$ is \textbf{admissible}, i.e. $\l=(\l_1,\ldots, \l_n)\in \Z^n$ is subject to
        $$
        \l_1\ge\ldots\ge \l_n, \text{ and } \l_1-\l_n \le p-n. 
        $$
        Define $a_i\in \{0,1,\ldots, p-1\}, s_i\in \Z$ so that
        $$
        \l_i+1-i = a_i + p\cdot s_i.
        $$
        Put $s = \sum_{i=1}^n s_i$. The $\C$-linear map $\phi: [\cc]\to (\Lambda^n U)[t,t^{-1}]$ is defined on the basis as follows:
        $$
        \phi: [V_\l] \mapsto (v_{a_1}\wedge v_{a_2}\wedge \ldots\wedge v_{a_n})\cdot t^{-s}, 
        $$
        where $\{v_0,\ldots, v_{p-1}\}$ is the standard basis in $U$, so that
        $$
        f_a v_b = \delta_{a,b} v_{a+1}, \text{ if } a\neq p-1,  
        $$
        $$
        e_a v_b = \delta_{a+1, b} v_{a}, \text{ if } a\neq p-1.
        $$
        (Here we consider $e_0, \ldots, e_{p-2}, f_0, \ldots, f_{p-2}$ to be the generators of the finite-dimensional algebra $\ssl_p$.)

        We claim that $\phi$ is an isomorphism of $\widehat\ssl_p$ modules.

        First, note that the proof of Lemma \ref{lemma_content} implies that if $\l$ is admissible, then integers $$\{\l_i+1-i~|~i=1,\ldots, n\}$$ are pairwise distinct modulo $p$, which means that $a_1,\ldots,a_n$ are pairwise distinct elements in $\mathbb F_p$.
        
        Note that for every $n$-tuple $(a_1, \ldots, a_n)\in \mathbb F_p^n$ and an integer $s$, there exists a unique admissible weight  $\l$, such that $\phi([V_\l]) = (v_{a_1}\wedge v_{a_2}\wedge \ldots\wedge v_{a_n})\cdot t^{-s}$. Namely, if $s = n\cdot q + r$, where $0\le r<n$, and we consider $a_i$ as an integer between $0$ and $p-1$, we get
        $$
        \l_i = a_i - 1 + i + p\cdot q, \text{ if } i > r,
        $$
        $$
        \l_i = a_i -1 + i + p\cdot (q+1), \text{ if } i\le r.
        $$
        This shows that the map $\phi$ is bijective. Let us thus denote the basis element $(v_{a_1}\wedge v_{a_2}\wedge \ldots\wedge v_{a_n})\cdot t^{-s}$ by $v_\l$ when $\l$ is the corresponding admissible weight.
        
        Recall that if $\mu = \l_i+e_i$ is admissible then $\mu\in \l+\square_{a_i+p\cdot s_i}$, and
        $$
        f_{a_i} [V_\l] = [V_{\l+e_i}] = [V_\mu].
        $$
        If $a\neq a_i$ for each $i=1,\ldots, n$ or if $a=a_i$ and $\l+e_i$ is not admissible, we have
        $$
        f_{a}[V_\l] = 0.
        $$
        Similarly, as $E_a$ is left adjoint to $F_a$, we get
        $$
        e_{a_i} [V_\mu] = [V_{\mu-e_i}] = [V_\l], \text{ if } \mu-e_i  \text{ is admissible },
        $$
        $$
        e_a [V_\mu] = 0, \text{ otherwise}.
        $$
        (Here $a_i = \l_i+1-i = \mu_i - i$.)

        Note that if $a_i\neq p-1$, we have $\l_i+1 = (a_i+1)+p\cdot s_i$, and thus $$\phi([V_{\l+e_i}]) = (v_{a_1}\wedge\ldots\wedge v_{a_i+1}\wedge\ldots\wedge v_{a_n})\cdot t^{-s}.$$ 
        Whereas if $a_i=p-1$, we get $\l_i+1 =0+ p\cdot (s_i+1)$, so
        $$
        \phi([V_{\l+e_i}] = (v_{a_1}\wedge\ldots\wedge v_{a_i+1}\wedge\ldots\wedge v_{a_n})\cdot t^{-s-1}.
        $$
        
        Let us now describe the action of the generators $e_c, f_c, c\in \mathbb F_p$ on the loop module $(\Lambda^n U)[t,t^{-1}]$.

        We have
        $$
        f_c v_\l= f_c((v_{a_1}\wedge\ldots \wedge v_{a_n})\cdot t^{-s}) \neq 0
        $$
        only if $c= a_i$ for some $i=1,\ldots, n$, and $a_i+1\neq a_{j}$ for all $j=1,\ldots, n$. Let us rephrase this condition in terms of the corresponding admissible weight $\l$. If $a_i+1=a_j$ it means that 
        $$
        \l_i-i+1+1 = \l_j -j+1\mod p. 
        $$
        Note that since $|\l_i-\l_j|\le p-n$ for all $i,j$, the equality above can happen only if $j=i-1$ and $\l_i = \l_{i-1}$, and in this case $\l+e_i$ is not an admissible weight, or if $i=1$ and $j=p-1$, and in this case $\l_1 = \l_n -n + p$, that is $\l+e_1$ is not admissible. Thus $f_c v_\l\neq 0$ if and only if $c=a_i$ and $\l+e_i$ is admissible.

        Further, if $\l+e_i$ is admissible we get
        $$
        f_{a_i}v_\l = (v_{a_1}\wedge\ldots \wedge v_{a_i+1}\wedge\ldots\wedge v_{a_n})\cdot t^{-s}=v_{\l+e_i}, \text{ if } a_i\neq p-1,
        $$
        $$
        f_{p-1}v_\l = (E_{0,p-1}\cdot t^{-1})v_\l = (v_{a_1}\wedge\ldots \wedge v_{a_i+1}\wedge\ldots\wedge v_{a_n})\cdot t^{-s-1}=v_{\l+e_i}, \text{ if } a_i= p-1.
        $$
    (Here we used the explicit formula for the generators corresponding to the affine simple root as operators in $\ssl_p[t,t^{-1}]\subset \En(U)[t,t^{-1}]$:
    $f_{p-1} = E_{0,p-1}\cdot t^{-1}, \\ e_{p-1} = E_{p-1,0} \cdot t$.)
    
    The computation for the action of $e_a, a\in \mathbb F_p$ is completely analogous. 

    Thus we conclude that 
    $$
    f_c (\phi([V_\l]))= f_c v_\l = \phi(f_c [V_\l]), \text{ and }
    $$
    $$
    e_c(\phi([V_\l])= e_c v_\l = \phi(e_c[V_\l]),
    $$
    and thus $\phi$ is an isomorphism of $\widehat\ssl_p$-modules.

     \end{enumerate}
\end{proof}
\section{Representations of \texorpdfstring{$GL(X)$}{GL(X)} in \texorpdfstring{$Ver_p$}{Ver p} for non-simple \texorpdfstring{$X$}{X}} \label{section_simple_reps_of_GL(X)}

In this section we will restate the results of Venkatesh on the classification of simple representations of $GL(X)$ for $X\in Ver_p$. For further details see \cite{V24}. 

\subsection{Representations of \texorpdfstring{$GL(L_n^{\oplus k})$}{GL(L n k)}}
The first case to consider is $X= L_n^{\oplus k}$. Let us choose an ordering of simple summands in $X$, so that
$$
X = \bigoplus_{i=1}^k L_n^{(i)},
$$
where $L_n^{(i)}$ denotes the $i$'th copy of $L_n$ inside $X$. Let $G=GL(X)$. We have
$$
\mathbf{HC}(G) = (GL_k, \gl(X))=(GL_k, \bigoplus_{1\le i,j\le k} L_n^{(i)}\otimes(L_n^{(j)})^*),
$$
where $\mathbf{HC}$ is the Harish-Chandra functor defined in Section \ref{section_HC}.

Recall that the \textbf{maximal torus} in $G$ is the subgroup scheme $T = GL(L_n)^k = \prod_{i=1}^k GL(L_n^{(i)})$ with
$$
\mathbf{HC}(T)=(\mathbb G_m^k, \bigoplus_{i=1}^k \gl(L_n^{(i)})).
$$
The standard Borel subgroup $B\subset G$ is defined via
$$
\mathbf{HC}(B) = (B_k, \mathfrak b(X)) = (B_k, \bigoplus_{1\le i\le j\le k} L_n^{(i)}\otimes (L_n^{(j)})^*),
$$
where $B_k\subset GL_k$ is the subgroup of upper-triangular matrices. 

We have also define the subgroup $N^-$ of ``strictly lower-triangular matrices" in $G$ via
$$
\mathbf{HC}(N^-)=(N_k^-, \bigoplus_{1\le j< i\le k} L_n^{(i)}\otimes (L_n^{(j)})^*),
$$
where $N_k^-$ is the actual subgroup of strictly lower-triangular matrices in $GL_k$.

By Corollary \ref{cor_reps_of_GL(L_n)}, we get the equivalence
$$
\Rep_{Ver_p}(T, \varepsilon) \simeq Ver_p(GL_n)^{\boxtimes k} = \boxtimes_{i=1}^n Ver_p(GL_n)^{(i)}.
$$

Let $\l^{(1)}, \ldots, \l^{(k)}\in \Z^n$ be a sequence of \textbf{admissible weights}. Such tuples are in bijection with isomorphism classes of simple representations of $T$:
$$
(\l^{(1)},\ldots, \l^{(k)})\leftrightarrow V_{\l^{(1)}}\boxtimes\ldots\boxtimes V_{\l^{(k)}}\in Ver_p(GL_n)^{(1)}\boxtimes\ldots\boxtimes Ver_p(GL_n)^{(k)}.
$$
Let us abuse the notation and identify the $k$-tuple of weights with the corresponding representation:
$$
\bl=(\l^{(1)},\ldots, \l^{(k)})\in \Rep_{Ver_p}(T,\varepsilon).
$$

As $T$ is a quotient of $B$, we can pull $\bl$ back to $\Rep_{Ver_p}(B, \varepsilon)$.

\begin{definition}
    The generalized \textbf{Verma module} $M(\bl)$ over the distribution algebra $Dist(G)$  for $\bl\in sOb(\Rep_{Ver_p}(T,\varepsilon))$ is defined as follows:
    $$
    M(\bl) = Dist(G)\tens{Dist(B)} \bl.
    $$
\end{definition}

By Theorem \ref{thm_PBW}, as a left $Dist(N^-)$-module $M(\bl)$ is isomorphic to 
$Dist(N^-)\otimes \bl$. Here we consider $\bl$ an object of $Ver_p$ by applying the fiber functor $F:Ver_p(GL_n)^{\boxtimes k}\to Ver_p$ to $V_{\l^{(1)}}\boxtimes\ldots\boxtimes V_{\l^{(k)}}$.

Let us introduce a partial order on the set of $k$-tuples admissible weights in $\Z^n$ (equivalently, on the set of isomorphism classes of simple objects in $\Rep_{Ver_p}(T, \varepsilon)$). The classical subgroup $T_{0}=\mathbb G_m^k$ acts on every simple $T$-module $\bl$ with the character $(|\l^{(1)}|,\ldots, |\l^{(k)}|).$\footnote{Recall that for a $GL_n$-weight $\l$, we denote by $|\l|$ the sum $\sum_{j=1}^n \l_j$.} We then use the standard dominance order on characters of $T_{0}=\mathbb G_m^k$.

More specifically, we say that $\bl = (\l^{(1)},\ldots, \l^{(k)})$ is \textbf{dominant} if
$$
|\l^{(1)}|\ge \ldots\ge |\l^{(k)}|.
$$
And we write
$$
\bl = (\l^{(1)},\ldots, \l^{(k)}) \preceq \bm = (\mu^{(1)},\ldots, \mu^{(k)})
$$
if  
$$
|\mu^{(1)}|- |\l^{(1)}| \ge \ldots\ge |\mu^{(k)}|-|\l^{(k)}|.
$$

It makes sense now to talk about ``weight decomposition" and ``highest weights" for $G$-modules and $Dist(G)$-modules admitting a compatible action of $T$. The role of weights is played by irreducible representations of $T$. 
\begin{ex}
    Consider the adjoint action of $T$ on $$\mathfrak n^- = Lie(N^-) = \bigoplus_{1\le j<i\le k} L_n^{(i)}\otimes(L_n^{(j)})^*.$$ The component $L_n^{(i)}\otimes(L_n^{(j)})^*$ is a simple $T$-submodule with character $\ba_{i,j}$, such that 
    $$
    \ba_{i,j}\simeq (V^*)^{(j)}\boxtimes V^{(i)} :=
    $$
    $$
    = \on^{(1)}\boxtimes \ldots\boxtimes \on^{(j-1)} \boxtimes (V^*)^{(j)}\boxtimes \on^{(j+1)}\boxtimes\ldots\boxtimes \on^{(i-1)}\boxtimes V^{(i)} \boxtimes \on^{(i+1)}\boxtimes \ldots\boxtimes \on^{(k)}
    $$
    as objects in $\Rep_{Ver_p}(T,\varepsilon) = Ver_p(GL_n)^{(1)}\boxtimes \ldots \boxtimes Ver_p(GL_n)^{(k)}$. The corresponding $T_{0}$-character is
    $$
    \alpha_{i,j} = e_i-e_j := (0,\ldots, 0, \underbrace{-1}_j, 0,\ldots, 0 ,  \underbrace{1}_{i}, 0, \ldots, 0).
    $$
\end{ex}

The highest weight in a module $M$ is a weight $\bl$ (i.e. $\bl\in sOb(\Rep_{Ver_p}(T,\varepsilon))$), such that
$$
\Hom_T(\bl, M) = \ck, 
$$
and for any other weight $\bm$ with $\Hom_T(\bm, M)\neq 0$ we have
$$
\bm \preceq \bl.
$$
A $T$-module possessing the highest weight is called a \textbf{highest weight module}.
\begin{lemma}
    The generalized Verma module $M(\bl)$ is a highest weight module with highest weight $\bl$.
\end{lemma}
\begin{proof}
    It follows directly from the decomposition $M(\bl)=Dist(N^-)\otimes \bl$.
\end{proof}

\begin{theorem}[\cite{V24}, Proposition 5.3]\label{thm_simple_quotient}
    As a $Dist(G)$-module the generalized Verma module $M(\bl)$ has a unique maximal proper submodule $J(\bl)$ (namely the sum of all submodules, not containing the highest weight $\bl$). Consequently, it has a unique simple quotient $$L(\bl)=M(\bl)/J(\bl),$$ which is also a highest weight module with highest weight $\bl$.
\end{theorem}

\begin{theorem}[\cite{V23}, Corollary 1.4, \cite{V24}, Corollary 5.5] \label{thm_integrability} $~$
\begin{enumerate}
    \item For any object $X\in Ver_p$, the action of $Dist(GL(X))$ on a module $M$ integrates to the action of the group $GL(X)$ if and only if the action of $Dist(GL(X)_{0})$ on $M$ integrates to the action of $GL(X)_{0}$.

    \item When $X=L_n^{\oplus k}$, the action of $Dist(GL(X))$ on $L(\bl)$ integrates to the action of the group $GL(X)$ if and only if $\bl$ is dominant.

\end{enumerate}
    \end{theorem}

 \begin{corollary}
     For any object $X\in Ver_p$ and any $T(X)$-module $M$ of finite length (i.e. $M\in Ver_p$ as opposed to $M\in \Ind(Ver_p)$) if $M$ admits a compatible action of $Dist(GL(X))$ then this action is integrable.
 \end{corollary} 
\begin{proof}
    This is true about the action of $Dist(GL(X)_{0})$ on $M$.
\end{proof}

\begin{theorem}[\cite{V24}, Corollary 5.6 and Corollary 5.7] \label{thm_reps_of_GL(kL_n)} $~$
\begin{enumerate}
    \item If $L$ is an irreducible representation of $GL(X)$ then $L\simeq L(\bl)$ for some dominant $\bl\in \Rep_{Ver_p}(T,\varepsilon)$.

    \item If $L(\bl)\simeq L(\bl')$ for dominant weights $\bl,\bl'$ then $\bl=\bl'$.
\end{enumerate}
\end{theorem}

\subsection{Representations of \texorpdfstring{$GL(X)$}{GL(X)} for general \texorpdfstring{$X$}{X}}

Consider now any object $X\in Ver_p$ together with a decomposition into the direct sum of irreducibles
$$
X = \bigoplus_{n=1}^{p-1} L_n^{\oplus k_n}.
$$
The group scheme $GL(X)$ is the generalization of the general linear super group $GL(m|n)$. In a sense, we have
$$
GL(X) = GL(k_1|k_2|\ldots|k_{p-1}).
$$

Consider a ``Levi'' subgroup 
$$
M = \prod_{n=1}^{p-1} GL(L_n^{\oplus k_n})\subset GL(X).
$$ 
In the super group case $GL(m|n)$ the role of this subgroup would be played by the even subgroup $GL_m\times GL_n$. Note that in our case $M_{0} = GL(X)_{0}$ as well.

By Theorem \ref{thm_reps_of_GL(kL_n)}, simple $M$-modules are of the form
$$
L(\bl_1)\otimes \ldots \otimes L(\bl_{p-1}),
$$
where $\bl_n$ is a dominant weight in $\Rep_{Ver_p}(T(L_n^{\oplus k_n}), \varepsilon_n)$.

Define $\mathfrak p := \mathfrak b(X) \oplus Lie(M)$. It is clear that $\mathfrak p$ is a Lie subalgebra in $\gl(X)$ with 
$$
\mathfrak p_{0}  = Lie(M)_{0}.
$$
Thus we can define the ``parabolic" subgroup $P\subset G$ with
$$
\mathbf{HC}(P)=(M_{0}, \mathfrak p) = (\prod_{n=1}^{p-1} GL_{k_n}, \mathfrak p) = (GL(X)_{0}, \mathfrak p).
$$

We have
$$
\gl(X) = \mathfrak p \oplus \mathfrak u,
$$
with $\mathfrak u\subset \mathfrak n^-$ satisfying
$$
\mathfrak u_{0} = 0.
$$
Thus, by the PBW theorem and Lemma \ref{lemma_gps_with_no_even_part}, 
$$
Dist(GL(X)) = Dist(P)\otimes U(\mathfrak u),
$$
with $U(\mathfrak u)$ of finite length in $Ver_p$.

The action of $M$ on the simple module $L(\bl_1)\otimes \ldots \otimes L(\bl_{p-1})$ can be extended trivially to $P$. 
\begin{definition}
    Define the ``Kac module'' for $GL(X)$ as the parabolic induction:
    $$
    K(\bl_1,\ldots, \bl_{p-1}) := Dist(GL(X))\tens{Dist(P)} [L(\bl_1)\otimes \ldots \otimes L(\bl_{p-1})].
    $$
\end{definition}

Since $P_{0}=GL(X)_{0}$, the $Dist(GL(X))$ module $$K(\bl_1,\ldots, \bl_{p-1}) \simeq U(\mathfrak u)\otimes [L(\bl_1)\otimes \ldots \otimes L(\bl_{p-1})]$$ has finite length in $Ver_p$ and is integrable by Theorem \ref{thm_integrability}.

Recall that the choice of Borel subgroup $B(X)$ in $GL(X)$ required the ordering of simple components in $X$:
$$
X = X_1\oplus\ldots \oplus X_k,
$$
where $k = k_1+\ldots+ k_{p-1}$ and
$$
\underbrace{X_1\simeq \ldots \simeq X_{k_1}}_{k_1}\simeq L_1,
$$  
$$
\underbrace{X_{k_1+1} \simeq \ldots \simeq  X_{k_1+k_2}}_{k_2} \simeq L_2,
$$
$$
\cdots
$$
$$
\underbrace{X_{k_1+\ldots + k_{p-2}+1}\simeq \ldots \simeq X_k}_{k_{p-1}} \simeq L_{p-1}.
$$

This ordering induces the standard dominance order on characters of $$
T(X)_{0}=(GL(X_1)\times\ldots \times GL(X_k))_{0} = \mathbb G_m^k,
$$
and consequently on simple representations of $T(X)$.

Note that $K(\bl_1,\ldots, \bl_{p-1})$ is a highest weight module with respect to this dominance order with highest weight
$$
\bl_1\boxtimes \ldots\boxtimes \bl_{p-1}.
$$
The classical torus $T(X)_{0} = \mathbb G_m^k$ acts on this highest weight with character
$$
(|\l_1^{(1)}|,\ldots, |\l_1^{(k_1)}|, |\l_2^{(1)}|,\ldots, |\l_2^{(k_2)}|,\ldots, |\l_{p-1}^{(1)}|,\ldots, |\l_{p-1}^{(k_{p-1})}|).
$$

Consequently, the $GL(X)$-module $K(\bl_1,\ldots, \bl_{p-1})$ has a unique simple quotient $L(\bl_1,\ldots, \bl_{p-1})$.

\begin{theorem}[\cite{V24}, Proposition 6.2 and Proposition 6.3] $~$
    \begin{enumerate}
        \item Let $L$ be a simple $GL(X)$-module then 
        $$
        L \simeq L(\bl_1,\ldots, \bl_{p-1})
        $$
        for some dominant $\bl_n\in sOb(Ver_p(GL_n)^{\boxtimes k_n})$ for $n=1,\ldots, p-1$.

        \item If $L(\bl_1,\ldots, \bl_{p-1}) \simeq L(\bm_1,\ldots, \bm_{p-1})$ as $GL(X)$-modules then
        $$
        (\bl_1, \ldots ,  \bl_{p-1}) = (\bm_1, \ldots, \bm_{p-1}).
        $$
        
    \end{enumerate}
\end{theorem}

\begin{remark}
    Note that we can construct modules $L(\bl_1,\ldots, \bl_{p-1})$ without the use of the parabolic subgroup $P$. First, let us pull back the representation $\bl_1\boxtimes \ldots\boxtimes \bl_{p-1}$ of $T(X)$ to $B(X)$. Then $L(\bl_1,\ldots, \bl_{p-1})$ is the unique simple quotient of the generalized Verma module
    $$
    M(\bl_1,\ldots, \bl_{p-1}):= Dist(GL(X))\tens{Dist(B(X))} [\bl_1\boxtimes \ldots\boxtimes \bl_{p-1}].
    $$
    
\end{remark}

\section{Different Borels}\label{s_8}

The whole construction of simple representations of $GL(X)$ described in Section \ref{section_simple_reps_of_GL(X)} depends heavily on the choice of Borel subgroup $B(X)\subset GL(X)$, that is, on the choice of ordering of simple components of $X$.

\subsection{Permuting simple summands}
Let $X = \bigoplus_{n=1}^{p-1} L_n^{\oplus k_n} = X_1\oplus \ldots \oplus X_k$, and suppose as before that each $X_i$ is simple and
$$
\text{if } X_i\simeq L_a, X_{i+1}\simeq L_b \text{ then } a<b,
 $$
 that is, $k = k_1+\ldots+ k_{p-1}$ and
$$
\underbrace{X_1\simeq \ldots \simeq X_{k_1}}_{k_1}\simeq L_1,
$$  
$$
\underbrace{X_{k_1+1} \simeq \ldots \simeq  X_{k_1+k_2}}_{k_2} \simeq L_2,
$$
$$
\cdots
$$
$$
\underbrace{X_{k_1+\ldots + k_{p-2}+1}\simeq \ldots \simeq X_k}_{k_{p-1}} \simeq L_{p-1}.
$$

For any permutation $w$ in the symmetric group $S_k$ we can define the Borel subalgebra $\mathfrak b_w(X)\subset \gl(X)$ as
$$
\mathfrak b_w(X) = \bigoplus_{1\le w(i)\le w(j)\le k} X_i\otimes X_j^*.
$$
The standard Borel subalgebra $\mathfrak b(X)$ is the same as $\mathfrak b_1(X)$ for the identity element $1\in S_k$.

For each $n=1,\ldots, p-1$ let $w^{(n)}\in S_{k_n}$ be the unique permutation satisfying the following property: for any $1\le i,j \le k_n$ 
$$
w^{(n)}(i) < w^{(n)}(j) \text{ if and only if } w(k_1+\ldots+k_{n-1}+i) < w(k_1+\ldots+k_{n-1}+j). 
$$
That is $w^{(n)}$ is the permutation controlling the total ordering of elements
$$
w(k_1+\ldots+k_{n-1}+1), w(k_1+\ldots+k_{n-1}+2), \ldots, w(k_1+\ldots+k_{n}).
$$
Then $\mathfrak b_w(X)_{0} = \bigoplus_{n=1}^{p-1} \mathfrak b_{w^{(n)}}(\ck^{\oplus k_n}) \subset \gl(X)_{0} = \bigoplus_{n=1}^{p-1} \gl_{k_n},$
where $\mathfrak b_{w^{(n)}}(\ck^{\oplus k_n})$ is the subalgebra of $\gl_{k_n}$ spanned by matrices $E_{i, j}$ with $1\le i,j\le k_n$ and $w^{(n)}(i)\le w^{(n)}(j)$.

Note that in the classical case we have a description of the subalgebra $\mathfrak b_{w^{(n)}}=\mathfrak b_{w^{(n)}}(\ck^{\oplus k_n})$ in terms of the standard Borel subalgebra $\mathfrak b$ of upper-triangular matrices in $\gl_{k_n}$ and the Weyl group $W\simeq S_{k_n}$ of $GL_{k_n}$:
$$
\mathfrak b_{w^{(n)}} = ad_{\widetilde w^{(n)}}^{-1} (\mathfrak b)
$$
for some representative $\widetilde w^{(n)}$ of $w^{(n)}$ in $GL_{k_n}$. 

Let $B_{w^{(n)}}(\ck^{\oplus k_n}) = Ad_{\widetilde w^{(n)}}^{-1}(B(\ck^{\oplus k_n}))$ be the corresponding Borel subgroup in $GL_{k_n}$.

\begin{definition}
    For each $w\in S_k$ define the Borel subgroup $B_w(X)\subset GL(X)$ with
    $$
    \mathbf{HC}(B_w(X)) = (\prod_{n=1}^{p-1} B_{w^{(n)}}(\ck^{\oplus k_n}), ~\mathfrak b_w(X)).
    $$
\end{definition}

\begin{remark}
    Note that by construction, $T(X)$ is a subgroup of $B_w(X)$ for each $w\in S_k$.
\end{remark}

Our goal is to understand what happens to the highest weight theory for $GL(X)$ described in Section \ref{section_simple_reps_of_GL(X)}  when we use $B_w(X)$ instead of $B(X)$.

As in Section \ref{section_simple_reps_of_GL(X)}, we can construct \textbf{the generalized Verma module}
$$
    M_w(\bl) := Dist(GL(X))\tens{Dist(B_w(X))} \bl.
$$

First, let us define a new dominance order induced by $w\in S_k$ on simple representations of $T(X) = \prod_{i=1}^k GL(X_i)$.

Let $\bl = (\l^{(1)},\ldots, \l^{(k)}) = V_{\l^{(1)}}\boxtimes\ldots\boxtimes V_{\l^{(k)}}$ be a simple representation of $T(X)$. Here for each $i\in \{1,\ldots, k\}$, if $X_i\simeq L_a$, we identify $\l^{(i)}\in \Z^a$, an admissible weight for $GL_a$, with the corresponding simple object  $V_{\l^{(i)}}$ in $Ver_p(GL_a)$.

Permutation $w\in S_k$ defines the dominance order on characters of $T(X)_{0}= \mathbb G_m^k$. We call $\alpha\in \Z^k$ $w$-\textbf{dominant} if
$$
\alpha_{w^{-1}(1)}\ge \ldots\ge \alpha_{w^{-1}(k)}.
$$
We say that $\bl=(\l^{(1)},\ldots, \l^{(k)})$ is $w$\textbf{-dominant} if $(|\l^{(1)}|,\ldots,|\l^{(k)}|)$ is $w$-dominant.

For two $k$-tuples of admissible weights $\bl=(\l^{(1)},\ldots,\l^{(k)})$ and $\bm=(\mu^{(1)},\ldots, \mu^{(k)})$ we write
$$
\bm \preceq_w \bl, \text{ if }
$$

$$
(|\l^{(1)}|,\ldots, |\l^{(k)}|) - (|\mu^{(1)}|,\ldots, |\mu^{(k)}|)\in \Z^k \text{ is } w\text{-dominant}.
$$
The standard dominance order $\preceq$ on representations of $T(X)$ introduced in Section \ref{section_simple_reps_of_GL(X)} coincides with $\preceq_1$ for $1\in S_k$.

\begin{definition}
    A weight $\bl$ is called $w$\textbf{-integrable} if the action of $Dist(GL(X))$ on $L_w(\bl)$ integrates to the action of $GL(X)$.
\end{definition}

\begin{ex}\label{ex_std_integrability}
    A weight $\bl=(\l^{(1)},\ldots,\l^{(k)})$ is $1$-integrable if and only if for every $n=1,\ldots, p-1$
    $$
    |\l^{(i)}|\ge |\l^{(j)}|
    $$
    for each $k_1+\ldots+k_{n-1}+1\le i\le j\le k_1+\ldots +k_n$.
\end{ex}

With respect to partial order $\preceq_w$ the generalized Verma module $M_w(\bl)$ is a highest weight module with highest weight $\bl$. It has the unique simple quotient $L_w(\bl)$.
\begin{itemize}
    \item Every simple $Dist(GL(X))$-module admitting a highest weight with respect to the partial order $\preceq_w$, in particular, every simple  $GL(X)$-module  must be isomorphic to $L_w(\bl)$ for some simple $\bl\in \Rep_{Ver_p}(T(X),\varepsilon)$.

    \item Clearly, $L_w(\bl)\simeq L_w(\bm)$ if and only if $\bl\simeq \bm$.

    \item If $\bl$ is $w$-integrable then $L_w(\bl)$ must be isomorphic to $L(\bm)$ for some $1$-integrable weight $\bm$. Let us denote this weight by $\bl^w$.
\end{itemize}
\begin{definition}
    For any $w$-integrable weight $\bl$ the weight $\bl^w$ is defined as a $1$-integrable weight for which
    $$
    L_w(\bl)\simeq L(\bl^w).
    $$
\end{definition}    

In the remainder of this paper we will try to answer the following:

\begin{question}\label{quest_1}
    Given a permutation $w\in S_k$ and a $w$-integrable weight $\bl$, how can we describe $\bl^w$?
\end{question}

\subsection{Conjugate Borels}\label{s_conj_Borels}
\begin{definition}
    For $X=\bigoplus_{n=1}^{p-1} L_n^{\oplus k_n}$ with $\sum_{n=1}^{p-1}k_n = k$, put $W=S_k$, and define the \textbf{classical Weyl group} $W_{0}$ to be the parabolic subgroup $S_{k_1}\times \ldots\times S_{k_n}\subset W=S_k$, i.e. the Weyl group of the classical subgroup $GL(X)_{0}=\prod_{n=1}^{p-1} GL_{k_n}$. 

    One can embed $W_{0}$ into $GL(X)(\ck)=GL(X)_{0}(\ck)$ as permutation matrices.
\end{definition}

The group of $\ck$-points of any affine group scheme $G$ in $Ver_p$ acts on $\co(G)$ by conjugation as follows. Let $g\in G(\ck)$ be a map
$$
g: \co(G)\to \ck,
$$
and let $\Delta: \co(G)\to \co(G)\otimes \co(G)$ be the comultiplication map.
We define
$$
Ad_g = (g^{-1}\otimes id_{\co(G)}\otimes g) \circ  \Delta^2: \co(G)\to \co(G)\otimes \co(G)\otimes \co(G)\to \co(G).
$$

\begin{definition}
    Let $g$ be an element of  $GL(X)(\ck) = G(\ck)$ and let $H\subset G$ be a closed subgroup scheme with
    $$
    \co(H)= \co(G)/J_H
    $$
    for an ideal $J_H\subset \co(G)$. Define
    $$
    H^g = Ad_g(H) \subset GL(X)
    $$
    be the subgroup cut out by the ideal
    $$
    J_H^g=  Ad_g(J_H).
    $$
    \end{definition}

The adjoint action of $G(\ck)$ on $\co(G)$ induces the adjoint action on $\g =Lie(G)$.

When $w\in W_{0}\subset GL(X)(\ck)$, we get
$$
Ad_w(X_i\otimes X_j^*) = X_{w(i)}\otimes X_{w(j)}^*\subset \gl(X),
$$
so
$$
Ad_w^{-1}(\mathfrak b(X))= \mathfrak b_w(X).
$$

Moreover, for $w\in W_{0}$ we have
$$
B_w(X)_{0} = \prod_{n=1}^{p-1} B_{w^{(n)}}(\ck^{\oplus k_n}) = Ad_w^{-1} (\prod_{n=1}^{p-1} B(\ck^{\oplus k_n})) \subset  \prod_{n=1}^{p-1} GL_{k_n},
$$
where $w$ corresponds to $(w^{(1)},\ldots, w^{(p-1)})$ under the isomorphism $S_{k_1}\times \ldots\times S_{k_n}\simeq W_{0}$. 

We get $B_w(X) = Ad_w^{-1}(B(X))$ if $w\in W_{0}\subset GL(X)(\ck)$.

\begin{theorem} \label{thm_conjugate_Borels}
Let $w\in W_{0}$ and let $\bl=(\l^{(1)},\ldots, \l^{(k)})$ be a simple representation of $T(X)$. Then 
$$
w\cdot \bl :=(\l^{(w^{-1}(1))},\ldots, \l^{(w^{-1}(k))})
$$
is also a simple representation of $T(X)$ (it lies in the same category since $w\in W_{0}$, so $X_i\simeq X_{w(i)}$). This defines an action of $W_{0}$ on the set of simple representations of $T(X)$.
\begin{enumerate}
    \item If $\bl$ is $w$-integrable then the set of weights of $L_w(\bl)$ is $W_{0}$-stable.
    \item \label{conjug.Borels} We have
    $$
    \bl^w = w\cdot \bl,
    $$
    that is, $L_w(\bl)\simeq L(w\cdot \bl)$.

    \item Consequently, $\bl$ is $w$-integrable if and only if for each $n=1,\ldots, p-1$
    $$
    |\l^{(w^{-1}(i))}|\ge |\l^{(w^{-1}(j))}|
    $$
    for each $k_1+\ldots+k_{n-1}+1\le i\le j\le k_1+\ldots+k_n$.
\end{enumerate}
    
\end{theorem} 
\begin{proof}
    \begin{enumerate}
        \item The action of $W_{0}\subset GL(X)(\ck)$ on any $GL(X)$-module $M$ sends a weight $\bl\subset M|_{T(X)}$ to a weight $w\cdot \bl$.

        \item Let $\bm$ be any weight of $L_w(\bl)$ then
        $$
        \bm \preceq_w \bl,
        $$
        i.e.
        $$
        |\l^{(w^{-1}(1))}|-|\mu^{(w^{-1}(1))}|\ge |\l^{(w^{-1}(2))}|-|\mu^{(w^{-1}(2))}|\ge\ldots\ge |\l^{(w^{-1}(k))}|-|\mu^{(w^{-1}(k))}|.
        $$
        This condition is equivalent to saying that
        $$
        w\cdot \bm \preceq_1 w\cdot \bl.
        $$
        Thus $w\cdot \bl$ is the highest weight in $L_w(\bl)$ with respect to the standard dominance order. Thus, $L_w(\bl)\simeq L(w\cdot \bl)$.

        \item This condition follows immediately from part (\ref{conjug.Borels}) and the integrability condition for the standard Borel subgroup (see Example \ref{ex_std_integrability})
    \end{enumerate}
\end{proof}

\subsection{Odd reflections}\label{s_odd_refl}

Similarly to the case of the super group $GL(m|n)$, not all Borel subgroups $B_w(X)$ in $GL(X)$ are conjugate to each other.  

The group $W=S_k$ is generated by $W_{0}=S_{k_1}\times\ldots\times S_{k_{p-1}}$ and \textbf{simple odd reflections} $s_j$ for $j=1,\ldots, p-2$, where $s_j$ is the elementary transposition 
$$
s_j = (a_j,~a_j+1),
$$
where $a_j=k_1+\ldots+k_j$.

That is, $W_{0}$ is responsible for permuting isomorphic simple summands of $X$, while $s_j$ permutes non-isomorphic simple summands.

It is enough, therefore, to answer Question \ref{quest_1} for simple odd reflections, i.e. for $w=s_j$. 

Consider the subgroup 
$$
H_j = GL(X_1)\times\ldots \times GL(X_{a_j-1}) \times [ GL(X_{a_j}\oplus X_{a_j+1}) ] \times GL(X_{a_j+2})\times\ldots\times GL(X_k) \subset GL(X).
$$
Suppose $X_{a_j}\simeq L_m$ and $X_{a_{j}+1}\simeq L_n$. Since $s_j$ is an ``odd" reflection, we know that $m<n$.

Simple representations of $H_j$ are of the form
$$
\l^{(1)}\boxtimes\ldots\boxtimes \l^{(a_j-1)}\boxtimes L(\bm) \boxtimes \l^{(a_j+2)}\boxtimes\ldots \boxtimes \l^{(k)},
$$
where $\l^{(i)}$ denotes the corresponding objects $V_{\l^{(i)}}$ in $Ver_p(GL_a)$ when $X_i\simeq L_a$, and $\bm=(\mu_1, \mu_2)$, where $\mu_1$ and $\mu_2$ are admissible weights for $GL_m$ and $GL_n$ respectively. Representation $L(\bm)$ is the simple representation of $GL(X_{a_j}\oplus X_{a_j+1}) = GL(L_m\oplus L_n)$ constructed in Section \ref{section_simple_reps_of_GL(X)}.

\begin{theorem}\label{thm_odd_refl}
Let $a=k_1+\ldots+k_j$ for some $j=1,\ldots, p-1$, let $s \in S_k$ be the transposition $(a, ~a+1)$.
    Suppose as before that $X_{a}\simeq L_m, X_{a+1}\simeq L_n$ with $m<n$.
    
    Let $\bl = (\l^{(1)},\ldots, \l^{(k)})$ be a simple, $s$-integrable representation of $T(X)$,  and let $\bm
    =  (\l^{(a)}, \l^{(a+ 1)})$. 
    
    Let $\sigma=(1~2)\in S_2$, suppose $\bn = (\nu_1,\nu_2)$ is such that
    $$
    \bm^\sigma = \bn,
 $$
 That is, we have an isomorphism of representations of $GL(L_m\oplus L_n)$:
 $$
 L_\sigma(\bm) \simeq L(\bn).
 $$

Then $$\bl^{s} = (\l^{(1)},\ldots, \l^{(a-1)}, \nu_1, \nu_2, \l^{(a+2)},\ldots \l^{(k)}).$$

Therefore, the description of the bijection $\bl\leftrightarrow \bl^s$ for a simple odd reflection $s$ boils down to the description of the bijection $\bm\leftrightarrow \bm^\sigma$ on weights of $GL(L_m\oplus L_n)$ for $\sigma = (1~2)$.
 
\end{theorem}
\begin{proof}
    First, let us note that the generalized Verma module with highest weight $\bm$ for $GL(L_m\oplus L_n)$ coincides with the corresponding Kac module, which we will denote by $K(\bm)$. To avoid confusion, we will reserve the notation $M(\bl)$ for the generalized Verma module for $GL(X)$.

    For each $i=1,\ldots, k$, denote the subobject $X_i\otimes X_{j}^*\subset \gl(X)$ by $\mathbf{E}_{i,j}$. Given any $GL(X)$-module $M$ and any subobject $\mathbf v\subset M$ denote by $\mathbf{E}_{i,j}\cdot \mathbf{v}$ the image of $\mathbf{E}_{i,j}\otimes \mathbf{v}$ under the action map
    $$
    \gl(X)\otimes M\to M.
    $$
    
    Note that $\mathbf{E}_{a,a}, \mathbf{E}_{a,a+1}, \mathbf{E}_{a+1,a}$, and $\mathbf{E}_{a+1,a+1}$ are naturally subobjects of $\gl(X_a\oplus X_{a+1}) = \gl(L_m\oplus L_n)$.
    As a representation of $T(L_m\oplus L_n) = GL(L_m)\times GL(L_n)$ 
    $$
    K_{\sigma}(\bm) \simeq Dist(\mathfrak n_{\sigma}^-(X_a\oplus X_{a+1}))\otimes \bm  = S(\mathbf{E_{a,a+1}})\otimes \bm.
    $$

    By construction, 
$$
\l^{(1)}\boxtimes\ldots\boxtimes \l^{(a_j-1)}\boxtimes K_{\sigma}(\bm) \boxtimes \l^{(a_j+2)}\boxtimes\ldots \boxtimes \l^{(k)}
$$
is a submodule of $\Res_{H_j}^{GL(X)}M_s(\bl)$, containing the $s$-highest weight $\bl$. For any weight $\ba$ of $K_{\sigma}(\bm)$, let us abuse the notation and denote the corresponding weight $$\l^{(1)}\boxtimes\ldots\boxtimes \l^{(a_j-1)}\boxtimes \ba \boxtimes \l^{(a_j+2)}\boxtimes\ldots \boxtimes \l^{(k)}$$
of $M_s(\bl)$ simply by $\ba$.

Note that, being the maximal weight with respect to partial order $\preceq_s$, the subobject $\bl$ of $M_s(\bl)$ satisfies:
$$
\mathbf{E}_{i,i+1}\cdot \bl = 0 \text{ for all } i \neq a-1, a, a+1, 
$$
$$
\mathbf{E}_{a-1,a+1}\cdot \bl = 0,
$$
$$
\mathbf{E}_{a,a+2}\cdot \bl = 0,
$$
$$
\mathbf{E}_{a+1,a}\cdot \bl = 0.
$$

Let $N_\sigma(\bm)$ be the maximal submodule of $K_\sigma(\bm)$. Let $\ba=(\alpha_1,\alpha_2)$ be some maximal weight of $N_\sigma(\bm)$ with respect to partial order $\preceq_{\sigma}$. It means that
$$
\mathbf{E}_{a+1,a}\cdot \ba =0.
$$

Since $\ba\subset \mathbf{E}^k_{a,a+1}\cdot \bl$ for some $k$, we get
$$
\mathbf{E}_{i,i+1}\cdot \ba \subset \mathbf{E}_{i,i+1} \cdot \mathbf{E}^k_{a,a+1} \cdot \bl = \mathbf{E}^k_{a,a+1}\cdot \mathbf{E}_{i,i+1}\cdot \bl = 0, \text{ for all } i\neq a-1,a,a+1.
$$
Moreover,
$$
\mathbf{E}_{a-1,a+1}\cdot \ba\subset \mathbf{E}_{a-1,a+1}\cdot \mathbf{E}^k_{a,a+1}\cdot \bl =  \mathbf{E}^k_{a,a+1}\cdot \mathbf{E}_{a-1,a+1}\cdot \bl =0,
$$
$$
\mathbf{E}_{a,a+2}\cdot \ba\subset \mathbf{E}_{a,a+2}\cdot \mathbf{E}^k_{a,a+1}\cdot \bl = \mathbf{E}_{a,a+1}\cdot \mathbf{E}_{a,a+2}\cdot \bl =0,
$$

Therefore, $\ba\subset M_s(\bl)$ must lie in the maximal submodule $J_s(\bl)$. We deduce that $N_\sigma(\bm)\subset J_s(\bl)$. 

Therefore, the $H_j$-module generated by the highest weight (with respect to $\preceq_s$) subobject $\bl\subset L_s(\bl)$ is isomorphic to
$$
\l^{(1)}\boxtimes\ldots\boxtimes \l^{(a_j-1)}\boxtimes L_{\sigma}(\bm) \boxtimes \l^{(a_j+2)}\boxtimes\ldots \boxtimes \l^{(k)} \simeq \l^{(1)}\boxtimes\ldots\boxtimes \l^{(a_j-1)}\boxtimes L(\bn) \boxtimes \l^{(a_j+2)}\boxtimes\ldots \boxtimes \l^{(k)}.
$$
   Consequently,  $\l^{(1)}\boxtimes\ldots\boxtimes \l^{(a_j-1)}\boxtimes \bn \boxtimes \l^{(a_j+2)}\boxtimes\ldots \boxtimes \l^{(k)}$ is a weight in $L_s(\bl)$. Moreover, it is the highest weight with respect to the standard dominance order $\preceq=\preceq_1$.

   We deduce that $L_s(\bl)\simeq L(\l^{(1)}\boxtimes\ldots\boxtimes \l^{(a_j-1)}\boxtimes \bn \boxtimes \l^{(a_j+2)}\boxtimes\ldots \boxtimes \l^{(k)}).$
\end{proof}

\section{Representations of \texorpdfstring{$GL(L_m|L_n)$}{GL(L m|L n)}} \label{s_9}
In light of Theorem \ref{thm_odd_refl}, we will now focus on studying representations of $GL(L_m\oplus L_r)$ with $1\le m<r\le p-1$. For the rest of this section put $G=GL(L_m\oplus L_r), B=B(L_m\oplus L_r), T=T(L_m\oplus L_r), N^\pm = N^\pm(L_m\oplus L_r)$.

\subsection{Super group scheme convention}\label{sec_super}

Let us adopt a useful convention. Put $n = p-r$ so that
$$
L_n\simeq L_{p-1}\otimes L_r.
$$
We will treat $G=GL(L_m\oplus L_r)$ as a super group scheme $GL(L_m|L_n)$, where $L_n$ is now ``odd". The torus $T=T(L_m\oplus L_r)=GL(L_m)\times GL(L_r)$ is isomorphic to $GL(L_m)\times GL(L_n)$.
\begin{remark}
  Note that even though the group schemes $GL(L_r)$ and $GL(L_n)$ are isomorphic, the natural homomorphisms from $\pi_{Ver_p}$ to them are slightly different. 

    With respect to the natural map $\varepsilon: \pi_{Ver_p}\to  GL(L_n)$ we have
    $$
    \Rep_{Ver_p}(GL_n, \varepsilon)\simeq Ver_p(GL_n).
    $$

    Whereas if we take the map $\varepsilon': \pi_{Ver_p}\to GL(L_r)\xrightarrow{\sim} GL(L_n)$, we get a twisted category: $\Rep_{Ver_p}(GL_n, \varepsilon')$, which is isomorphic to the tensor subcategory of $Ver_p(GL_n)\boxtimes sVec$ generated by $V\boxtimes I$, where $V$ is the $n$-dimensional object in $Ver_p(GL_n)$ coming from the tautological representation of $GL_n$ and $I$ is the $(0|1)$-dimensional super vector space.
  \end{remark}  

Using level-rank duality (\ref{level_rank})
$$
Ver_p(GL_r)\boxtimes sVec \xrightarrow{\sim }Ver_p(GL_n)\boxtimes sVec,
$$
we will regard the highest $GL(L_m|L_n)$-weights as simple objects in the tensor subcategory of $Ver_p(GL_m)\boxtimes Ver_p(GL_n)\boxtimes sVec$ generated by $V^{(m)}\boxtimes \on\boxtimes \on$ and $\on\boxtimes V^{(n)}\boxtimes I$. That is, 
to each pair of admissible weights $\mu\in \Z^m, \nu\in \Z^n$ we assign a Kac module
$$
K(\mu|\nu) := Dist(G)\tens{Dist(B)} V^{(m)}_\mu \boxtimes V^{(n)}_\nu\boxtimes I^{\otimes |\nu|};
$$
and its simple quotient $L(\mu|\nu)$.

\textbf{Notations:}
    \begin{itemize}
        \item  Let us abuse the notation and denote the $T$-module  $V^{(m)}_\mu \boxtimes V^{(n)}_\nu\boxtimes I^{\otimes |\nu|}$ simply by $(\mu|\nu)$.

        \item For an admissible weight $\mu=(\mu_1,\ldots, \mu_m)$ let us denote by $\mu^*$ both the weight $(-\mu_n,\ldots, -\mu_1)$ and the object $(V^{(m)}_\mu)^* = V^{(m)}_{\mu^*}$.

        \item For any partition $\l$ denote by $\l^t$ the transpose of $\l$. 

        \item  When it is suitable, we will treat $(\mu|\nu)$ as a vector $(\mu_1,\ldots,\mu_m|\nu_1,\ldots,\nu_n)$ in $\Z^{m+n}$, so that the sum of two weights is defined.  To distinguish from the classical situation we will denote $\Z^{m+n}$  by $\Z^{m|n}$.

    \item We will henceforth refer to all elements of $\Z^{m|n}$ as weights, and weights $(\mu|\nu)$ with both $\mu$ and $\nu$ admissible will be called admissible weights.

    \item Fix standard bases $\epsilon_1,\ldots, \epsilon_m$ and $\delta_1,\ldots, \delta_n$ of $\Z^m$ and $\Z^n$ correspondingly, so that $\Z^{m|n}= \bigoplus_{i=1}^m\Z\epsilon_i\oplus\bigoplus_{j=1}^n\Z\delta_j$.

    \item Weights $\ba_{i,j} = \epsilon_i - \delta_j$ will be referred to as odd positive roots. We put $\bb = \sum_{i\le m}\sum_{j\le n} \ba_{i,j} = (n,\ldots, n|-m,\ldots, -m)$. 

    \item Define  $$
 2\br^{(m|n)}:= 2(\rho^{(m)}|\rho^{(n)})-\bb=
 $$
 $$
 =(m-1-n, m-3-n,\ldots,1-m-n|n-1+m, n-3+m\ldots, 1-n+m).
 $$
 The weight $2\br^{(m|n)}$ is the sum of all  ``even'' positive roots minus the sum of all odd positive roots.

    \item Define a symmetric form $\langle -,-\rangle $ on $\Z^{m|n}$ via
    $$
    \langle\epsilon_i, \epsilon_j\rangle=\mathrm{\delta}_{i,j}, ~\langle \epsilon_i, \delta_j\rangle = 0,~\langle \delta_i, \delta_j\rangle = -
    \mathrm{\delta}_{i,j}.
    $$
    \end{itemize}

Recall that as a $Dist(N^-)$-$Dist(T)$-bimodule $$K(\mu|\nu)\simeq Dist(N^-)\otimes (\mu|\nu) = S(\mathfrak n^-)\otimes (\mu|\nu)  =
$$
$$
=S((V^{(m)})^*\boxtimes V^{(n)}\boxtimes I)\otimes (\mu|\nu) \simeq \bigoplus_{\l} (\l^*|\l^t)\otimes (\mu|\nu),
$$
where the sum is taken over all partitions $\l=(\l_1
, \ldots, \l_k)$ with $\l_1\ge\ldots\ge\l_k\ge 0$, such that $\l$ is an admissible weight for $GL_m$ and $\l^t$ is an admissible weight for $GL_n$. That is, we ask that
\begin{equation}\label{cond_on_lambda}
    k\le m \text{ and } \l_1\le n.
\end{equation}
Note that our condition that $m<r = p-n$ ensures that $\l$ is an admissible weight for $GL_m$ as $\l_1-\l_m \le \l_1\le n<p-m$; and that $\l^t$ is an admissible weight for $GL_n$ as $\l_1^t-\l_n^t\le \l_1^t = k \le m<p-n$.

This isomorphism induces a natural $\Z$-grading on $K(\mu|\nu)$ with $K(\mu|\nu)^q = S^q(\mathfrak n^-)\otimes (\mu|\nu)$. Moreover, it is easy to deduce from condition (\ref{cond_on_lambda}) that the top nonzero graded component is
$$
K(\mu|\nu) ^{m\cdot n}= S^{m\cdot n}(\mathfrak n^-)\otimes (\mu|\nu) = (det^{-n}\boxtimes det^m)\otimes (\mu|\nu).
$$
So $(\widetilde \mu|\widetilde \nu) := (\mu|\nu)-\bb = (\mu_1-n,\ldots, \mu_m-n|\nu_1+m,\ldots,\nu_n+m)$ is the lowest weight of $K(\mu|\nu)$.

The question of describing the bijection $(\mu|\nu)^\sigma \leftrightarrow (\mu|\nu)$ for $\sigma=(1~2)\in S_2$ is equivalent to the question of determining the lowest weight in $L(\mu|\nu)$. More precisely, if $(\mu'|\nu')$ is the lowest weight in $L(\mu|
\nu)$ with respect to the standard dominance order (i.e. $|\mu'|$ is minimal among $\{|\alpha|,~ \text{with } (\alpha|\beta)$ a weight of $L(\mu|\nu)\}$) then
$$
L_\sigma(\mu'|\nu') \simeq L(\mu|\nu).
$$

By necessity every submodule in $K(\mu|\nu)$ contains the lowest weight subobject $(\widetilde \mu|\widetilde \nu)$. In particular, it must have a simple submodule $L(\alpha, \beta)$ (its socle), whose lowest weight is $(\widetilde\mu, \widetilde\nu)$. Thus,
$$
L_\sigma(\widetilde\mu|\widetilde\nu) \simeq L(\alpha, \beta),
$$
i.e.
$$
(\widetilde\mu|\widetilde\nu)^\sigma   = (\alpha|\beta).
$$

The rest of this section will be dedicated to determining the socle of $K(\mu|\nu)$.

\subsection{Highest weight category} \label{s_hw_cat}
The statements and proofs in this section are almost identical to those made about $\gl(m|n)$-modules in \cite{Z96}.

Let $m+n<p$.

 Consider the locally finite poset $\Lambda$ of all weights $\bl = (\mu|\nu)$ with $\mu=(\mu_1,\ldots, \mu_m)\in \Z^m, \nu=(\nu_1,\ldots, \nu_n)\in \Z^n$ satisfying 
 $$
 \mu_1\ge\ldots\ge\mu_m, ~\nu_1\ge\ldots\ge \nu_m, \text{ and}
 $$
 $$
 \mu_1\le p-m, ~\nu_1\le p-n;
 $$
 with the partial order
 $$
 (\alpha|\beta) \preceq (\mu|\nu) \text{ if } |\alpha|+|\beta| = |\mu|+|\nu| \text{ and } |\alpha|\le|\mu|. 
  $$

\begin{remark}
    We identify elements of $\Lambda$ with simple representations of the torus $T\subset G=GL(L_m|L_n)$.
\end{remark}

The category $\cc = \Rep_{Ver_p}(G, \varepsilon)$ has simple objects $L(\bl)=L(\mu|\nu)$ labeled by elements $\bl=(\mu|\nu)$ of poset $\Lambda$. The Kac modules $K(\bl) = K(\mu|\nu)$ play the role of standard objects in $\cc$. We have for each $\bl\in \Lambda$ a surjection
$$
p_\bl: K(\bl)\to L(\bl)
$$
satisfying $[\Ker p_\bl: L(\ba)]\neq 0$ only if $\ba\prec \bl$. Note that $\Hom_G(K(\bl), K(\ba)) = \Hom_B(\bl, K(\ba))$ and thus it is nonzero only if $\bl\preceq \ba$, whereas $\mathrm{End}_G(K(\bl))=\ck$. 

\begin{theorem}\label{thm_enough_proj}
    The category $\cc$ has enough projectives.
\end{theorem}
\begin{proof}
    For any $\bl\in \Lambda$ define 
    $$
    \widetilde P(\bl) := Dist(G)\tens{Dist(T)} \bl.
    $$
    The triangular decomposition for $G$ yields the following isomorphism of objects in $Ver_p$:
    $$
    \widetilde P(\bl) \simeq Dist(N^-)\otimes Dist(N^+)\otimes \bl\simeq S(\mathfrak n^-)\otimes S(\mathfrak n^+)\otimes \bl.
    $$
    Both $S(\mathfrak n^-)$ and $S(\mathfrak n^+)$ have finite length in $Ver_p$ (since in our case $\mathfrak n^\pm\simeq L_m\otimes L_{p-n}$ does not contain a copy of $\on$), which makes $\widetilde P(\bl)$ an integrable $Dist(G)$-module.

    By Frobenius reciprocity we get
    $$
    \Hom_G(\widetilde P(\bl), -) = \Hom_T(\bl, \Res^G_T(-)).
    $$
    As all representations of $T$ are semisimple, we get that $\Hom_G(\widetilde P(\bl),-)$ is exact, and thus $\widetilde P(\bl)$ is projective.

    Moreover, there is a natural surjective homomorphism $\widetilde\pi_\bl: \widetilde P(\bl)\to K(\bl)$, and thus we get a surjection $\widetilde P(\bl)\to L(\bl)$. 
\end{proof}

Let us denote by $P(\bl)$ the indecomposable summand of $\widetilde P(\bl)$, such that the homomorphism $\widetilde \pi_\bl$ factors through $P(\bl)$:
\[ \xymatrix{ \widetilde P(\bl) \ar[rr]^{~\widetilde\pi_\bl} \ar@{->>}[rd] &&K(\bl)\\ &P(\bl)\ar@{-->}[ur]_{~\pi_\bl}} \]

Clearly, $P(\bl)$ is the projective cover of $L(\bl)$. 

Note that $\dim\Hom_G(\widetilde P(\bl), K(\bl)) = \dim\Hom_T(\bl, \Res^G_T(K(\bl))) = 1$, so we deduce that the multiplicity of $P(\bl)$ in $\widetilde P(\bl)$ is one. We get
\begin{equation}\label{eq_decomp_of_proj}
\widetilde P(\bl) = P(\bl)\oplus\bigoplus_{\ba\neq \bl} m_{\bl,\ba} P(\ba).
\end{equation}

As $Dist(G)$ is a free right $Dist(B)$-module, the induction functor $Dist(G)\tens{Dist(B)} (-) $  is exact. Since, as a $Dist(B)$-module, $Dist(B)\tens{Dist(T)} \bl$ has a filtration by representations $\ba$ of $Dist(B)$ on which $Dist(N^+)$ acts trivially (i.e.  those pulled back from a $T$-module $\ba$), we get that
$$
\widetilde P(\bl) = Dist(G)\tens{Dist (B)} (Dist(B)\tens{Dist(T)} \bl)
$$
has a filtration by Kac modules $K(\ba)$ with $\ba$ a weight of $Dist(B)\tens{Dist(T)} \bl = S(\mathfrak n^+)\otimes \bl$.

By the above, the multiplicity $d_{\bl,\ba} = [\widetilde P(\bl): K(\ba)]$ is equal to the multiplicity of $\ba$ in the (semisimple finite length) $T$-module $S(\mathfrak n^+)\otimes \bl$, i.e.
\begin{equation}\label{eq_mult_in_proj}
[\widetilde P(\bl): K(\ba)] = \dim\Hom_T( S(\mathfrak n^+)\otimes \bl, \ba) = \dim \Hom_T(\bl, S(\mathfrak n^+)^*\otimes \ba)
=\end{equation}
\[
=\dim\Hom_T(\bl, S(\mathfrak n^-)\otimes \ba) = \dim\Hom_T(\bl, \text{Res}^G_T K(\ba)) = \dim \Hom_G(\widetilde P(\bl), K(\ba)).\]

Consequently, $P(\bl)$ has a standard filtration by $K(\ba)$ satisfying
\begin{enumerate}
    \item $[P(\bl):K(\bl)] = 1$,
    \item $[P(\bl):K(\ba)]\neq 0$ only if $\bl\preceq \ba$,
    \item $[P(\bl):K(\ba)] = \dim\Hom_G(P(\bl), K(\ba)).$
\end{enumerate}
Let us prove the last part by induction. We can assume that $\ba$ is a weight of $S(\mathfrak n^+)\otimes \bl$ (which is equivalent to saying that $\bl$ is a weight of $S(\mathfrak n^-)\otimes \ba = \text{Res}^G_T K(\ba)$, since otherwise the statement holds trivially, as
$$
[P(\bl):K(\ba)] = 0 = \dim \Hom_G(P(\bl), K(\ba)) = [K(\ba):L(\bl)].
$$
By (\ref{eq_decomp_of_proj}) and (\ref{eq_mult_in_proj}), we get
\begin{equation}\label{eq_induct_proj}
[\widetilde P(\bl):K(\ba)] = [P(\bl):K(\ba)] + \sum m_{\bl,\bm} [P(\bm): K(\ba)] = 
\end{equation}
$$
=\dim\Hom(P(\bl),K(\ba))+\sum_{\bm\neq \bl} m_{\bl,\bm} \cdot \dim\Hom(P(\bm), K(\ba))
$$
Clearly, $m_{\bl, \bm}=[\widetilde P(\bl):P(\bm)]\le d_{\bl, \bm}=[\widetilde P(\bl):K(\bm)]$ and thus it is nonzero only if $\bm$ is a weight in $S(\mathfrak n^+)\otimes \bl$. So, only the weights $\bm$ in the range $\bl\preceq \bm\preceq \ba$ are contributing non-trivially to the sum above.

Let us fix $\ba$ and assume that 
$$
[P(\bm):K(\ba)] = \dim\Hom_G(P(\bm), K(\ba))
$$
for all weights $\bm$ 
satisfying 
$$
\bl \prec \bm \preceq \ba,
$$
with the base case $\bm = \ba$, for which we have
$$
[P(\ba):K(\ba)] = 1 = \dim\Hom_G(P(\ba), K(\ba)).
$$
We can now prove it for $\bm = \bl$ using equation (\ref{eq_induct_proj}).

\begin{corollary}
    $\cc$ is a \textbf{highest weight category} in the sense of Cline, Parshall, Scott \cite{CPS88}.
\end{corollary}

\begin{corollary}[BGG reciprocity]\label{cor_bgg}
    For any $\bl, \ba\in \Lambda$ we get
    $$
    [P(\bl):K(\ba)] = [K(\ba):L(\bl)].
    $$
\end{corollary}
\begin{proof}
    We get
    $$
    [K(\ba):L(\bl)] = \dim\Hom_G(P(\bl),K(\ba)) = [P(\bl):K(\ba)].
    $$
\end{proof}
\begin{corollary}\label{cor_exts_standards}
    We have 
    $$
    Ext^i(K(\bl), K(\bm))\neq 0 \text{ for some } i>0 
    $$
    only if $\bl\prec \bm$.
\end{corollary}
\begin{proof}
Recall that
    $$
    [P(\bl):K(\bm)]\neq 0 ~\Rightarrow~\bl\preceq \bm,
    $$
    with $[P(\bl):K(\bl)]=1$.
    
The kernel $N_1$ of the projection $P(\bl)\to K(\bl)$ has standard filtration by $K(\ba)$'s with $\ba\in A_1,$ where $A_1$ is some finite multiset of weights satisfying $\ba\succ\bl$ for all $\ba\in A_1$. $N_1$ can then be covered by $P_1=\bigoplus_{\ba\in A_1} P(\ba)$, and the kernel $N_2$ of the map $P_1\to N_1$ has standard filtration by $K(\ba)$'s with $\ba\in A_2$ for some finite multiset of weights $A_2$. For each weight $\ba_2$ in $A_2$ there must exist a weight $\ba_1\in A_1$ so that $\ba_2\succeq \ba_1$, and consequently $\ba_2\succ\bl$. We then put $P_2=\bigoplus_{\ba\in A_2} P(\ba)$, and so on. We obtain a projective resolution 
   $$
    \ldots\to P_3\to P_2\to P_1\to P_0=P(\bl)\to K(\bl)\to 0
    $$
of $K(\bl)$, where the positive degree terms  consist of direct sums of $P(\ba)$'s with $\bl\prec \ba$. The result then follows immediately from the fact that
$$\Hom_G(P(\ba),K(\bm))\neq 0 \text{ only if } [K(\bm):L(\ba)]\neq 0, 
$$ and thus $\bm\succeq \ba\succ \bl$.
\end{proof}

\begin{corollary}\label{cor_irred_kac_is_proj}
    Suppose $K(\bl)$ is irreducible, i.e. $K(\bl)= L(\bl)$. Then it is also projective: $K(\bl)= P(\bl)$.
\end{corollary}
\begin{proof}
    We have $[P(\bl): K(\ba)] = [K(\ba):L(\bl)]$. Since $L(\bl)=K(\bl)$, the right-hand-side multiplicity is nonzero only if $\ba = \bl$ (because all Kac modules are $\mathbb Z$-graded and have the same size = number of nonzero graded components $ = m\cdot n$). So, $P(\bl)=K(\bl)$.
\end{proof}

The following statement is completely analogous to the case of the super group $GL(m|n)$ (the proof in the case of quasi-reductive super groups can be found in (\cite{S11}, Section 9)).
\begin{theorem}
    Projective objects $P(\bl)$ are also injective. Moreover, $P(\bl)$ is the injective hull of $L(\bl)$.
\end{theorem}
\begin{proof}
Define $\mathrm{Ind} (\bl) = \widetilde P(\bl):=Dist(G)\tens{Dist(T)}\bl$ (the projective module defined in Theorem \ref{thm_enough_proj}) and
$$
\widetilde I(\bl) =\mathrm{Coind} (\bl) := \underline{\Hom}_{T}(Dist(G), \bl),
$$
(where $\underline{\Hom}_{T}$ denotes the internal hom in the category of $Dist(T)$-modules) the right adjoint functor to restriction applied to $\bl$. By definition (and Frobenius reciprocity),
$$
\Hom_G(-,\widetilde I(\bl))=\Hom_T(\Res(-),\bl),
$$
so $\widetilde I(\bl)$ is injective.

Let $I(\bl)$ denote the injective hull of $L(\bl)$. Then $I(\bl)$ is an indecomposable summand of $\widetilde I(\bl)$.

Let us first prove that $\widetilde P(\bl)\simeq \widetilde I(\bl)$.  

 Put  $S=S(\mathfrak n^+\oplus \mathfrak n^-)$, so that $Dist(G)\simeq S\otimes Dist(T)$ as a right $Dist(T)$-module (similarly $Dist(G)\simeq Dist(T)\otimes S$ as a left $Dist(T)$-module). 

We have the highest degree component
$$
S^{2mn}(\mathfrak n^+\oplus \mathfrak n^-)\simeq \on,
$$
which is a trivial $Dist(T)$-module (being trivial also implies that it lies in the centralizer of $Dist(T)$ in $Dist(G)$). Let $\pi: S\to S^{2mn}(\mathfrak n^+\oplus\mathfrak n^-)=\on$ be the natural projection.

We claim that $Dist(G)$ is a Frobenius extension of $Dist(T)$. Indeed $Dist(G)$ is a free finitely generated right $Dist(T)$-module and the map $\pi$ induces a homomorphism
$$
\widetilde \pi=\pi\otimes id_{Dist(T)}:Dist(G)\simeq S\otimes Dist(T)\to Dist(T)
$$
of $Dist(T)$-$Dist(T)$-bimodules (note that it is automatically a homomorphism of right $Dist(T)$-modules; and it commutes with the left multiplication by $Dist(T)$ because  $S^{2mn}(\mathfrak n^+\oplus \mathfrak n^-)$ lies in the centralizer of $Dist(T)$). 

This map defines a nondegenerate\footnote{It comes from the nondegenerate pairing on $S(\mathfrak n^+\oplus \mathfrak n^-)$ defined via $$S^k(\mathfrak n^+\oplus \mathfrak n^-)\otimes S^{2mn-k}(\mathfrak n^+\oplus \mathfrak n^-)\to S^{2mn}(\mathfrak n^+\oplus \mathfrak n^-)=\on$$} pairing  
$$
\langle-,-\rangle: Dist(G)\otimes Dist(G)\xrightarrow{m} Dist(G)\xrightarrow{\widetilde\pi} Dist(T)
$$
on $Dist(G)$ after the composition with the multiplication map $m$. Which, in turn, provides an isomorphism between $Dist(G)$ and $\underline{\Hom}_{T}(Dist(G), Dist(T))$ as $Dist(G)$-$Dist(T)$-bimodules (making $Dist(G)$ a self-injective extension of $Dist(T)$). 

It is then a well-known fact that the induction ($\mathrm{Ind}$) and coinduction ($\mathrm{Coind}$) functors coincide.

Now  $\widetilde P(\ba)=P(\bl_1)^{m_1}\oplus\ldots \oplus P(\bl_k)^{m_k} = I(\bl_1')^{m_1}\oplus\ldots\oplus I(\bl_k')^{m_k}=\widetilde I(\ba)$ with $I(\bl_i')\simeq P(\bl_i)$. Thus for  any simple representation $L$ of $G$ there exists some simple $L'$ with $P(L)\simeq I(L')$. We have for any simple $\ba\in \Rep T$:
$$
\dim \Hom_T(\ba, L)= \dim \Hom_G(\widetilde P(\ba), L) = [\widetilde P(\ba):P(L)] =
$$
$$
=[\widetilde P(\ba): I(L')]= \dim \Hom_G(L', \widetilde P(\ba))  = \dim \Hom_T(L', \ba)=\dim \Hom(\ba, L').
$$
Thus $L$ and $L'$ have the same $T$-character. In particular, when $L=\on$ is the trivial representation of $G$ we get that $L'=L=\on$. So for any projective module $P$,
$$
\dim \Hom_G(P,\on)=\dim \Hom_G(\on, P).
$$

Now for a general simple $G$-module $L$ we get that $P(L)\otimes (L')^* \simeq I(L')\otimes (L')^*$ is projective and 
$$
\dim \Hom_G(P(L), L')  = \dim \Hom_G(P(L)\otimes (L')^*, \on)=
$$
$$
=\dim \Hom_G(\on, P(L)\otimes (L')^*) =\dim \Hom_G(L', I(L'))=1,
$$
so $L=L'$, and thus $P(\bl)\simeq I(\bl)$ for any weight $\bl$.

\end{proof}

\begin{corollary}
    $\cc$ is a Frobenius category.
\end{corollary}

\begin{corollary}
    Both socle and cosocle of $P(\bl)$ are isomorphic to $L(\bl)$.
\end{corollary}

\begin{corollary}\label{cor_lowest_weight}
    Suppose $\bm$ is maximal among weights $\ba$, for which $[P(\bl):K(\ba)]\neq 0$. Then such $\bm$ is unique and, moreover, the lowest weight of $L(\bl)$ is 
    $$
    \bm-\bb,
    $$
    (which answers Question \ref{quest_1} in case of $G=GL(L_m|L_n)$ as long as we know $\bm$).

    By necessity, $\bm$ is the highest weight of $P(\bl)$. 
\end{corollary}
\begin{proof}
    By Corollary \ref{cor_exts_standards}, $K(\bm)$ must be a submodule of $P(\bl)$. Since $P(\bl)$ has an irreducible socle, we find that $\bm$ is unique and 
    $$
    soc(K(\bm))=L(\bl).
    $$
    Since the lowest weight of $K(\bm)$ is $\bm-\bb$, we get that 
    $$
    \bm-\bb \text{ is the lowest weight of } L(\bl)
    $$
    (see the discussion at the end of Section \ref{sec_super}).
\end{proof}

We will thus be interested in determining the correspondence $\bl \leftrightarrow \bm$ as above by understanding the multiplicities $[P(\bl):K(\ba)]$ for the standard filtration of $P(\bl)$.

\subsection{Irreducible Kac modules} \label{s_irrKac}

Let us try to understand when the Kac module  $K(\mu|\nu)$ for $GL(L_m|L_n)$ is irreducible.

\subsubsection{The Shapovalov form}


Consider the top nonzero graded component $K(\mu|\nu)^{mn} = S^{mn}(\mathfrak n^-)\otimes (\mu|\nu) = (\widetilde \mu|\widetilde \nu)$ of $K(\mu|\nu)$. The action map restricted to $S^{mn}(\mathfrak n^+)$ sends $K(\mu|\nu)^{mn}$ to the highest weight component $K(\mu|\nu)^0=(\mu|\nu)$. Under the identification $\mathfrak n^+ = (\mathfrak n^-)^*$ we get the dualized map:
$$
(\widetilde \mu, \widetilde \nu) = S^{mn}(\mathfrak n^-)\otimes (\mu|\nu)\to S^{mn}(\mathfrak n^-)\otimes (\mu|\nu) = (\widetilde \mu, \widetilde \nu),
$$
which is $T$-equivariant, and thus must be multiplication by some constant $Sh(\mu|\nu)$. We claim that this constant depends polynomially on the components of $\mu$ and $\nu$.

\begin{remark}\label{rem_Shapovalov}
    The constant $Sh(\mu| \nu)$ is a shadow of the actual Shapovalov form on $K(\mu|\nu)$. Namely, using the same logic, for any graded component $$K(\mu|\nu)^k = S^k(\mathfrak n^-)\otimes(\mu|\nu)$$ we get a $T$-equivariant map
    $$
    Sh^k: K(\mu|\nu)^k\to K(\mu|\nu)^k,
    $$
    obtained from dualizing the action map
    $$
    S^k(\mathfrak n^+)\otimes S^k(\mathfrak n^-)\otimes(\mu|\nu)\to (\mu|\nu).
    $$
    Now for each weight $(\alpha|\beta)$ with $|\alpha|+|\beta|=|\mu|+|\nu|$ and $|\alpha|=|\mu|-k$ we define a form $Sh_{\mu,\nu}^{\alpha,\beta}$ on the finite-dimensional vector space
    $$
    H^{\alpha,\beta}_{\mu,\nu} := \Hom_T((\alpha|\beta), K(\mu|\nu)^k)
    $$
    as follows:
    $$
    Sh_{\mu,\nu}^{\alpha,\beta}(f,g)=(Sh^k \circ f, g),
    $$
    where $(-,-)$ is a scalar product on $H_{\mu,\nu}^{\alpha,\beta}$ defined with respect to any choice of basis.

    The nullity of the determinant of the form $Sh_{\mu,\nu}^{\alpha,\beta}$ determines the multiplicity of $(\alpha|\beta)$ in the maximal proper submodule $J(\mu|\nu)$ of $K(\mu|\nu)$.

    If the multiplicity of weight $(\alpha|\beta)$ in $K(\mu|\nu)$ is $1$ then it lies in the maximal proper submodule $J(\mu|\nu)$ if and only if the map
    $$
    (\alpha|\beta)\hookrightarrow S^k(\mathfrak n^-)\otimes (\mu|\nu) \xrightarrow{Sh^k} (\alpha|\beta)\subset S^k(\mathfrak n^-)\otimes (\mu|\nu)
    $$
    is zero upon restriction to the weight subspace $(\alpha|\beta)$. 
\end{remark}

\subsubsection{The classical case}
In fact, the construction above is defined categorically, so it may be applied to representations of the super group $GL(m|n)$. Namely, given an irreducible representation $L_{ev}(\l)$ of the even part $GL(m|n)_{ev}=GL_m\times GL_n$ over $\ck$,  we can define the Kac module $K(\l)$ as the induced $Dist(GL(m|n))$-module
$$
K(\l) := Dist(GL(m|n))\tens{Dist(GL(m|n)_{ev}} L_{ev}(\l) = U(\gl(m|n))\tens{U(\gl_m\times\gl_n)} L_{ev}(\l)\simeq S(\n^-)\otimes L_{ev}(\l).
$$
The subalgebra $\n^-$ is a purely odd vector space, and its restriction to the even subgroup is 
$$\Res_{GL_m\times GL_n}^{GL(m|n)}\n^-=I\boxtimes (V^{(m)})^*\boxtimes V^{(n)}\subset sVec\boxtimes \Rep_\ck GL_m\boxtimes \Rep_\ck GL_n.$$ 
So $S^k(\n^-)=I^{\otimes k}\boxtimes \Lambda^k((V^{(m)})^*\boxtimes V^{(n)})$.
Thus, Kac module is finite dimensional, integrable  and $\Z$-graded, with the topmost nonzero component of degree $m\cdot n$, which is an irreducible $GL(m|n)_{ev}$-representation, isomorphic to $L_{ev}(\widetilde \l)$ with $\widetilde\l = \l-(n,\ldots, n|-m,\ldots, -m)$ (tensored with the appropriate power of the purely odd one-dimensional space $I$). The action of $S^{mn}(\mathfrak n^+)$ on this component sends it to the degree zero component $L_{ev}(\l)$ in $K(\l)$. Dualizing the action map, we obtain a $Dist(GL(m|n)_{ev})$-equivariant map $L_{ev}(\widetilde \l)\to L_{ev}(\widetilde \l)$, which is given by multiplication by some constant $Sh(\l)$.

The following  theorem was proved by Kac for example in (\cite{K06}, Proposition 2.9) in characteristic zero. His argument  can most likely be generalized directly to  the case of characteristic $p$. However, we will provide an alternative proof here.

\begin{theorem}[\cite{K06}]\label{thm_irrkac} Let $\Delta_1^+$ be the set of positive odd roots for $\gl(m|n)$, i.e. 
$$
\Delta_1^+=\{\epsilon_i-\delta_j~|~1\le i\le m, 1\le j\le n\},
$$
where $\epsilon_1,\ldots, \epsilon_m, \delta_1, \ldots, \delta_n$ are the standard basis vectors in $\Z^{m|n}$ with  a fixed symmetric form $\langle-, -\rangle$, s.t.
$$
\langle \epsilon_i, \epsilon_j\rangle =\delta_{i,j},~\langle\epsilon_i, \delta_j\rangle =0,~\langle\delta_i, \delta_j\rangle = -\delta_{i,j}.
$$
Let $\l=(\mu_1,\ldots,\mu_m, \nu_1,\ldots,\nu_n)\in \Z^{m|n}$ be such that
$$
\mu_i\ge \mu_{i+1}, \nu_j\ge \nu_{j+1},
$$
and let $\rho$ be the half-sum of positive even roots minus the half-sum of positive odd roots:
$$
2\rho = (m-1-n, m-3-n, \ldots, 1-m-n| n-1+m, n-3+m, \ldots, 1-n+m).
$$
Then the constant $Sh(\l)$ is nonzero (and  consequently, $K(\l)$ is irreducible) if and only if $\langle \l+\rho, \alpha\rangle \not\equiv 0\mod p$ for all positive odd roots $\alpha\in \Delta_1^+$.  \end{theorem}

\begin{remark}
     In fact, Kac proves that $Sh(\l)$ is a polynomial in $\mu_1,\ldots, \mu_m, \nu_1, \ldots, \nu_n$ of degree $m\cdot n$ which is proportional to the product 
$$
\prod_{\alpha\in \Delta_1^+} \langle \l+\rho, \alpha\rangle.
$$
We won't need this particular statement, so we will not give a proof for it.
\end{remark}

To prove Theorem \ref{thm_irrkac} we will need some background knowledge about weights of $GL(m|n)$-modules. For each pair of positive odd roots $\epsilon_i-\delta_j, \epsilon_a-\delta_b\in \Delta_1^+$ let us say
$$
\epsilon_i-\delta_j \succeq \epsilon_a - \delta_b,
$$
if their difference is a nonnegative linear combination of positive roots. That is, if $i\le a$ and $j\ge b$.

Let $L(\l)$ be the irreducible quotient of $K(\l)$ and consider its composition series with respect to the even subgroup $GL_m\times GL_n$. The lowest weight vector in $L(\l)$ is contained in some composition factor $L_{ev}(\hat\l)$. The weight $\hat\l$ can be computed using Serganova's algorithm (proved for characteristic $p$ in \cite{BK03}).
\begin{lemma}[\cite{BK03}, Theorem 4.3]\label{l_serg_alg}
    Pick an ordering $\beta_1,\ldots, \beta_{mn}$ of the positive odd roots $\Delta^+_1=\{\epsilon_i - \delta_j~|~1 \le i \le m, 1 \le j \le n\}$ such that $\beta_i\preceq \beta_j$ implies $i\le  j$. Set $\l^{(0)} = \l$, and inductively define 
    $$
    \l^{(i)} = \begin{cases}
       \l^{(i-1)},\text{ if } \langle \l^{(i-1)},\beta_i\rangle \equiv 0\mod p,\\
       \l^{(i-1)} - \beta_i, \text{ if } \langle \l^{(i-1)},\beta_i\rangle \not\equiv  0 \mod p, 
\end{cases}
    $$
    for $i = 1,\ldots,mn$. Then, $\hat\l = \l^{(mn)}$.
\end{lemma}

\begin{lemma}\label{l_oddroots}
    Let $\beta_1,\ldots, \beta_{mn}$ be the ordering of positive odd roots as in Lemma \ref{l_serg_alg}. Then for each $1\le k\le mn$
    $$
    \langle \sum_{i=1}^{k-1} \beta_i, \beta_{k}\rangle = -\langle \rho, \beta_{k}\rangle.
    $$
\end{lemma}
\begin{proof}
    Let $\beta_{k}=\epsilon_a-\delta_b$ then all positive odd roots preceding $\beta_{k}$ in our order and not orthogonal to it are either of the form
    $$
    \epsilon_a-\delta_d, \text{ with } d<b,
    $$
    or of the form
    $$
    \epsilon_c - \delta_b, \text{ with } c>a.
    $$
    We have
    $$
    \langle \epsilon_a-\delta_d, \epsilon_a-\delta_b \rangle = 1, \text{ and }
    $$
    $$
    \langle \epsilon_c-\delta_b, \epsilon_a-\delta_b\rangle = -1.
    $$
    Thus,
    $$
    \langle \sum_{i=1}^{k-1} \beta_i, \beta_{k}\rangle = (b-1)-(m-a)=-m+a+b-1.
    $$
    On the other hand, 
    $$
    \langle \rho, \beta_{k}\rangle = \langle \rho, \epsilon_a-\delta_b\rangle = \frac{1}{2}(m-2a+1-n)+\frac{1}{2}(n-2b+1+m) = m-a-b+1.
    $$
\end{proof}

\subsubsection*{Proof of Theorem \ref{thm_irrkac}}
First, note that $K(\bl)$ is irreducible if and only if its $(GL_m\times GL_n)$-equivariant lowest weight $\widetilde\l$ is the same as the $(GL_m\times GL_n)$-equivariant lowest weight $\hat\l$ of $L(\l)$. Thus, $Sh(\l)$ is nonzero if and only if $$ \hat\l = \l - \sum_{i=1}^{mn} \beta_i=\widetilde \l$$. 

This can only happen if in Serganova's algorithm (see Lemma \ref{l_serg_alg}) at each step $1\le k\le mn$ we have $\l^{(k)} = \l^{(k-1)}-\beta_{k}$, i.e.
$$
\langle \l^{k-1},\beta_k\rangle  = \langle \l - \sum_{i=1}^{k-1} \beta_i, \beta_k\rangle \not\equiv 0 \mod p.
$$

By Lemma \ref{l_oddroots}, this condition is equivalent to saying that for every $1\le k\le mn$ we have
$$
\langle \l+\rho, \beta_k\rangle\not\equiv 0 \mod p.
$$
  \blacksquare

 \subsubsection{Deducing the $Ver_p$ case from the classical case}
 Let $\mu, \nu$ be admissible weights for $GL_m$ and $GL_n$ correspondingly with $m+n<p$. Let $\l=(\mu, \nu)\in \Z^{m|n}$ and consider the restriction of the $GL(m|n)$-module $K(\l)$ to the even subgroup $GL_m\times GL_n$. Since $\mu, \nu$ are admissible, the simple representation
 $$
L_{ev}(\l)=L(\mu)^{(m)}\boxtimes L(\nu)^{(n)}$$
is a tilting module for $GL(m|n)_{ev}$.
Moreover, the top degree component of $K(\l)$
$$S^{mn}(\mathfrak n^-)\otimes L_{ev}(\l) = I^{\otimes mn}\boxtimes L_{ev}(\widetilde \l) = I^{\otimes mn}\boxtimes (L(\mu)^{(m)}\otimes det^{-n})\boxtimes (L(\nu)^{(n)}\otimes det^m)$$
also lies in the subcategory of tilting modules.

Consider the semisimplification functor
$$
\mathbf{S}: sVec\boxtimes \Rep_\ck(GL_m)\boxtimes \Rep_\ck(GL_n) \to sVec\boxtimes \overline{\Rep_\ck(GL_m)}\boxtimes \overline{\Rep_\ck(GL_n)}.
$$

The image of $S^{mn}(\n^-)\otimes L_{ev}(\l)$ under $\mathbf{S}$ lies in the full subcategory 
$$
sVec\boxtimes Ver_p(GL_m)\boxtimes Ver_p(GL_n)
$$
and is isomorphic to $(\widetilde \mu|\widetilde \nu).$

From the categorical definition of $Sh(\mu|\nu)$ it follows that
$$
Sh(\mu|\nu)=Sh(\l)
$$
(as morphisms map to morphisms under the semisimplification).

\begin{corollary}\label{cor_irr_kac_critereon}
    The Kac module $K(\mu|\nu)$ for $G=GL(L_m|L_n)$ is irreducible if and only if
    $$
    \langle (\mu|\nu)+\br^{(m|n)}, \ba_{i,j}\rangle \not\equiv 0\mod p
    $$
    for all positive odd roots $\ba_{i,j}=\epsilon_i-\delta_j, 1\le i\le m, 1\le j\le n$ (as defined in Section \ref{sec_super}).
\end{corollary}

\subsection{Casimir element}

\begin{definition} \label{def_casimir}
    Let $X\in Ver_p$. A \textbf{quadratic Casimir element} $C_X\in \Hom_{Ver_p}(\on,  U(\gl(X)))$ is defined as the composition:
    \[\xymatrixcolsep{1.1pc}
    \xymatrix{
    \on \ar[rr]^-{coev_X} && X\otimes X^*\ar[r]^-{\sim} & X\otimes \on \otimes X^*\ar[r]^-{id_X\otimes coev_{X^*}\otimes id_{X^*}}& X\otimes (X^*\otimes X)\otimes X^*\ar[dl]_-{\sim}\\& && (X\otimes X^*)\otimes (X\otimes X^*)=\gl(X)\otimes \gl(X)\ar[r]^-{m} & U(\gl(X)),
    }
    \]
    where $m$ denotes the multiplication on $U(\gl(X))$.
\end{definition}
For any $GL(X)$-module $M$ define the action of $C_X$ on $M$ as the map
$$
M\xrightarrow{\sim} \on\otimes M\xrightarrow{C_X\otimes id_M} U(\gl(X))\otimes M \to M.
$$

Clearly, $C_X$ is $\gl(X)$-invariant, i.e. $C_X\in \Hom_{\gl(X)}(\on, U(\gl(X))^{ad})$ as all maps in the diagram in Definition \ref{def_casimir} are $GL(X)$-equivariant, and thus $C_X$ acts by scalar multiplication on every simple $GL(X)$-module. 

When $X=L_n$ is simple, we get that $C_X$ acts on the simple object $V_\lambda\in Ver_p(GL_n)= \Rep_{Ver_p}(GL(X),\varepsilon)$ via multiplication by
$$
\langle \l+2\rho^{(n)}, \l\rangle,
$$
where $2\rho^{(n)} = (n-1, n-3, \ldots, 3-n, 1-n)\in \Z^n$ and the form $\langle- ,-\rangle$ is the standard scalar product on $\Z^n$ (this is true because it was so before the semisimplification).

 Let $X = L_m\oplus L_{p-n}$ let $\mu\in \Z^m$ and $\pi\in \Z^{p-n}$ be two admissible weights. Recall that the level-rank duality gives us a bijection $\pi\leftrightarrow D(\pi)$ between admissible weights in $\Z^{p-n}$ and in $\Z^n$. 

 We can thus label Kac modules for $GL(X)$ both by pairs $(\mu, \pi)$ of admissible weights in $\Z^m$ and $\Z^{p-n}$, and by pairs $(\mu|\nu)$ of admissible weights in $\Z^m$ and $\Z^n$ with the convention that
 $$
 K(\mu, \pi) = K(\mu|D(\pi)).
 $$

 \begin{theorem}
 \label{tcasimir}
 \begin{enumerate}
     \item The Casimir element $C_X$ acts on $K(\mu, \pi)$ via multiplication by
 $$
 \langle \mu+2\rho^{(m)},\mu\rangle+\langle \pi+2\rho^{(p-n)},\pi\rangle-n|\mu|-m|\pi|.
 $$
 If we put $2\br^{(m+p-n)}:=(m+p-n-1, m+p-n-3, \ldots, 1-m-p+n)$ then this constant can be rewritten as
 $$
 \langle (\mu,\pi)+2\br^{(m+p-n)}, (\mu,\pi)\rangle
 $$
 for the standard scalar multiplication $\langle-,-\rangle$ on $\Z^{m+p-n}$.
 \item The Casimir element $C_X$ acts on $K(\mu|\nu)$ via multiplication by 
 $$
 \langle \mu+2\rho^{(m)},\mu\rangle - \langle \nu+2\rho^{(n)}, \nu\rangle -n|\mu|-m|\nu|.
 $$
 If we put $$
 2\br^{(m|n)}:= 2(\rho^{(m)}|\rho^{(n)})-\bb=
 $$
 $$
 =(m-1-n, m-3-n,\ldots,1-m-n|n-1+m, n-3+m\ldots, 1-n+m)
 $$
 then the above formula can be rewritten as
 $$
 \langle (\mu|\nu)+2\br^{(m|n)}, (\mu|\nu)\rangle
 $$
 for the symmetric form $\langle -,-\rangle$ on $\Z^{m|n}$ defined in Section \ref{sec_super}.
  \end{enumerate}
\end{theorem}
\begin{proof}
\begin{enumerate}
    \item Let $X= Y\oplus Z$ with $Y=L_n, Z=L_{p-n}$. Recall that we defined subobjects $$
\mathbf{E}_{1,1} = Y\otimes Y^*,
$$
$$
\mathbf{E}_{1,2} = Y\otimes Z^*,
$$
$$
\mathbf{E}_{2,1} = Z\otimes Y^*,
$$
$$
\mathbf{E}_{2,2} = Z\otimes Z^*
$$
of $\gl(X)$.

The map $C_X = m\circ ((coev_X)_{14}\otimes  (coev_{X^*})_{23}):\on\to (X\otimes X^*)\cdot (X\otimes X^*) $ is equal to the composition of the map 
\begin{align*}
    (coev_Y)_{14}\otimes(coev_{Y^*})_{23}+&(coev_Z)_{14}\otimes  (coev_{Z^*})_{23}+ \\
    +(coev_Y)_{14}\otimes &(coev_{Z^*})_{23}+(coev_Z)_{14}\otimes (coev_{Y^*})_{23}
\end{align*}
from $\on$ to the direct sum
\begin{align*}
    \left((Y\otimes Y^*)\otimes (Y\otimes Y^*)\right)\oplus &\left((Z\otimes Z^*)\otimes (Z\otimes Z^*)\right) \oplus \\
    &\oplus ((Y\otimes Z^*)\otimes (Z\otimes Y^*))\oplus ((Z\otimes Y^*)\otimes (Y\otimes Z^*)]=\\
     = \mathbf{E}_{1,1}\otimes &\mathbf{E}_{1,1} \oplus \mathbf E_{2,2}\otimes \mathbf E_{2,2} \oplus \mathbf E_{1,2} \otimes \mathbf E_{2,1} \oplus \mathbf E_{2,1} \otimes \mathbf{E}_{1,2}.
\end{align*}

 with the multiplication map $m$.
 
The term  $\mathbf E_{2,1} \cdot \mathbf{E}_{1,2}$ acts by zero on the highest weight $T$-submodule $(\mu,\pi)$ in $K(\mu,\pi)$. The first two summand give us the action of $C_Y+C_Z$ via multiplication by $(\mu+2\rho^{(m)},\mu)+(\pi+2\rho^{(p-n)},\pi)$. The action of the term $\mathbf E_{1,2} \cdot \mathbf E_{2,1}$ is computed via taking the commutator of $\mathbf E_{1,2}$ and $\mathbf E_{2,1}$ (which includes composing the evaluation map with coevaluation map, which spits out the dimensions of $Y$ and $Z$):
$$
(ev_Z)_{23}\circ ((coev_Y)_{14}\otimes (coev_{Z^*})_{23})= \dim Z\cdot coev_Y = -n\cdot coev_Y: 
$$
$$
\on \to \mathbf E_{1,2}\otimes \mathbf E_{2,1} \to \mathbf E_{1,1},
$$
$$
-(ev_Y)_{23}\circ 
\tau_{\mathbf E_{1,2}, \mathbf E_{2,1}}\circ ((coev_Y)_{14}\otimes (coev_{Z^*})_{23}) = -\dim Y \cdot coev_Z = -m\cdot coev_Y: 
$$
$$
\on\to  \mathbf E_{1,2}\otimes \mathbf E_{2,1} \to  \mathbf E_{2,1}\otimes \mathbf E_{1,2}\to \mathbf E_{2,2} 
$$
Finally, note that for any simple $V_\l\in \Rep_{Ver_p}(GL(L_k), \varepsilon)$ the composition of $coev_{L_k}\otimes id_{V_\l}:\on\otimes V_\l\to L_k\otimes L_k^*\otimes V_\l$ and the action of $\gl(L_k)=L_k\otimes L_k^*$ on $V_\l$ (i.e. the action of ``the identity matrix") is scalar multiplication by $|\l|$.


\item If $X=Y\oplus (Z\otimes I)$, where $I=L_{p-1}$ is the purely odd one-dimensional vector space then we can apply a similar reasoning as in part (1). The only difference now is that $coev_{Z\otimes I} = -coev_{Z}$. 
Thus the map $C_{Z\otimes I}:\on\to \mathbf E_{2,2}\cdot \mathbf E_{2,2}$ is the same as $-C_Z$. Similarly, the composition $ev_{Z\otimes I}\circ \coev_{(Z\otimes I)^*}$ is $-\dim Z  = -n$.
\end{enumerate}
\end{proof}
\begin{corollary}
    When $\nu\in \Z^n, \pi\in \Z^{p-n}$ are admissible weights (i.e. weights lying in the fundamental alcove $C_p$) with $\nu = D(\pi)$, we get 
    \begin{equation}\label{eq_weird_equality}
    \langle \pi+2\rho^{(p-n)},\pi \rangle = -\langle \nu+2\rho^{(n)}, \nu\rangle \mod p.     
    \end{equation}
    In particular, when $\pi=(\pi_1,\ldots, \pi_{p-n})$ satisfies
    $$
    p-n\ge \pi_1\ge \ldots\ge \pi_{p-n-1}\ge \pi_{p-n}\ge 0,
    $$
    we have $\nu = \pi^t$, and equation (\ref{eq_weird_equality}) becomes a combinatorial statement about a partition whose Young diagram fits into the $n$-by-$(p-n)$ box and its transpose. 
\end{corollary}

\subsection{Categorical \texorpdfstring{$\widehat{\ssl}_p$-}{ }action}\label{sec_catact}  Let $X\in Ver_p$ and let $\cc = \Rep_{Ver_p}(GL(X), \varepsilon)$. Recall that $[\cc]$ denotes the complexified Grothendieck group of $\cc$.

Consider biadjoint endofunctors $F=X\otimes - , E= X^*\otimes -$ on $\cc$ and the operators $x\in \mathrm{End}(F), t\in \mathrm{End}(F^2)$ defined in the same way as in Section \ref{catact}, i.e.
$$
x_M: X\otimes M\to X\otimes M
$$
is obtained by dualizing the action map
$$
a_M: X\otimes X^*\otimes M\to M,
$$
and
$$
t_M = \tau_{X,X}\otimes id_M: X\otimes X\otimes M\to X\otimes X\otimes M.
$$

It is easy to see that the operator $x$ is the tensor Casimir, i.e.
$$
x_M = \frac{1}{2}(\Delta\circ C_X-C_X\otimes 1 - 1\otimes C_X)|_{M\otimes X} 
$$
where $\Delta: U(\gl(X))\to U(\gl(X))\otimes U(\gl(X))$ is the comultiplication map. This is because 
$$
\frac{1}{2}(\Delta\circ C_X-C_X\otimes 1 - 1\otimes C_X)  = (coev_X)_{14}\otimes (coev_{X^*})_{23}: \on\to \gl(X)\otimes \gl(X).
$$
It acts on $X\otimes M$ via the map
$$
X\otimes M\xrightarrow{\sim} \on \otimes X\otimes M \xrightarrow{(coev_X)_{14}\otimes (coev_{X^*})_{23}} [X\otimes X^*]_1\otimes [X\otimes X^*]_2\otimes X\otimes M \xrightarrow{\tau_{[X\otimes X^*]_2, X}} 
$$
$$
\xrightarrow{\tau_{[X\otimes X^*]_2, X}} ([X\otimes X^*]_1\otimes X)\otimes ([X\otimes X^*]_2\otimes M) \xrightarrow{(ev_X)_{23}} X\otimes (X\otimes X^*\otimes M)\xrightarrow{a_M} X\otimes M.
$$
Using the equality
$$
id_X = (ev_X)_{23}\circ (coev_X\otimes id_X): X\to X\otimes X^*\otimes X \to X,
$$
we simplify the above map to $(id_X\otimes a_M)\circ (\tau_{(X\otimes X^*), X}\otimes id_M)\circ(id_X\otimes coev_{X^*}\otimes id_M)$, which, by definition, is the dualization of the action map $a_M$, i.e. $x_M$:
$$
x_M:X\otimes M\xrightarrow{coev_{X^*}} X\otimes X^*\otimes X\otimes M\xrightarrow{\tau_{(X\otimes X^*), X}} X\otimes (X\otimes X^*\otimes M)\xrightarrow{a_M} X\otimes M.
$$

\begin{theorem}\label{tcatact}
    \begin{enumerate}
        \item The data $(E,F,x,t)$ defines the categorical $\widehat{\ssl}_p$-action on $\cc$.

        \item Let $X=L_m\oplus L_{p-n}$ and let $[\cc]^\Delta$ be the subspace of $[\cc]$ spanned by standard (Kac) modules. Then we have an isomorphism of $\widehat{\ssl}_p$-modules 
        $$
        [\cc]^\Delta \simeq (\Lambda^m U)^{loop}\otimes (\Lambda^{p-n} U)^{loop} \simeq (\Lambda^m U)^{loop} \otimes (\Lambda^n U^*)^{loop},
        $$
        where $U=\mathrm{span}_\C(e_0,\ldots, e_{p-1})$ is the tautological $p$-dimensional $\ssl_p$-module, and the action of $\widehat\ssl_p$ factors through the action of $\ssl_p^{loop}=\ssl_p[t,t^{-1}]$ as follows: the generators $f_i$ acts via $E_{i+1, i}$ (and $e_i$ via $E_{i,i+1}$) for $0\le i\le p-2$, whereas $f_{p-1}$ acts via $E_{0, p-1}\cdot t^{-1}$ and $e_{p-1}$ acts via $E_{p-1,0}\cdot t$.
    \end{enumerate}
\end{theorem}
\begin{proof}
    \begin{enumerate}
        \item The proof is identical to the proof of part (1) of Theorem \ref{thm_sl_p_act}.

        \item The $\widehat{\ssl}_p$-modules $(\Lambda^{p-n}U)^{loop}$ and $(\Lambda^{n} U^*)^{loop}$ are clearly isomorphic, so we will only prove the isomorphism
        \begin{equation*}\label{iso_sl_p_act}
         [\cc]^\Delta \simeq (\Lambda^{m} U)^{loop}\otimes(\Lambda^{n} U^*)^{loop}.   
        \end{equation*}

        First recall that we constructed an isomorphism $[Ver_p(GL_m)]\simeq (\Lambda^m U)^{loop}$, which sent $[V_\mu]$ to $v_\mu$ defined as follows:
        $$
        v_\mu = (v_{a_1}\wedge v_{a_2}\wedge \ldots\wedge v_{a_m})\cdot t^{-s},
        $$
        where $\mu_i + 1 -i = a_i+p\cdot s_i$, $0\le a_i\le p-1$, and $s=\sum_{i=1}^m s_i$.

        Let $v_b^*\in U^*$ be such that $v_b^*(v_a) = (-1)^b \delta_{a,b}$. We have
        $$
        f_c v_{b+1}^* = \delta_{b,c} v_b^* \text{ when } c<p-1.
        $$
        Now define (slightly abusing the notation) $v_\nu^*\in (\Lambda^n U^*)[t,t^{-1}]$ via
        $$
        v_\nu^* = (v_{b_1}^*\wedge v_{b_2}^*\wedge\ldots\wedge v_{b_n}^*)\cdot t^{r},
        $$
        where $-m-\nu_j+j = b_j + p\cdot r_j$, $0\le b_j\le p-1$, and $r=\sum_{j=1}^n r_j$. 
        Note that
        $$
        f_c v_\nu^* = \sum_{j=1}^n \delta^{(p)}_{c, b_j-1} v_{\nu+e_j}^* = \sum_{j=1}^n \delta^{(p)}_{c,-m- (\nu_j-j+1)}v_{\nu+e_j}^*,
        $$
        where $\delta^{(p)}_{i,j} = 1$ if $i=j$ modulo $p$ and zero otherwise.

        Recall also our notation:
        $$
        V^{(k)}\otimes V_\l^{(k)} = \bigoplus_{\l'\in \l+\square} V_{\l'}^{(k)},
        $$
        meaning that $\l' = \l+e_i$ for some basis vector $e_i$ and $\l'$ is admissible. 
        Recall also that for the functor $F^{(k)}=V^{(k)}\otimes -$, we have  $V^{(k)}_{\l'} = F^{(k)}_c(V^{(k)}_\l)$ if $\l'\in \l+\square_c$, i.e. $\l'=\l+e_i$ and $c = \l_i+1-i$.

        Consider the linear map 
        $$
        \phi: [\cc]^\Delta\to (\Lambda^{m} U)^{loop}\otimes(\Lambda^{n} U^*)^{loop},
        $$
        sending $[K(\mu|\nu)]$ to $v_\mu\otimes v_\nu^*$. This map is an isomorphism of vector spaces, so it is enough to show that it is $\widehat{\ssl}_p$-invariant.
        
        Note that since
        $$
        \Hom_G(X\otimes K(\mu|\nu),-)=\Hom_G(K(\mu|\nu),X^*\otimes -) = 
        $$
        $$
        =\Hom_B((\mu|\nu), (X^*\otimes -)|_B) = \Hom_B(X|_B\otimes (\mu|\nu), -),
        $$
        the module $X\otimes K(\mu|\nu)$ is isomorphic to $Dist(G)\tens{Dist(B)} (X|_B\otimes (\mu|\nu))$. So, it 
        has filtration by Kac modules, with quotient $\bigoplus_{\nu'\in \nu+\square} K(\mu|\nu')$ and submodule $\bigoplus_{\mu'\in \mu+\square} K(\mu'| \nu)$. This is because $Dist(G)\tens{Dist(B)} -$ is exact and as a $B$-module $X$ is an extension between $V^{(n)}\otimes I$ and $V^{(m)}$. So, $X|_B\otimes (\mu|\nu)$ is the extension:
        $$
        0\to \bigoplus_{\mu'\in\mu+\square} (\mu' | \nu) \to X|_B\otimes (\mu | \nu) \to \bigoplus_{\nu'\in\nu+\square} (\mu | \nu')  \to 0.
        $$

        The operator $2x = (\Delta\circ C_X - C_X\otimes 1 - 1 \otimes C_X)$ acts on the composition factor $K(\mu+e_i| \nu)$  of $X\otimes K(\mu| \nu)$ via multiplication by the constant
        $$
        (C_X|_{K(\mu+e_i| \nu)} - C_X|_X - C_X|_{K(\mu|\nu)}) = 
        $$
        $$
        =\langle (\mu|\nu)+2\br^{(m|n)}+\epsilon_i, (\mu|\nu)+\epsilon_i\rangle - \langle \epsilon_1 + 2\br^{(m|n)}, \epsilon_1\rangle - \langle (\mu|\nu)+2\br^{(m|n)}, (\mu|\nu)\rangle=
        $$
        $$
        =2\langle (\mu|\nu) + \br^{(m|n)} , \epsilon_i\rangle + \langle \epsilon_i,\epsilon_i\rangle - (m-n) = 
        $$        
        $$
         =  2\mu_i + (m-n-2i+1) - (m-n-1) =   2(\mu_i+1-i).
        $$
        Similarly, $2x$ acts on the composition factor $K(\mu| \nu+e_j)$ of $X\otimes K(\mu| \nu)$ via multiplication by
        $$
        \langle (\mu|\nu)+2\br^{(m|n)}+\delta_j, (\mu|\nu)+\delta_j\rangle - \langle (\mu|\nu)+2\br^{(m|n)},(\mu|\nu)\rangle - \langle \epsilon_1 + 2\br^{(m|n)}, \epsilon_1\rangle =   
        $$
        $$
        = 2\langle (\mu|\nu) + \br^{(m|n)}, \delta_j\rangle +\langle\delta_j, \delta_j\rangle  - (m-n) = 
        $$
        $$
        =-2\nu_j - (n+m-2j+1) -1 -(m-n) = 2(-m-\nu_j + j -1 ).
        $$

        Thus, if $K(\mu|\nu+e_j)$ is a composition factor in $F_c(K(\mu|\nu))$ then $$c = -m - (\nu_j-j+1).$$ Note also that for each $j\in \{1,\ldots, n\}$, for which $\nu+e_j$ is admissible, the constants $\nu_j-j+1$ are distinct.

        Let $f_c = [F_c]$ acting on the Grothendieck group $[\cc]$. The computation above shows that
        $$
        f_c[K(\mu| \nu)] = [K(\mu'|\nu)]+[K(\mu| \nu')],
        $$
        where $V_{\mu'}^{(m)} = F_c(V^{(m)}_\mu)$ and $V_{\nu'}^{(n)} = F_{-m-c}(V^{(n)}_\nu)$. (Here we put $[K(\mu'|\nu)]=0$ in case $F_c(V^{(m)}_\mu)=0$, with similar assumption on $[K(\mu| \nu')]$.)

        The formula above is consistent with 
        $$
        f_c(v_\mu\otimes v_\nu^*) = f_c(v_\mu)\otimes v_\nu^* + v_\mu\otimes f_c(v_\nu^*),
        $$
    after noticing that
    $$
    f_c(v_\nu^*) = \sum_{j=1}^n \delta^{(p)}_{\nu_j-j+1, -m-c}v_{\nu+e_j}^*.
    $$

        The argument for the action of $e_c$ is completely analogous.

    \end{enumerate}
\end{proof}

\begin{theorem}\label{thm_projfun}
    Functors $F_i, E_i, i\in \mathbb F_p$ are \textbf{projective}, i.e. for any projective $GL(X)$-module $P$ the modules $F_i(P), E_i(P)$ are also projective. 
\end{theorem}
\begin{proof}
    As $F_i, E_i$ are direct summands of $F$ and $E$ respectively, it is enough to check that $F$ and $E$ are projective. 

    It follows from the fact that for any $GL(X)$-module $Y$ of finite length, the functor $Y\otimes - $ is projective, as 
    $$
    \Hom_{GL(X)}(Y\otimes P, -)=\Hom_{GL(X)}(P, Y^*\otimes -).
    $$
\end{proof}

\subsection{Weight diagrams}\label{s_wd}
In Theorem \ref{tcatact}, we assigned to each standard object $K(\mu|\nu)$ a vector $v_{\mu}\otimes v_\nu^*$ in $(\Lambda^m U)[t_1,t_1^{-1}]\otimes (\Lambda^n U^*)[t_2,t_2^{-1}]$, where $U=\C^p=\mathrm{span}_\C(v_0,v_1,\ldots, v_{p-1})$. Namely, put $$a_i = \mu_i-i+1 ~\mathrm{ mod } ~p,$$ 
$$b_j = -m-\nu_j+j~\mathrm{mod}~p,$$
and define $s$ and $r$ so that
$$
\mu_i-i+1 = a_i + p\cdot s_i,~-m-\nu_j+j = b_j+p\cdot r_j,
$$
and $s=\sum_{i=1}^m s_i$,  $r=\sum_{i=1}^n r_i$.
Then $[K(\mu|\nu)]$ corresponds to the vector
$$
(v_{a_1}\wedge v_{a_2}\wedge\ldots\wedge v_{a_m})\cdot t_1^{-s}\otimes (v_{b_1}^*\wedge v_{b_2}^* \wedge \ldots\wedge v_{b_n}^*)\cdot t_2^r.
$$

\begin{remark}\label{rem_ab}
    We have
    $$
    a_i+p\cdot s_i = \langle (\mu|\nu) + \br^{(m|n)}, \epsilon_i\rangle +\frac{n-m-1}{2};
    $$
    $$
    b_j+p\cdot r_i = \langle (\mu|\nu) + \br^{(m|n)}, \delta_j\rangle +\frac{n-m-1}{2}.
    $$
\end{remark}

We will store the information about constants $a_1,\ldots, a_m, b_1, \ldots, b_n\in \{0,\ldots, p-1\}$ and $s,r\in \Z$ graphically in terms of weight diagrams. We will then describe the action of translation functors $F_i, E_i$ on $\cc=\Rep_{Ver_p}(GL(L_m|L_n), \varepsilon)$ in terms of these diagrams.
 
The weight diagram of $(\mu|\nu)$ is defined in the same way as for weights of the Lie superalgebra $\gl(m|n)$ (see, for example, \cite{GS11})\footnote{Our notation is similar to the one commonly used for Lie superalgebras. Initially, though, weight diagrams were intruduced by Brundan and Stroppel in \cite{BS08} (in \cite{BS12} for $\gl(m|n)$ specifically). Their notation is slightly different, using symbols $\circ, \wedge, \vee, \times$ in place of $>, \circ, \times, <$ respectively.} with one main difference. The symbols encoding components of the weight are placed not on the number line but on a circle, so that we only keep track of their remainders modulo $p$.

\begin{definition}
    The \textbf{weight diagram} of a weight $(\mu|\nu)\in \Z^{m|n}$ is an arrangement of symbols $\circ, <, >,$ and $\times$ on the vertices of a regular $p$-gon defined as follows.
    \begin{itemize}
        \item We start with a regular $p$-gon on a circle, whose vertices are labeled in the clock-wise manner by numbers $0,1, \ldots, p-1$:
        
         
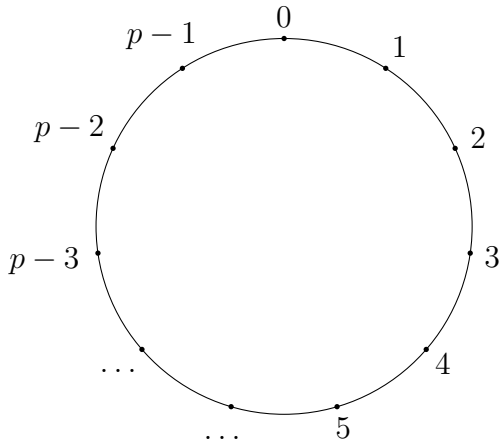
\begin{tikzpicture}
  \coordinate (O) at (1,2);
  \def\radius{2.5cm}

  \draw (O) circle[radius=\radius];

  \def\x{90};
\dotnode
    
   \def\x{122.73};
      \dotnode
      
    \def\x{155.45};
   \dotnode

    \def\x{188.18};
    \dotnode

    \def\x{220.9};
    \dotnode

    \def\x{253.64};
   \dotnode

    \def\x{286.36};
  \dotnode

    \def\x{319.1};
  \dotnode

    \def\x{351.82};
    \dotnode

    \def\x{24.54};
    \dotnode

    \def\x{57.27};
   \dotnode

    \def\y{0.7em};
    \def\x{90};
  \nodelabel{$0$}
   \def\x{122.73};
   \def\y{1.2em};
   \nodelabel{$p-1$}
    \def\x{155.45};
      \def\y{1.5em};
   \nodelabel{$p-2$}
    \def\x{188.18};
    \def\y{1.7em};
   \nodelabel{$p-3$}
    \def\x{220.9};
      \def\y{1em};
    \nodelabel{$\ldots$}
    \def\x{253.64};
   \nodelabel{$\ldots$}
    \def\x{286.36};
    \def\y{0.7em};
   \nodelabel{$5$}
    \def\x{319.1};
    \def\y{0.7em};
   \nodelabel{$4$}
    \def\x{351.82};
    \def\y{0.7em};
    \nodelabel{$3$}
    \def\x{24.54};
    \def\y{0.8em};
    \nodelabel{$2$}
    \def\x{57.27};
    \def\y{0.8em};
   \nodelabel{$1$}

\end{tikzpicture}

\item For each vertex $V_k$ labeled by $k$ (with $0\le k\le p-1$) choose one of the following four options:
\begin{enumerate}
    \item If there exists $a_i := \mu_i + i - 1$, which is equal to $k$ modulo $p$, and none of the $b_j:= m - (\nu_j-1)$ are equal to $k$ modulo $p$, we place an arrow pointing clockwise $(\cw)$ on $V_k$.
    \item If there exists $b_j\equiv k\mod p$ and none of the $a_i$ are equal to $k$ modulo $p$, we place an arrow pointing counter-clockwise $(\ccw)$ on $V_k$.
    \item If $k\equiv a_i\equiv b_j\mod p$ for some $i$ and $j$, we combine  $\cw$ and $\ccw$ resulting in a cross symbol $(\times)$ placed on $V_k$.
    \item If none of $a_i, b_j$ are equal to $k$ modulo $p$, we place an empty circle $(\circ)$ on $V_k$.
\end{enumerate}

\item The last step consists of placing the label $t_1^{-s} t_2^r$ in the center of the diagram.
\end{itemize}
\end{definition}

Figure \ref{fig:wd} below is an example of a weight diagram for $p=11$, $m = 5, n = 4$.

\begin{figure}[h]
    \centering
    \caption{Example of a weight diagram}
    \label{fig:wd}

       \begin{tikzpicture}
  \coordinate (O) at (1,2);
  \def\radius{2.5cm}

  \draw (O) circle[radius=\radius];

  \def\x{90};
\onode
    
   \def\x{122.73};
      \lenode
      
    \def\x{155.45};
   \xnode

    \def\x{188.18};
    \genode

    \def\x{220.9};
    \genode

    \def\x{253.64};
   \xnode

    \def\x{286.36};
  \onode

    \def\x{319.1};
  \lenode

    \def\x{351.82};
    \onode

    \def\x{24.54};
    \genode

    \def\x{57.27};
   \onode

    \def\y{1.3em};
    \def\x{90};
  \nodelabel{$0$}
   \def\x{122.73};
   \def\y{1.7em};
   \nodelabel{$10$}
    \def\x{155.45};
      \def\y{2em};
   \nodelabel{$9$}
    \def\x{188.18};
    \def\y{2em};
   \nodelabel{$8$}
    \def\x{220.9};
      \def\y{1.7em};
    \nodelabel{$7$}
    \def\x{253.64};
   \nodelabel{$6$}
    \def\x{286.36};
    \def\y{1.5em};
   \nodelabel{$5$}
    \def\x{319.1};
    \def\y{1.5em};
   \nodelabel{$4$}
    \def\x{351.82};
    \def\y{1.5em};
    \nodelabel{$3$}
    \def\x{24.54};
    \def\y{1.6em};
    \nodelabel{$2$}
    \def\x{57.27};
    \def\y{1.5em};
   \nodelabel{$1$}

   \node at (O) {\Large{$t_1^{-3} t_2^{2}$}};
\end{tikzpicture}
\end{figure}
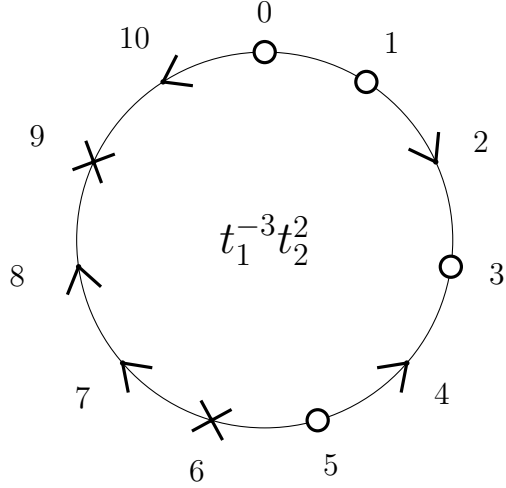

\begin{remark}\label{rem_recover_weight}
    An admissible weight $(\mu|\nu)$ can be recovered uniquely from its diagram:
\begin{itemize}
    \item The total number of vertices is $p =$ char $\ck$;
    \item $m$ is the total number of $\cw$ symbols (including those that are part of a $\times$ symbol);
    \item $n$ is the total number of $\ccw$ symbols (including those that are part of a $\times$ symbol).
    \item We recover $s$ and $r$ from the label $t_1^{-s}t_2^r$.
    \item Let $s=ms'+k, ~r=nr'+l$, where $0\le k\le m-1$ and $0\le l\le n-1$. Then 
    $$
    s_i = \begin{cases}
        s'+1, \text{ if } i\le k,\\
        s', \text{ if } i>k;
    \end{cases}
    $$
    and
    $$
    r_j = \begin{cases}
        r'+1, \text{ if } j> n-l,\\
        r', \text{ if } j\le n - l.
    \end{cases}
    $$
    \item Consider the set $A$ of all vertices labeled by $\cw$  (including those labeled by $\times$), so that $|A|=m$. Order the elements of $A=\{c_1,\ldots, c_m\}$ so that $c_1>c_2>\ldots>c_m$. Then
    $$
    a_i=\begin{cases}
       c_{m-k+i}, \text{ if } i\le k,\\
       c_{i-k}, \text{ if } i>k.
    \end{cases}
    $$
    \item Consider the set $B$ of all vertices labeled by $\ccw$  (including those labeled by $\times$), so that $|B|=n$. Order the elements of $B=\{d_1,\ldots, d_n\}$ so that $d_1<d_2<\ldots<d_n$. Then
    $$
    b_j=\begin{cases}
       d_{l+j}, \text{ if } j\le  n-l,\\
       d_{j+l-n}, \text{ if } j>n-l.
    \end{cases}
    $$
    \item We have 
    $$\mu_i =  i - 1 + (a_i+ p\cdot s_i),$$
    $$
    \nu_j = -m  + j - (b_j+ p\cdot r_j).
    $$
\end{itemize}
\end{remark}

\begin{ex}\label{ex_computing_weight_from_diag}
    In the diagram pictured in Figure \ref{fig:wd}, $m=5, n =4,$ $s=3=k, r=2=l$,
    $$
    A = \{9, 8, 7, 6, 2\},~B=\{4, 6, 9,10\}.
    $$
    We get that
    $$
    (a_1+p\cdot s_1,\ldots, a_5+p\cdot s_5) = (7+11, 6 + 11, 2+11, 9, 8) = (18,17, 13, 9, 8); 
    $$
    $$
    (b_1+p\cdot r_1,\ldots, b_4+p\cdot r_4) = (9, 10, 4+11, 6+11) = (9, 10, 15, 17).
    $$
    Therefore,
    $$
    \mu =(18, 18, 15, 12, 12),
    $$
    $$
    \nu = (-13, -13, -17, -18).
    $$
    
\end{ex}

Let us denote the weight diagram of $\bl := (\mu|\nu)$ by $f_\bl$. 

\begin{definition}
  For a weight diagram $f$, each $i$ and a pair of symbols $x,y\in\{<,>,\circ, \times\}$ denote by $f^{(i)}_{xy}$ the weight diagram obtained from $f$ by replacing the symbol attached to the $i$'th vertex in $f$ with $x$, and the symbol attached to the $(i+1)$'st vertex in $f$ (where $i+1$ is considered modulo $p$) with $y$.
  \end{definition}
  \begin{ex}
      Let $f$ be the weight diagram in Figure \ref{fig:wd}. Then the diagram $f^{(6)}_{\cw\times}$ is shown in Figure \ref{fig:changedwd}.
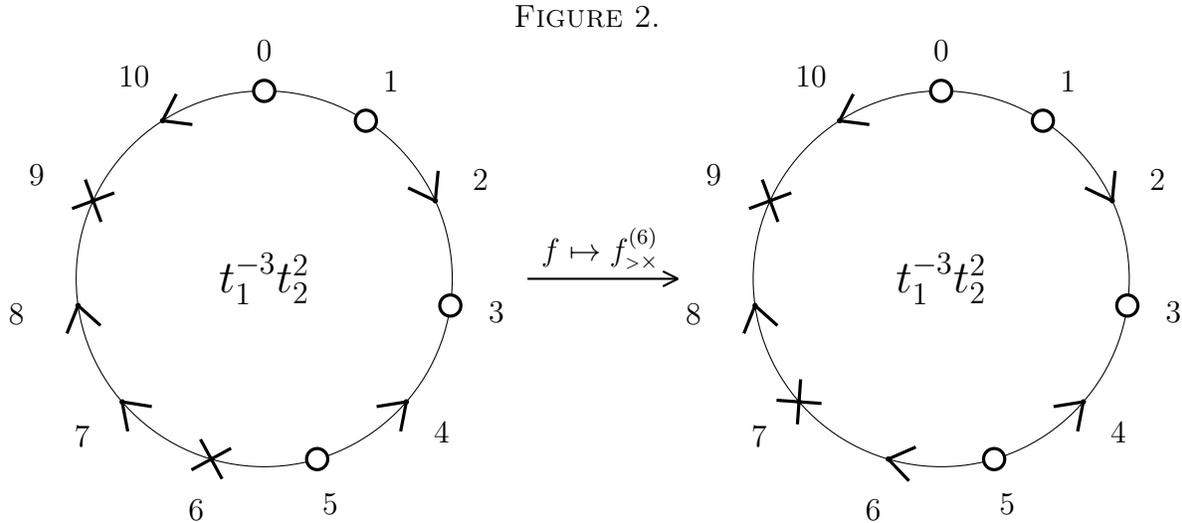
      \begin{figure}[ht]
      \caption{}
          \label{fig:changedwd}
      \begin{tikzpicture}
 \coordinate (O) at (1,2);
  \def\radius{2.5cm}

  \draw (O) circle[radius=\radius];

  \def\x{90};
\onode
    
   \def\x{122.73};
      \lenode
      
    \def\x{155.45};
   \xnode

    \def\x{188.18};
    \genode

    \def\x{220.9};
    \genode

    \def\x{253.64};
   \xnode

    \def\x{286.36};
  \onode

    \def\x{319.1};
  \lenode

    \def\x{351.82};
    \onode

    \def\x{24.54};
    \genode

    \def\x{57.27};
   \onode

    \def\y{1.3em};
    \def\x{90};
  \nodelabel{$0$}
   \def\x{122.73};
   \def\y{1.7em};
   \nodelabel{$10$}
    \def\x{155.45};
      \def\y{2em};
   \nodelabel{$9$}
    \def\x{188.18};
    \def\y{2em};
   \nodelabel{$8$}
    \def\x{220.9};
      \def\y{1.7em};
    \nodelabel{$7$}
    \def\x{253.64};
   \nodelabel{$6$}
    \def\x{286.36};
    \def\y{1.5em};
   \nodelabel{$5$}
    \def\x{319.1};
    \def\y{1.5em};
   \nodelabel{$4$}
    \def\x{351.82};
    \def\y{1.5em};
    \nodelabel{$3$}
    \def\x{24.54};
    \def\y{1.6em};
    \nodelabel{$2$}
    \def\x{57.27};
    \def\y{1.5em};
   \nodelabel{$1$}

   \node at (O) {\Large{$t_1^{-3} t_2^{2}$}};

   \draw[thick] (4.5,2)--(6.5,2);
    \draw[thick] (6.5,2) -- (6.3, 2.1);
    \draw[thick] (6.5,2) -- (6.3, 1.9);
 \node at (5.4, 2.4) {{ $f\mapsto f^{(6)}_{\cw\times} $ }};
 
  \coordinate (O) at (10,2);
  \def\radius{2.5cm};

  \draw (O) circle[radius=\radius];

  \def\x{90};
\onode
    
   \def\x{122.73};
      \lenode
      
    \def\x{155.45};
   \xnode

    \def\x{188.18};
    \genode

    \def\x{220.9};
    \xnode

    \def\x{253.64};
   \genode

    \def\x{286.36};
  \onode

    \def\x{319.1};
  \lenode

    \def\x{351.82};
    \onode

    \def\x{24.54};
    \genode

    \def\x{57.27};
   \onode

    \def\y{1.3em};
    \def\x{90};
  \nodelabel{$0$}
   \def\x{122.73};
   \def\y{1.7em};
   \nodelabel{$10$}
    \def\x{155.45};
      \def\y{2em};
   \nodelabel{$9$}
    \def\x{188.18};
    \def\y{2em};
   \nodelabel{$8$}
    \def\x{220.9};
      \def\y{1.7em};
    \nodelabel{$7$}
    \def\x{253.64};
   \nodelabel{$6$}
    \def\x{286.36};
    \def\y{1.5em};
   \nodelabel{$5$}
    \def\x{319.1};
    \def\y{1.5em};
   \nodelabel{$4$}
    \def\x{351.82};
    \def\y{1.5em};
    \nodelabel{$3$}
    \def\x{24.54};
    \def\y{1.6em};
    \nodelabel{$2$}
    \def\x{57.27};
    \def\y{1.5em};
   \nodelabel{$1$}

   \node at (O) {\Large{$t_1^{-3} t_2^{2}$}};
\end{tikzpicture}
 \end{figure}
      
  \end{ex}
  \begin{remark}
      Recall that $\cw$ (resp. $\ccw$) denotes the clockwise (resp. counterclockwise) facing arrow. Thus, in our notation, $(xy)$ depicts the view on vertices $i$ and $i+1$ in $f^{(i)}_{xy}$ from the center of the circle.
  \end{remark}

  Identifying $\bl=(\mu|\nu)$ with the vector $v_\mu\otimes v_\nu^*$ in $\widehat{\ssl}_p$-module  $(\Lambda^m U)^{loop}\otimes (\Lambda^nU^*)^{loop}$, we get the action of operators $F_i, E_i$\footnote{The same operators are denoted by $f_i, e_i$ when acting on the vectors of $(\Lambda^m U)^{loop}\otimes (\Lambda^nU^*)^{loop}$. We will use upper-case letters for the action on weight diagrams to avoid confusion between $f_i$ and $f_\bl$.}, $0\le i\le p-1$ on the set of $\C$-linear combinations of weight diagrams $f_\bl$. Let us describe this action.
\begin{theorem}\label{thm_action_on_wd}
     Let $f$ be a weight diagram, let $0\le i\le p-1$ and let $\epsilon_i=\delta_{i, p-1}$. Then
     $F_i(f)$ is equal to 
      \begin{itemize}
          \item $f^{(i)}_{\circ\cw}\cdot t_1^{-\epsilon_i},~ \text{if } f = f^{(i)}_{\cw\circ}$,
          \medskip
          \item $f^{(i)}_{\ccw \circ}\cdot t_2^{-\epsilon_i},~ \text{if } f = f^{(i)}_{\circ \ccw}, $
          \medskip
          \item $f^{(i)}_{\circ\times}\cdot t_1^{-\epsilon_i} +f^{(i)}_{\times\circ}\cdot t_2^{-\epsilon_i},~ \text{if } f = f^{(i)}_{\cw\ccw}, $
          \medskip
         \item $f^{(i)}_{\ccw\times}\cdot t_1^{-\epsilon_i}, ~\text{if } f = f^{(i)}_{\times\ccw}, $
          \medskip
         \item $f^{(i)}_{\times\cw}\cdot t_2^{-\epsilon_i},~ \text{if } f = f^{(i)}_{\cw\times}, $  
          \medskip       
         \item $f^{(i)}_{\ccw \cw}\cdot t_1^{-\epsilon_i},~ \text{if } f = f^{(i)}_{\times \circ}, $
          \medskip
         \item $f^{(i)}_{\ccw\cw}\cdot t_2^{-\epsilon_i},~ \text{if } f = f^{(i)}_{\circ\times} $,
          \medskip
        \item $0$ otherwise.
        \end{itemize}
        
        Similarly, $E_i(f)$ is equal to
        \begin{itemize}
          \item $f^{(i)}_{\circ\ccw}\cdot t_2^{\epsilon_i},~ \text{if } f = f^{(i)}_{\ccw\circ}$,
          \medskip
          \item $f^{(i)}_{\cw \circ}\cdot t_1^{\epsilon_i},~ \text{if } f = f^{(i)}_{\circ \cw}, $
          \medskip
          \item $f^{(i)}_{\times\circ}\cdot t_1^{\epsilon_i}+f^{(i)}_{\circ\times}\cdot t_2^{\epsilon_i},~ \text{if } f = f^{(i)}_{\ccw\cw}, $
          \medskip
         \item $f^{(i)}_{\cw\times}\cdot t_2^{\epsilon_i}, ~\text{if } f = f^{(i)}_{\times\cw}, $
          \medskip
         \item $f^{(i)}_{\times\ccw}\cdot t_1^{\epsilon_i},~ \text{if } f = f^{(i)}_{\ccw\times}, $     
          \medskip    
         \item $f^{(i)}_{\cw \ccw}\cdot t_2^{\epsilon_i},~ \text{if } f = f^{(i)}_{\times \circ}, $
          \medskip
         \item $f^{(i)}_{\cw\ccw}\cdot t_1^{\epsilon_i},~ \text{if } f = f^{(i)}_{\circ\times} $,
          \medskip
        \item $0$ otherwise.
        \end{itemize}
\end{theorem}

\begin{remark}
    The rule of thumb for this action is that both $F_i$ and $E_i$ only affect vertices $i$ and $i+1$ of the weight diagram. While $F_i$ moves the arrows one step in the direction they are facing, $E_i$ moves the arrows in the opposite direction.

    Note that arrows have fermionic behavior: two arrows facing the same direction are not allowed to sit on the same vertex. That is why if, for instance, $f = f^{(i)}_{\cw\cw}$ then $F_i(f) = 0$.

    A factor of $t_1^{-1}$ is added when $\cw$ is moved from $p-1$ to $0$ (this is because the corresponding $s_i$ increases by one). When $\cw$ is moved from $0$ to $p-1$, we add a factor of $t_1$. Similarly, moving $\ccw$ from $p-1$ to $0$ results in  factor of $t_2$, and moving it from $0$ to $p-1$  produces a factor of $t_2^{-1}$.
    
\end{remark}

\begin{definition}\label{def_atypical}
    Let $\bl$ be an admissible weight and $f_\bl$ its weight diagram. The number of $\times$'s in $f_\bl$ is called the \textbf{degree of atypicality} of $\bl$ denoted $\#\bl$.

    If $\#\bl = 0$ we say $\bl$ is \textbf{typical}.
\end{definition}

The results of Theorem \ref{thm_action_on_wd} imply that for any translation functor $$T\in\{F_i, E_i~|~0\le i \le p-1\}$$ and any weight $\bl$, as long as $Tf_\bl\neq 0$, we can define $\#T\bl$ to be the degree of atypicality of all diagrams in $Tf_\bl$ (one can see that $Tf_\bl$ is either a single diagram or a sum of two diagrams of the same degree of atypicality).
 
\begin{theorem}\label{thm_transl_of_kac}
    Let $T\in \{F_i, E_i~|~0\le i\le p-1\}$ be one of the translation functors.
    
    \begin{enumerate}
        \item Suppose
    $$
    Tf_\bl = f_\bm.
    $$
    Then $\#\bl = \#\bm$ and
    $$
    TK(\bl) = K(\bm).
    $$
    \item Suppose 
    $$
    Tf_\bl = f_\bm+f_\bn.
    $$
    Then $\#\bl < \#\bm = \#\bn$, $\bm\prec \bn$ (or the other way around) and $TK(\bl)$ is an extension between $K(\bm)$ and $K(\bn)$:
    $$
    0\to K(\bn)\to TK(\bl)\to K(\bm)\to 0.
    $$
     \end{enumerate}
   
\end{theorem}
\begin{proof}
    It follows directly from the proof of Theorem \ref{tcatact} and from Theorem \ref{thm_action_on_wd}.
\end{proof}

\begin{lemma}\label{l_atyp}
    The degree of atypicality $\#\bl$ of $\bl$ measures the number of positive odd roots $\ba$, for which
    $$
    \langle \bl +\br^{(m|n)}, \ba\rangle \equiv 0\mod p.
    $$
\end{lemma}
\begin{proof}
    We have $\ba= \epsilon_i-\delta_j$ for some $1\le i\le m, 1\le j\le n$. By Remark \ref{rem_ab},
    $$
    \langle \bl +\br^{(m|n)}, \epsilon_i\rangle \equiv \langle\bl +\br^{(m|n)}, \delta_j\rangle \mod p
    $$
    if and only if 
    $$
    a_i=b_j.
    $$
\end{proof}

\begin{theorem}\label{thm_irr_Kac}
    Kac module $K(\bl)$ is irreducible (and projective) if and only if $\bl$ is typical.
\end{theorem}
\begin{proof}
    This follows immediately from Corollary \ref{cor_irr_kac_critereon} and Lemma \ref{l_atyp}. Namely, we proved that $K(\bl)$ is irreducible if and only if for all $1\le i\le m, 1\le j\le n$ 
    $$
    \langle \bl+\br^{(m|n)}, \epsilon_i-\delta_j\rangle \not\equiv 0 \mod p. 
    $$
\end{proof}

\subsection{The action of translation functors on indecomposable projective modules}\label{s_act_on_ipm}

Our goal in this section, following the idea of Brundan in \cite{B03}, is to construct each indecomposable projective module $P(\bl)$ from a typical one $P(\bm)\simeq K(\bm)$ using translation functors. This will allow us to determine the multiplicities $[P(\bl):K(\ba)]$ in the standard filtration of $P(\bl)$. A lot of arguments in this section are similar to the ones in \cite{GS11}.

Let us start with a few useful lemmas.

\begin{lemma}\label{l_deg_one_submod}
    Suppose $\bl\in \Rep_{Ver_p}(T,\varepsilon)$ and assume that $\bm = \bl - \ba_{i,j}$ is admissible for some odd positive root $\ba_{i,j}=\epsilon_i-\delta_j$, for which
    $$
    \langle \bl +\br^{(m|n)},\ba_{i,j}\rangle \equiv 0\mod p.
    $$
     Then 
    $$
    \dim\Hom_T(\bm, K(\bl))=1
    $$
    and $\bm$ generates a proper $Dist(G)$-submodule in $K(\bl)$. Consequently,
    $$
    [K(\bl):L(\bm)] = 1 = [P(\bm):K(\bl)].
    $$
\end{lemma}
\begin{proof}
First, notice that $\langle \ba_{i,j},\ba_{i,j}\rangle = 0$, so
$$
\langle \bm +2\br^{(m|n)}, \bm\rangle = \langle \bl - \ba_{i,j} +2\br^{(m|n)}, \bl-\ba_{i,j}\rangle =
$$
$$
=\langle \bl +2\br^{(m|n)}, \bl\rangle -2\langle \bl+\br^{(m|n)}, \ba_{i,j}\rangle +\langle \ba_{i,j},\ba_{i,j}\rangle   = \langle \bl +2\br^{(m|n)}, \bl\rangle.
$$
By Theorem \ref{tcasimir}, this means that the Casimir element $C_X$ acts with the same constant on $K(\bl)$ and $K(\bm)$. Note that this condition makes sense without the supergroup convention introduced in Section \ref{sec_super}. For simplicity, we will drop this convention and assume $X = L_m \oplus L_r$, where $r=p-n$ and $G=GL(X)=GL(L_m|L_n)=GL(L_m\oplus L_r)$.

     Since $\bm$ is admissible, it is a weight of
     $$
     (V^{(m)})^*\boxtimes V^{(r)}\otimes \bl = \mathfrak n^- \otimes \bl = K(\bl)^1,
     $$
     the degree one component of $K(\bl)$. The rules of the tensor product in $Ver_p(GL_m)\boxtimes Ver_p(GL_r)$ ensure that the multiplicity of $\bm$ in $K(\bl)^1$ is precisely one.

     Now, $\bm$ generates a proper submodule in $K(\bl)$ if and only if the restriction of the action map
     $$
     a:\mathfrak n^+\otimes K(\bl)^1 \to K(\bl)
     $$
    to $\bm$ is zero. By degree consideration, its image lands in $K(\bl)^0 = \bl$. Identifying $\n^+ = (\n^-)^*$ and $K(\bl)^1=\n^-\otimes\bl$, we can dualize the map above and get:
    $$
    Sh^1: \n^-\otimes \bl \to \n^-\otimes \bl,
    $$
    which is the map defined in Remark \ref{rem_Shapovalov}.

    Thus, $\bm$ generates a proper submodule if and only if $Sh^1|_\bm = 0$.

    Let $\bl = V_{\alpha}^{(m)}\boxtimes V_\beta^{(r)}\in \Rep_{Ver_p}(T,\varepsilon)= Ver_p(GL_m)\boxtimes Ver_p(GL_r)$. Then
    $$
    \bm = V_{\alpha-e_i}^{(m)}\boxtimes V_{\beta+e_j}^{(r)}
    $$
    The action map $a$ can be written as
    $$
    a:[V^{(m)}\boxtimes (V^{(r)})^*]\otimes [(V^{(m)})^*\boxtimes V^{(r)}]\otimes  [V^{(m)}_\alpha \boxtimes V^{(r)}_\beta] \to V^{(m)}_\alpha\boxtimes V^{(r)}_\beta,
    $$
    $$
    a = a_\alpha \circ ev^{(r)}_{24} - (a_\beta\circ ev^{(m)}_{13})\circ \tau_{12,34},
    $$
where 
    $$
    \tau_{12,34}: [V^{(m)}\boxtimes (V^{(r)})^*]_{12}\otimes [(V^{(m)})^*\boxtimes V^{(r)}]_{34}\to [(V^{(m)})^*\boxtimes V^{(r)}]_{34}  \otimes [V^{(m)}\boxtimes (V^{(r)})^*]_{12}
    $$
    is the braiding, and for $k=m,r$ and $\gamma=\alpha, \beta$ we define $ev^{(k)}$
    $$
    ev^{(k)} = ev_{V^{(k)}}: (V^{(k)})^*\otimes V^{(k)}\to \on
    $$
    to be the evaluation map for $V^{(k)}$ and $a_\gamma$ to be the action map
    $$
    a_\gamma: \gl(V^{(k)})\otimes V_\gamma^{(k)} \to V_\gamma^{(k)}
    $$
    of $\gl(V^{(k)})=V^{(k)}\otimes (V^{(k)})^*$ on $V_\gamma^{(k)}.$
    
     Dualizing, we get an expression for 
     $$
     Sh^1: [(V^{(m)})^*\boxtimes V^{(r)}]\otimes  [V^{(m)}_\alpha \boxtimes V^{(r)}_\beta]\to [(V^{(m)})^*\boxtimes V^{(r)}]\otimes  [V^{(m)}_\alpha \boxtimes V^{(r)}_\beta],
     $$
     $$
     Sh^1 = x_{\alpha}^* -  x_{\beta},
     $$
     where $x_\gamma, x^*_\gamma$ are two dualizations of the action map $a_\gamma$ on $V_\gamma^{(k)}$:
     $$
     x_\gamma: V^{(k)}\otimes V_\gamma^{(k)}\to V^{(k)}\otimes V_\gamma^{(k)}, 
     $$
     $$
    x^*_\gamma: (V^{(k)})^*\otimes V_\gamma^{(k)}\to (V^{(k)})^*\otimes V_\gamma^{(k)} 
     $$
  Thus $x_\gamma$ is the tensor Casimir (see Section \ref{sec_catact}):
  $$
  x_\gamma = \frac{1}{2}(\Delta\circ C - C \otimes 1 - 1\otimes C)|_{V_\gamma\otimes V}.
  $$
  We can express $x^*_\gamma$ in terms of $x_\gamma$ as follows. Let $V_{\gamma^*}=V_\gamma^*$ then 
  $$
  x_\gamma^* = - (x_{\gamma^*})^t,
  $$
  where  for a map $f:A\to B$ we denote by $f^t$ its transpose:
  $$
  f^t: B^*\to A^*.
  $$
    
     Now, since
     $$
     \alpha^* = (-\alpha_m, -\alpha_{m-1},\ldots, -\alpha_2,-\alpha_1),
     $$
     we have $(\alpha-e_i)^*=\alpha^* + e_{m+1-i}$. Thus the restriction of 
     $x_{\alpha^*}$ to $V_{(\alpha-e_i)^*}^{(m)}$ is the  multiplication by
     $$
     c = \alpha_{m+1-i}^* +1 - (m+1-i) = -\alpha_i + i - m.
     $$
Whereas restriction of
     $x^{(r)}_\beta$ to $V_{\beta+e_j}^{(r)}$ is multiplication by
     $$
     d = \beta_j + 1 - j.
     $$
That is, $Sh^1$ is zero if and only if $-c-d=0$.

Recall now our condition that $C_X$ acts with the same constant on $K(\alpha,\beta)$ and \\ $K(\alpha-e_i, \beta+e_j)$. By Theorem \ref{tcasimir}, this is equivalent to
$$
\langle \alpha + 2\rho^{(m)},\alpha\rangle + \langle \beta+2\rho^{(r)},\beta\rangle +r|\alpha|-m|\beta| =
$$
$$
= \langle \alpha - e_i + 2\rho^{(m)},\alpha - e_i\rangle + \langle \beta + e_j+2\rho^{(r)},\beta + e_j\rangle + r|\alpha-e_i|- m|\beta+e_j|.
$$
Thus,
$$
2\langle e_i, \alpha+\rho^{(m)}\rangle - 2\langle e_j, \beta+\rho^{(r)}\rangle = \langle e_i, e_i\rangle + \langle e_j , e_j\rangle - r -m.
$$
So,
$$
2\alpha_i + m-2i+1 - 2\beta_j - r + 2j - 1 = 2-r-m,
$$
and we obtain
$$
m+\alpha_i-i  = \beta_j-j+1,
$$
i.e.
$$
-c = d.
$$
\end{proof}

\begin{lemma}\label{l_transl_of_simples}
    Let $T\in \{F_i, E_i~|~0\le i\le p-1\}$ be one of the translation functors. Suppose $Tf_\bl\neq 0$ with $\#\bl \ge \#T\bl$. Then
   \begin{enumerate}
       \item Either $T$ maps irreducible $L(\bl)$ to zero or
        $$
        TL(\bl) = L(\bm)
        $$
        is irreducible and $Tf_\bm = f_\bl$.

        \item If $TL(\bl) = TL(\bn)\neq 0$ then $\bl=\bn$.
    \end{enumerate}
\end{lemma}
\begin{proof}
    \begin{enumerate}
        
        \item First note that since $~\# T\bl \le \#\bl$, we must have
        $$
        Tf_\bl = f_\bm
        $$
        for some $\bm$ (as $Tf_\bl=f_\bm+f_\bn$ only when $\#\bl < \#\bm=\#\bn$). 
        
       Moreover, observing the rules of action described in Theorem \ref{thm_action_on_wd}, we deduce that if
       $$
       T^*f_\ba = f_\bl + \text{ possibly some other term},
       $$
       where $T^*$ is the adjoint functor to $T$ (i.e. $E_i$ for $F_i$ and $F_i$ for $E_i$), then 
       $$
       \ba = \bm.
       $$

       Now, all $\n^+$-singular weights in $TL(\bl)$  produce nonzero maps in 
       $$
       \Hom_G(K(\ba), TL(\bl)) = \Hom_G(T^* K(\ba), L(\bl)).
       $$
        Note that $T^* K(\ba)$ admits a nonzero map to $L(\bl)$ if and only if $f_\bl$ is a summand in $T^*f_\ba$ (see Theorem \ref{thm_transl_of_kac}). Thus $\bm$ is the only singular weight of $TL(\bl)$ and so 
        $$
        TL(\bl)=L(\bm).
        $$

        \item Suppose $TL(\bl)=TL(\bn)=L(\bm)$ then
        $$
        Tf_\bl = Tf_\bn = f_\bm.
        $$
        Assume further that $\bl\neq\bn$. Considering the cases in Theorem \ref{thm_action_on_wd}, we see that this happens only if for some $k$:
        $$
        (f_\bl)^{(k)}_{xy} = (f_\bn)^{(k)}_{xy}
        $$
        for any $x,y$, and
        $$
        (f_\bl)^{(k)}_{\times\circ} = f_\bl, 
        $$
        $$
        (f_\bn)^{(k)}_{\circ\times}=f_\bn,
        $$
        or vice versa. 
        
        Without loss of generality, assume $(f_\bl)^{(k)}_{\circ\times}=f_\bl$ and $(f_\bn)^{(k)}_{\times\circ}=f_\bn$.
        
        We get that
        $$
        \bl = \bn + \ba_{i,j}
        $$
        for some positive odd root $\ba_{i,j}=\epsilon_i-\delta_j$, responsible for the cross $\times$ on the  $(k+1)$'st vertex of $f_\bl$, that is
        $$
        \langle \bl+\rho^{(m|n)},\ba_{i,j}\rangle = 0
        $$
        (see Lemma \ref{l_atyp}).

        Then $\bn = \bl - \ba_{i,j}$ with $\langle \bl+\rho^{(m|n)},\ba_{i,j}\rangle = 0$, so by Lemma \ref{l_deg_one_submod}, we get that 
        $$
        [K(\bl):L(\bn)] = 1.
        $$

        On the other hand
        $$
        TK(\bl)=K(\bm),
        $$
        so we get that unless $TL(\bl)$ or $TL(\bn)$ is zero, the simple module $L(\bm)$ appears in the composition series of $K(\bm)$ at least twice, which is a contradiction.
       \end{enumerate}
\end{proof}

\begin{theorem}\label{thm_indec_proj}
    Let $T\in \{F_i, E_i~|~0\le i\le p-1\}$ be one of the translation functors. Suppose $Tf_\bl \neq 0$ and $\#\bl \le \#T\bl$. Then $TP(\bl) = P(\bm)$ is an indecomposable projective module, with 
    $$
        T f_\bl = f_\bm, \text{ if } \#\bl = \#T\bl,
        $$
        or 
        $$Tf_\bl = f_\bm + f_\bn, \text{ where } \bm\prec \bn, \text{ if } \#\bl < \# T\bl.
        $$
    \end{theorem}
    \begin{proof}
        The fact that $TP(\bl)$ is projective follows from Theorem \ref{thm_projfun}. 
        
        Now we need to show that
        $$
        \Hom_G(TP(\bl), L(\ba))\neq 0
        $$
        for a single weight $\ba$ only. Let $T^*$ be the adjoint functor ($E_i$ for $F_i$ and $F_i$ for $E_i$). We have
        $$
        \Hom_G(TP(\bl), L(\ba))= \Hom_G(P(\bl), T^* L(\ba)).
        $$

        Now by Lemma \ref{l_transl_of_simples}, $T^*L(\ba)$ is either simple or zero, and is isomorphic to $L(\bl)$ for precisely one weight $\ba=\bm$. Thus $TP(\bl)=P(\bm)$.
        
        For this weight we have
        $$
        T^*f_\bm = f_\bl,
        $$
        so observing the rules in Theorem \ref{thm_action_on_wd}, we deduce that
        $$
        T f_\bl = f_\bm + \text{possibly some other term}.
        $$

        Finally, if $Tf_\bl = f_\bm + f_\bn$ then $T^*K(\bm)=T^*K(\bn)= K(\bl)$. We have
        $$
        \dim \Hom_G(TP(\bl), K(\bn)) = \dim \Hom_G(P(\bl), K(\bl)) = 1.
        $$
        On the other hand,
        $$
        \dim\Hom_G(TP(\bl), K(\bn))=\dim\Hom_G(P(\bm),K(\bn)) = [K(\bn):L(\bm)].
        $$
        So $\bm\prec \bn$.
    \end{proof}

Before proving the main result of this section, let us introduce a few convenient notations.

\subsubsection{Cutting up diagrams}
Given a weight diagram $f$ and a vertex $k$ with $0\le k\le p-1$, we can turn $f$ into a linear diagram by cutting the circle at vertex $k$ and unwrapping it into an interval on the line while aligning the clockwise direction with the left-to-right direction on the line. We will drop the labels of the vertices and the line from the notation, so that only a string of $p$ symbols is left.  To keep track of the label $t_1^{-s} t_2^r$  we will put $\cdot (t_1^{-s}t_2^r)$ at the tail of the string  of symbols we obtain. The resulting diagram is denoted $f_{\downarrow k}$. The first symbol in $f_{\downarrow k}$ is the one on vertex $k$ in $f$, the second is the one on vertex $(k+1)$, and so on. The diagram $f_{\downarrow k}$ preserves all the information in $f$ as long as $k$ is speciefied.
\begin{ex}
    Let $f$ be the weight diagram in Figure \ref{fig:wd}. Then 
    $$
    f_{\downarrow 3} = \krug~ \ccw~\krug~\times ~\cw~ \cw~ \times~ \ccw ~\krug~ \krug ~\cw\cdot (t_1^{-3}t_2^2).
    $$
\end{ex}

\subsubsection{Acting by permutations on weight diagrams}
For any permutation $\sigma\in S_p$, where $S_p$ is the group of permutations of the set $\{0,1,\ldots, p-1\}$, and a weight diagram $f$ define the action of $\sigma$ on $f$ as follows. The symbol attached to vertex $k$ in $\sigma f$ is the same as symbol on vertex $\sigma(k)$ in $f$. 

It is easy to see that $\sigma f$ is a weight diagram for the same $m,n$ and with the same degree of atypicality as $f$.

Moreover, we get an action of the group algebra $\C[t_1^{\pm 1}, t_2^{\pm 1}]S_p$ on $\C$-linear combinations of weight diagrams (with multiplication by $t_1^{\pm 1}$ and $t_2^{\pm 1}$, as before, affecting only the label inside the diagram). 

\begin{ex}
    Let $\sigma = (0~4~6)(3~9)\in S_{11}$ and let $f$ be the weight diagram in Figure \ref{fig:wd}. Then the diagram $t_1 t_2^{-1}\cdot \sigma f$ is shown in Figure \ref{fig:permwd}.

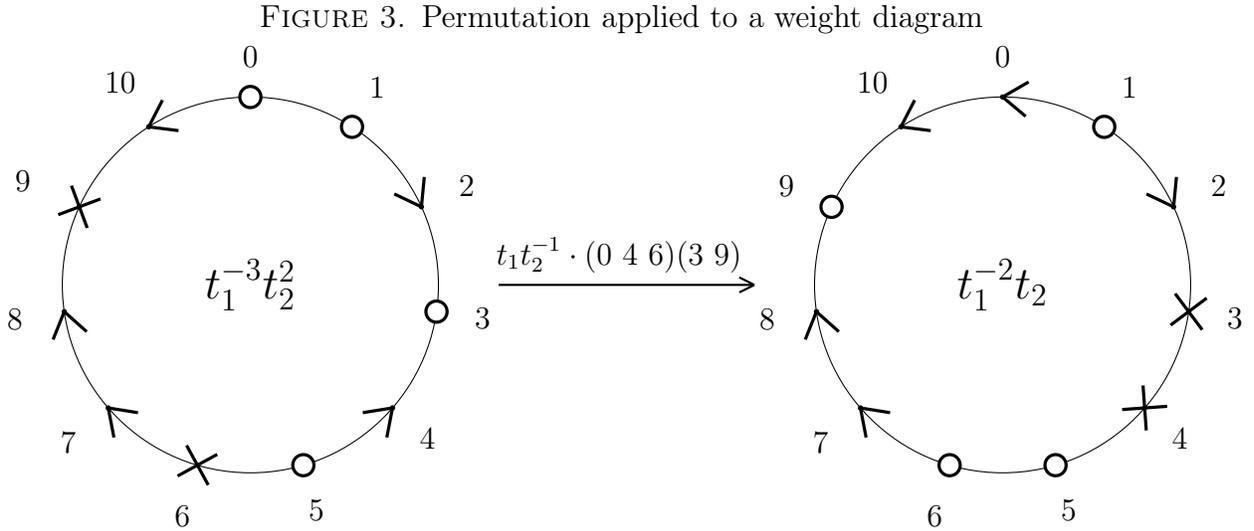
    \begin{figure}[h]
    \caption{Permutation applied to a weight diagram}
    \label{fig:permwd}

       \begin{tikzpicture}

  \coordinate (O) at (0.5,2);
  \def\radius{2.5cm}

  \draw (O) circle[radius=\radius];

  \def\x{90};
\onode
    
   \def\x{122.73};
      \lenode
      
    \def\x{155.45};
   \xnode

    \def\x{188.18};
    \genode

    \def\x{220.9};
    \genode

    \def\x{253.64};
   \xnode

    \def\x{286.36};
  \onode

    \def\x{319.1};
  \lenode

    \def\x{351.82};
    \onode

    \def\x{24.54};
    \genode

    \def\x{57.27};
   \onode

    \def\y{1.3em};
    \def\x{90};
  \nodelabel{$0$}
   \def\x{122.73};
   \def\y{1.7em};
   \nodelabel{$10$}
    \def\x{155.45};
      \def\y{2em};
   \nodelabel{$9$}
    \def\x{188.18};
    \def\y{1.6em};
   \nodelabel{$8$}
    \def\x{220.9};
      \def\y{1.7em};
    \nodelabel{$7$}
    \def\x{253.64};
   \nodelabel{$6$}
    \def\x{286.36};
    \def\y{1.5em};
   \nodelabel{$5$}
    \def\x{319.1};
    \def\y{1.5em};
   \nodelabel{$4$}
    \def\x{351.82};
    \def\y{1.5em};
    \nodelabel{$3$}
    \def\x{24.54};
    \def\y{1.6em};
    \nodelabel{$2$}
    \def\x{57.27};
    \def\y{1.5em};
   \nodelabel{$1$}

   \node at (O) {\Large{$t_1^{-3} t_2^{2}$}};

   \draw[thick] (3.8,2)--(7.2,2);
    \draw[thick] (7.2,2) -- (7, 2.1);
    \draw[thick] (7.2,2) -- (7, 1.9);
 \node at (5.4, 2.4) {{$t_1t_2^{-1}\cdot (0~4~~6)(3~9)$ }};
  \coordinate (O) at (10.5,2);
  \def\radius{2.5cm}

  \draw (O) circle[radius=\radius];

  \def\x{90};
\lenode
    
   \def\x{122.73};
      \lenode
      
    \def\x{155.45};
   \onode

    \def\x{188.18};
    \genode

    \def\x{220.9};
    \genode

    \def\x{253.64};
   \onode

    \def\x{286.36};
  \onode

    \def\x{319.1};
  \xnode

    \def\x{351.82};
    \xnode

    \def\x{24.54};
    \genode

    \def\x{57.27};
   \onode

    \def\y{1.3em};
    \def\x{90};
  \nodelabel{$0$}
   \def\x{122.73};
   \def\y{1.7em};
   \nodelabel{$10$}
    \def\x{155.45};
      \def\y{1.6em};
   \nodelabel{$9$}
    \def\x{188.18};
    \def\y{1.6em};
   \nodelabel{$8$}
    \def\x{220.9};
      \def\y{1.7em};
    \nodelabel{$7$}
    \def\x{253.64};
   \nodelabel{$6$}
    \def\x{286.36};
    \def\y{1.5em};
   \nodelabel{$5$}
    \def\x{319.1};
    \def\y{1.5em};
   \nodelabel{$4$}
    \def\x{351.82};
    \def\y{1.5em};
    \nodelabel{$3$}
    \def\x{24.54};
    \def\y{1.6em};
    \nodelabel{$2$}
    \def\x{57.27};
    \def\y{1.5em};
   \nodelabel{$1$}

   \node at (O) {\Large{$t_1^{-2} t_2$}};
\end{tikzpicture}
\end{figure}
\end{ex}

For any $\times$ or $\circ$ in a weight diagram there exists a translation functor that swaps it with its clockwise neighbor $\ccw$ or $\cw$. If the swap happens between vertices $p-1$ and $0$ then some factor of the form $t_1^at_2^b$ is gained. 

That is, if $f_\bl = (f_\bl)^{(k)}_{xy}$ with $x\in \{\times, \circ\}$, $y\in\{\cw, \ccw\}$ then there exists $T\in \{F_k, E_k
\}$ so that
$$
Tf_\bl =  \sigma_k f_\bl \cdot t_{1}^{a}t_2^b= 
(f_\bl)^{(k)}_{yx}\cdot t_{1}^{a}t_2^b
$$
(for some $a,b\in \{-1,0,1\}$), where $\sigma_k\in S_p$ is the transposition $k\to k+1, k+1\to k$.

By Theorem \ref{thm_indec_proj}, for such $\bl$ we have $P(\bl) = T^*P(T\bl)$ as $\#\bl = \#T\bl$.

Let us denote this transformation $T:(xy)\to (yx)$ and its adjoint by $T^*:(yx)\to (xy)$.

\subsubsection{Cap diagrams} 


Let $\bl\in \Z^{m|n}$ be any admissible weight, and let $f_\bl$ be its weight diagram. By our assumption, $m+n<p$. This implies that the number of $\circ$'s in $f_\bl$ is strictly higher than the number of $\times$'s, and is thus strictly positive. 

\begin{definition}
 Let $f$ be a weight diagram. The \textbf{cap diagram} of $f$ is a decoration of $f$ with \textbf{caps}, i.e. arched oriented intervals on the outside of the circle with endpoints at vertices, defined as follows. 
 \begin{enumerate}
     \item Each cap has a \textbf{source}, which must be labeled by $\times$, and a \textbf{tail}, which must be labeled  by $\circ$.

     \item Every $\times$ in $f$ must be a source of precisely one cap.

     \item We start with a vertex $k$ labeled by $\times$, which is not yet involved in a cap, and put $s = k$. If there are no such vertices the cap diagram is complete.

     \item We move in the clockwise direction until we reach $\times$ or $\circ$ that is not yet involved in any cap.

     \item If we first reach $\times$ at vertex $k'$ then we put $s = k'$ and go back to step $(4)$.

     \item If we reach $\circ$ at vertex $l$, we put $z = l$.

     \item We draw a cap with a source at $s$ and a tail at $z$. Then go back to step $(3)$.
 \end{enumerate}
\end{definition}

It is easy to see from the definition that cap diagrams have the following properties:
\begin{itemize}
\item Each weight diagram $f$ has a unique cap diagram assigned to it. We denote it by $\caap f$.
\item The caps do not intersect. That is, if there are two caps in $\caap f$, one with the source $s_1$ and the tail $z_1$\footnote{We wish we could name the tails $t_1$ and $t_2$, but we already used it for the variables $t_1, t_2$.}, and the other with the source $s_2$ and the tail $z_2$ then clockwise intervals $(s_1,z_1)$ and $(s_2,z_2)$ on a circle are either embedded or disjoint. 

We will call a cap \textbf{inner} if there are no other caps embedded underneath it.

\item As there are strictly more $\circ$'s than $\times$'s in $f$, there must be at least one $\circ$ not involved in any cap. We will call such $\circ$ \textbf{free}.

\item If there is a cap with the source $s$ and tail $z$ then there are no free $\circ$'s 
 in the clockwise interval $(s,z)$.

 \item If vertex $k$ is labeled with a free $\circ$ then the ray connecting the center of the circle with $k$ does not intersect any of the caps. Consequently, we can cut up the diagram $\caap f$ at vertex $k$ (without cutting any caps).
\end{itemize}

\begin{ex}\label{ex_cap}
    Let $f$ be the weight diagram in Figure \ref{fig:wd}. Then
    
    \medskip 
    \[
    \xymatrixcolsep{0.1mm}
    \xymatrixrowsep{1mm}
    \caap f_{\downarrow 3} = \xymatrix{\circ & \ccw&\circ&\times \ar@/^1.9pc/ [rrrrrr]!<4pt,0pc>& \cw & \cw& \times\ar@/^1.1pc/[rr]!<2pt,0pc>& \ccw &\circ& \circ &\cw\cdot (t_1^{-3}t_2^2)}.
    \]

    The cap with the source at $9$ and the tail at $0$ is inner, whereas the cap with the source at $6$ and the tail at $1$ is not.
\end{ex}

     Let $\bl\in \Rep_{Ver_p}(T,\varepsilon)$  be a weight with $\#\bl = r$ and let $\caap f_\bl$ be the corresponding cap diagram. Let $s_1, s_2, \ldots, s_r\in \{0,\ldots, p-1\}$ be  the sources of all the caps in $\caap f_\bl$ (the numbers of vertices with $\times$'s on them) and let $z_1,\ldots, z_r\in \{0,\ldots, p-1\}$ be the corresponding tails. 

     Define
     $$
     \tau_j = t_1^{a_j}t_2^{b_j}\cdot (s_j~z_j)\in \C[t_1^{\pm 1}t_2^{\pm 1}]S_p,
     $$
     where 
     $$
     (a_j,b_j) = \begin{cases}
         (-1,1), \text{ if } s_j > t_j,\\
         (0,0), \text{ if } s_j<t_j.
     \end{cases}
     $$

     \begin{definition}\label{def_plambda}
     The set $\mathscr P(\bl)$ consists of all weights that can be obtained from $\bl$ by swapping some of the $\times$'s in $f_\bl$ with $\circ$'s at the tails of the corresponding caps in $\caap f_\bl$ (with an appropriate factor $t_1^at_2^b$). That is,
     $$
     \mathscr P(\bl) = \{\tau_{i_1}\tau_{i_2}\ldots \tau_{i_l}\bl~|~0\le l \le r=\#\bl, 1\le i_1<i_2<\ldots < i_l\le r\}.
     $$
\end{definition}

\begin{remark}
    The appearance of the factor $t_1^{a_j} t_2^{b_j}$ in $\tau_j$, can be understood using the following intuition. We move $\times$ from $s_j$ to $t_j$ clockwise one step at a time. If $s_j>t_j$ then we must move it from $p-1$ to $0$ at some point. Moving $\cw$ and $\ccw$ that constitute our $\times$ from $p-1$ to $0$ must result in a factor of $t_1^{-1}t_2$.
\end{remark}

By definition, $\mathscr P(\bl)$ has $2^{\#\bl}$ elements.

\begin{ex}
    Let $f$ be the weight diagram in Figure \ref{fig:wd} and let $\bl$ be such that $f=f_\bl$. The cap diagram for $f$ is shown in Example \ref{ex_cap}. Let us order $\times$'s in $f$ so that $s_1 = 9, t_1=0, s_2=6, t_2 = 1$
    
    The set $\mathscr P(\bl)$ consists of four weights with the following weight diagrams:
\[
    \xymatrixcolsep{0.1mm}
    \xymatrixrowsep{1mm}
    \xymatrix{ f_{\downarrow 3} = &  \circ & \ccw&\circ&\times& \cw & \cw& \times& \ccw &\circ& \circ &\cw\cdot (t_1^{-3}t_2^2),\\
\tau_1 f_{\downarrow 3} = & \circ & \ccw&\circ&\times& \cw & \cw& \circ& \ccw &\times& \circ &\cw\cdot (t_1^{-4}t_2^3),\\
     \tau_2 f_{\downarrow 3} = & \circ & \ccw&\circ&\circ& \cw & \cw& \times& \ccw &\circ& \times &\cw\cdot (t_1^{-4}t_2^3),\\
     \tau_1\tau_2 f_{\downarrow 3} = & \circ & \ccw&\circ&\circ& \cw & \cw& \circ& \ccw &\times& \times &\cw\cdot (t_1^{-5}t_2^4).}
    \]

\end{ex}

\subsubsection{The main result}

\begin{theorem}\label{thm_compos_series}
    Let $\bl\in \Rep_{Ver_p}(T,\varepsilon)$ be a weight. Then
    \begin{enumerate}
        \item  There exists a typical weight $\bm$,a sequence of weights 
    $$
    \bm = \bm_0, \bm_1, \ldots, \bm_l = \bl,
    $$
    and translation functors $T_1, \ldots, T_l\in \{F_i, E_i~|~0\le i\le p-1\}$, such that
    $$
    P(\bm_i) = T_{i}P(\bm_{i-1}).
    $$
    In particular, $P(\bl) = T_lT_{l-1}\ldots T_2 T_1K(\bm)$ (as $P(\bm)=K(\bm)$, see Theorem \ref{thm_irr_Kac}).

    \item We have
    $$
    [P(\bl):K(\ba)] = \begin{cases}
        1, \text{ if } \ba\in \mathscr P(\bl),\\
        0, \text{ otherwise }.
    \end{cases}
    $$
    In other words,
    $$
    [P(\bl)] = \sum_{\ba\in\cp(\bl)} [K(\ba)]\in  [\cc].
    $$
    \end{enumerate}
\end{theorem}
\begin{proof}
Suppose, as before, that $s_1,\ldots, s_r$ is the list of all the vertices in $f_\bl$ with $\times$'s on them, and suppose $z_1,\ldots, z_r$ are the tails of the corresponding caps in $\caap f_\bl$. We denote these caps  $\xymatrixcolsep{1mm} \xymatrix{._{s_i}\ar `u[r] `[r]!<-12pt,0pt> &._{z_i}}$. Let $\tau_i,~ 1\le i \le r,$ be the corresponding transformation swapping $\times$ at $s_i$ with $\circ$ at $z_i$, so that we have
$$
\cp(\bl)=\{\tau_{i_1}\ldots \tau_{i_l}~|~0\le l\le r, 1\le i_1<i_2<\ldots<i_l\le r\}.
$$

Without loss of generality, let us choose an ordering of $s_i$'s such that the cap $\xymatrixcolsep{1mm} \xymatrix{._{s_i}\ar `u[r] `[r] &._{z_i}}$ is \textbf{inner} modulo all caps $\xymatrixcolsep{1mm} \xymatrix{._{s_j}\ar `u[r] `[r] &{._{z_j}}}$ with $j<i$ (that is, none of the caps starting at $s_j$ with $j>i$ are embedded under it).

We will prove both statements by induction on $\#\bl$. 
\begin{enumerate}
    \item  If $\#\bl = 0$ then we put $\bm = \bl$ and we are done.

    Now, suppose $\#\bl = r>0$. Start with $a=1$.
    
    Put $k=z_a$ and $i = k-s_a\mod p$. For each $k-i< j<k$ (i.e. $j$ in a clockwise open interval from $k-i$ to $k$), let  $y_j\in\{\ccw,\cw\}$ denote the symbol assigned to vertex $j$ in $f_\bl$ (note that $y_j$ cannot be $\times$ or $\circ$, since $\xymatrixcolsep{1mm} \xymatrix{._{s_a}\ar `u[r] `[r]!<-12pt,0pt> &._{z_a}}$ is inner). 

    Applying to $\bl$ transformations $$S_1: (\times~ y_{k-i+1}) \to (y_{k-i+1}~\times),$$ $$S_2:(\times~ y_{k-i+2}) \to (y_{k-i+2}~\times),$$ $$ \ldots $$ $$S_{i-1}:(\times~ y_{k-1})\to (y_{k-1}~ \times)$$ we obtain a weight $\bn$, for which
    $$
    f_\bn = (f_\bn)^{(k-1)}_{\times \circ}.
    $$
    Let $f_{\bl'}=F_{k-1}f_\bn$. Then 
    $$
    f_{\bl'}=(f_\bn)^{(k-1)}_{\ccw\cw} \cdot t_{1}^{-\delta_{k,0}},
    $$
     and $\#\bl' = \#\bn-1 = \#\bl - 1$. We have
    $$
    E_{k-1}P(\bl') = P(\bn)
    $$
    by Theorem \ref{thm_indec_proj}, since $\bn \prec \bn'$, where
    $$
    E_{k-1}f_{\bl'} = f_\bn + f_{\bn'}, 
    $$
    i.e. $f_{\bn'} = t_1^{-\delta_{k,0}}t_2^{\delta_{k,0}} (f_\bn)^{(k-1)}_{\circ \times}$.

    Thus
    $$
    P(\bl) = S_1^*\ldots S_{i-1}^*E_{k-1}P(\bl').
    $$
    
    Note that by construction, the caps in the cap diagram for $\bl'$ are the same as for $\bl$, except without  $\xymatrixcolsep{1mm} \xymatrix{._{s_a}\ar `u[r] `[r]!<-12pt,0pt> &._{z_a}}$, and the cap $\xymatrixcolsep{1mm} \xymatrix{._{s_{a+1}}\ar `u[r] `[r]!<-12pt,0pt> &._{z_{a+1}}}$ is inner.

    Apply the same algorithm to $\bl := \bl'$ starting with $a:=a+1$, until you reach a typical weight $\bm$.

  \item If $\#\bl = 0$ then $P(\bl)= K(\bl)$ and $\mathscr P(\bl) = \{\bl\}$.

  Now suppose $\#\bl = r>0$, Using the notations from the proof of part $(1)$ (we only need the first iteration of the algorithm), assume by induction
  that 
  $$
  [P(\bl')]=\sum_{\ba\in \cp(\bl')} [K(\ba)].
  $$
  Note that by construction of $\bl'$ we have
  $$
  \cp(\bl') = \{\tau_{i_1}\ldots \tau_{i_l}\bl'~|~0\le l\le r-1, 2\le i_1<\ldots<i_l\le r\}.
  $$

  Then for any $\tau_{i_1}\ldots \tau_{i_l}\bl'\in \cp(\bl')$ 
  $$
  E_{k-1}[K(\tau_{i_1}\ldots \tau_{i_l}\bl')] = [K(\tau_{i_1}\ldots \tau_{i_l} \bn)]+[K(\tau_{i_1}\ldots \tau_{i_l} \bn')], 
  $$
  as none of the $\tau_{i_j}$'s affect vertices $k-1$ and $k=z_1$.

  Moreover, $\tau_i$ for $i\ge 2$ does not affect any of the vertices between $k-i=s_1$ and $k=z_1$. Therefore,
  $$
  S_1^*\ldots S_{i-1}^*E_{k-1}[K(\tau_{i_1}\ldots \tau_{i_l}\bl')]=
  $$
  $$
  =[K(\tau_{i_1}\ldots \tau_{i_l}(S_1^*\ldots S_{i-1}^*\bn)]
  +[K(\tau_{i_1}\ldots \tau_{i_l}(S_1^*\ldots S_{i-1}^*\bn')]=
  $$
  $$
  [K(\tau_{i_1}\ldots \tau_{i_l}\bl)] + [K(\tau_{i_1}\ldots \tau_{i_l}(\tau_1\bl))].
  $$
The last equality follows from the fact that applying $S_1^*\ldots S_{i-1}^*$ to $\bn'$ results in a series of transformations:
$$
S_{i-1}^*: (y_{k-i+1}~\krug)\to (\krug~ y_{k-i+1}),
$$
$$
S_{i-2}^*: (y_{k-i+2}~\krug) \to (\krug~ y_{k-i+2}),
$$
$$
\ldots
$$
$$
S_{1}^*: (y_{k-1}~\krug) \to (\krug~ y_{k-1}),
$$
which turns $\bn'$ into $\tau_1\bl$.

Consequently, we get that
$$
[P(\bl)] = S_1^*\ldots S_{i-1}^*E_{k-1} [P(\bl')] =
$$
$$
=S_1^*\ldots S_{i-1}^*E_{k-1}\left(\sum_{l=0}^{r-1} \sum_{2\le i_1<\ldots<i_l\le r-1} [K(\tau_{i_1}\ldots \tau_{i_l} \bl')]\right) =
$$
$$
= \sum_{l=0}^{r-1} \sum_{2\le i_1<\ldots<i_l\le r-1} \left( [K(\tau_{i_1}\ldots \tau_{i_l} \bl)]+[K(\tau_{i_1}\ldots \tau_{i_l} \tau_1\bl)]\right) =
$$
$$
=\sum_{\ba\in \cp(\bl)} [K(\ba)].
$$
  
    \end{enumerate}
\end{proof}

\subsection{Lowest weights}\label{s_lowest}

Theorem \ref{thm_compos_series} gives an explicit description of the highest weight in $P(\bl)$. Namely, Let $\#\bl = r$ and let$\tau_1,\ldots, \tau_r\in \C[t_1^{\pm 1}t_2^{\pm 1}]S_p$ be the swaps of $\times$'s and $\circ$'s in $f_\bl$ along the caps in the corresponding cap diagram (defined right before Definition \ref{def_plambda}). Then
$$
\hat{\bl}:=\tau_1\ldots \tau_r \bl
$$
is the highest weight of $P(\bl)$.

Following Corollary \ref{cor_lowest_weight}, we can now answer Question \ref{quest_1} for $G=GL(L_m|L_n)$.
\begin{theorem}\label{thm_lowest_weight}
    For $G=GL(L_m|L_n)$, $T=GL(L_m)\times GL(L_n)\subset G$ and any weight $\bl\in \Rep_{Ver_p}(T,\varepsilon)$ the lowest weight of the simple module $L(\bl)\in \Rep_{Ver_p}(G,\varepsilon)$ is
    $$
    \hat{\bl} - \bb,
    $$
    where $\bb = (n,\ldots, n|-m,\ldots,-m)$.
\end{theorem}
\begin{remark}
In light of Theorem \ref{thm_odd_refl}, we also get an algorithmic answer for Question \ref{quest_1} in general case $G=GL(X),~X\in Ver_p$.    
\end{remark}

\begin{ex}\label{ex_lowest_w}
    Let $f$ be the weight diagram in Figure \ref{fig:wd}, i.e.
     $$
    f_{\downarrow 3} =  \krug~ \ccw~\krug~\times ~\cw~ \cw~ \times~ \ccw ~\krug~ \krug ~\cw\cdot (t_1^{-3}t_2^2).
    $$
    
    Let $\bl$ be a weight, for which $f=f_\bl$. Recall that 
    $$
    \bl = (18,18,15,12,12|-13,-13,-17,-18)
    $$
    (see Example \ref{ex_computing_weight_from_diag}).

    Then the diagram for ${\hat{\bl}}$ is the following:
     $$
    (f_{\hat{\bl}})_{\downarrow 3} =  \krug~ \ccw~\krug~\krug ~\cw~ \cw~ \krug~ \ccw ~\times~ \times ~\cw\cdot (t_1^{-5}t_2^4).
    $$

    We can compute $\hat{\bl}$ explicitly using Remark \ref{rem_recover_weight}:
    $$
    \hat{\bl} = (19,19,15,15,15|-15,-15,-17,-22).    
    $$

    Thus the lowest weight of $L(\bl)$ is
    $$
    \hat{\bl}-\bb = (15,15,11,11,11|-10,-10,-12,-17).
    $$
\end{ex}

\begin{remark}
    The procedure for computing the lowest weight in $L(\bl)$ turns out to be quite similar to that for supergroup $GL(m|n)$ in characteristic zero, with one principal change: the weight diagrams for $GL(L_m|L_n)$ in $Ver_p$ are \textbf{circular}. It turns out that this change actually affects the answer, and for some weights $\bl$ we might get two different answers: the lowest weight of $L(\bl)\in \Rep_{Ver_p}(GL(L_m|L_n),\varepsilon)$ is not the same as the lowest\footnote{One should actually compute the lowest among highest weights of $GL(m|n)_{ev}=GL_m\times GL_n$-subquotient modules in $L(\l)$.} weight of $L(\l)\in \Rep_\ck (GL(m|n))$ (both when char $\ck = 0$ and when char $\ck = p$) after we identify simple representations of $T$ with integral sequences in $\Z^{m|n}$: $\bl\leftrightarrow \l$.
\end{remark}

\begin{ex}
    Let $\bl$ be as in Example \ref{ex_lowest_w}, and let $\l\in \Z^{m|n}$ be the corresponding integral sequence, i.e.
    $$
    \l=(18,18,15,12,12|-13,-13,-17,18).
    $$
    We can compute the lowest weight in the $GL_m\times GL_n$-equivariant character of $L(\l)$ (i.e. the composition series of $\Res^{GL(m|n)}_{GL_m\times GL_n} L(\l)$) using Serganova's algorithm (see Lemma \ref{l_serg_alg}). It is equal to
    $$
    (15, 15,11, 10, 8| -9, -11,-12,-15).
    $$
    Notice that it is different from  the lowest weight computed in Example \ref{ex_lowest_w}.
\end{ex}

\begin{corollary}
    The dual $L(\bl)^*$ of the irreducible representation $L(\bl)$ of $GL(L_m|L_n)$ is isomorphic to $L(\bb-\hat\bl)$.
\end{corollary}

\bibliographystyle{alphaurl} 
\bibliography{refs}

\end{document}